\documentclass[10pt, a4paper, leqno]{article}
\usepackage[utf8]{inputenc} 
\usepackage[english]{babel} 
\usepackage[usenames, dvipsnames]{xcolor}
\usepackage[a4paper, tmargin=1.0in,bmargin=1.0in, lmargin=1in, rmargin=1in]{geometry} 
\usepackage{bm} 
\usepackage{amsmath} 
\usepackage{amsthm}
\usepackage{amssymb}
\usepackage{cases} 
\usepackage{tikz} 
\usepackage{stmaryrd} 
\usepackage{multirow} 
\usepackage{lipsum} 
\usepackage{wrapfig} 

\usepackage{breakcites} 
\usepackage{mathtools}   
\usepackage{cases}
\usepackage{float}

\usepackage[colorlinks=true,linkcolor=OrangeRed,citecolor=RoyalBlue]{hyperref} 
\usepackage{cleveref} 

\newcommand{\lbm}{lattice Boltzmann}
\newcommand{\strong}[1]{\emph{#1}}
\newcommand{\confer}{\emph{cf.}}
\newcommand{\idEst}{\emph{i.e.}}

\newcommand{\lbmScheme}[2]{$\text{D}_{#1}\text{Q}_{#2}$}
\newcommand{\advectionVelocity}{V}
\newcommand{\solutionCauchyProblem}{u}
\newcommand{\definitionEquality}{:=}
\newcommand{\indexSpace}{j}
\newcommand{\timeVariable}{t}
\newcommand{\spaceVariable}{x}
\newcommand{\initial}{\circ}
\newcommand{\traceDirichlet}{g}
\newcommand{\timeGridPoint}[1]{\timeVariable^{#1}}
\newcommand{\indexTime}{n}
\newcommand{\spaceStep}{\Delta \spaceVariable}
\newcommand{\timeStep}{\Delta \timeVariable}
\newcommand{\relatives}{\mathbb{Z}}
\newcommand{\reals}{\mathbb{R}}
\newcommand{\complex}{\mathbb{C}}
\newcommand{\latticeVelocity}{\lambda}
\newcommand{\naturals}{\mathbb{N}}
\newcommand{\spaceGridPoint}[1]{\spaceVariable_{#1}}
\newcommand{\numberVelocities}{q}
\newcommand{\dimensionlessDiscreteVelocityLetter}{c}
\newcommand{\discrete}[1]{{#1}}
\newcommand{\distributionFunctionLetter}{f}
\newcommand{\momentLetter}{m}
\newcommand{\distributionFunctionDiscrete}{\discrete{\distributionFunctionLetter}}
\newcommand{\momentDiscrete}{\discrete{\momentLetter}}
\newcommand{\linearGroup}[2]{\mathsf{GL}_{#1}(#2)}
\newcommand{\matricial}[1]{\bm{#1}}
\newcommand{\momentMatrix}{\matricial{M}}
\newcommand{\baseField}{\reals}
\newcommand{\transpose}[1]{#1^{\mathsf{T}}}
\newcommand{\vectorial}[1]{\bm{#1}}
\newcommand{\atEquilibrium}{\text{eq}}

\newcommand{\integerInterval}[2]{\llbracket #1, #2 \rrbracket}
\newcommand{\relaxationParameterLetter}{s}
\newcommand{\stencilLeft}{r}
\newcommand{\stencilRight}{p}
\newcommand{\stencilLeftCharacteristic}{\overline{r}}
\newcommand{\stencilRightCharacteristic}{\overline{p}}
\newcommand{\stencilBoundaryCondition}{w}
\newcommand{\indexVelocity}{i}
\newcommand{\collided}{\star}
\newcommand{\diagonalMatrix}{\text{\textbf{diag}}}
\newcommand{\identityMatrix}[1]{\matricial{I}_{#1}}
\newcommand{\fourierShift}{\kappa}
\newcommand{\forwardSpaceShift}{\mathsf{T}}
\newcommand{\indexFreeOne}{\ell}
\newcommand{\indexFreeTwo}{h}
\newcommand{\boundaryCoefficient}{b}

\newcommand{\equilibriumVectorLetter}{\epsilon}
\newcommand{\equilibriumVector}{\vectorial{\equilibriumVectorLetter}}
\newcommand{\orderExtrapolation}{\sigma}
\newcommand{\schemeMatrixBulkByPower}{\matricial{E}}
\newcommand{\schemeMatrixBulkByPowerIncludingShift}{\tilde{\matricial{E}}}
\newcommand{\schemeMatrixBulk}{\matricial{\discrete{E}}}
\newcommand{\schemeMatrixBulkFourier}{\hat{\schemeMatrixBulk}}
\newcommand{\schemeMatrixBoundary}[1]{\matricial{\discrete{B}}_{#1}}
\newcommand{\schemeMatrixBoundaryByPower}[2]{\matricial{\discrete{B}}_{#1, #2}}

\newcommand{\canonicalBasisVector}[1]{\vectorial{e}_{#1}}
\newcommand{\matrixSpace}[2]{\mathcal{M}_{#1}(#2)}
\newcommand{\ring}[2]{#1[#2]}
\newcommand{\determinant}{\textnormal{det}}
\newcommand{\adjugate}{\textbf{\text{adj}}}
\newcommand{\waveNumber}{\vartheta}
\newcommand{\laplaceTransformed}[1]{\tilde{#1}}
\newcommand{\timeShiftOperator}{z}
\DeclareMathOperator*{\degree}{deg}
\newcommand{\companionPolynomial}{\matricial{C}_{\timeShiftOperator}}
\newcommand{\zeroMatrix}[1]{\matricial{0}_{#1}}
\newcommand{\rank}{\text{rank}}
\newcommand{\spectrum}{\textnormal{sp}}

\newcommand{\trace}{\text{tr}}
\newcommand{\coefficientCharEquationInFourier}{d}

\newcommand{\stableSubspace}{\mathbb{E}^{\text{s}}}
\usepackage{mathrsfs}  
\newcommand{\courantNumber}{\mathscr{C}}
\newcommand{\sign}{\textnormal{sgn}}
\newcommand{\targetEigenvalue}{\timeShiftOperator^{\odot}}
\newcommand{\targetFourierShift}{\fourierShift^{\odot}}
\newcommand{\bigO}[1]{O(#1)}
\newcommand{\bigOVectorial}[1]{\vectorial{O}(#1)}
\newcommand{\kernel}{\textnormal{ker}}

\newcommand{\differential}[1]{\text{d}#1}

\newcommand{\boundarySourceTerm}{g}
\newcommand{\boundarySourceTermMoments}{g}
\newcommand{\collisionMatrix}{\matricial{K}}
\newcommand{\cardinality}[1]{\textnormal{card}(#1)}
\newcommand{\matrixPolynomialBulk}[1]{\matricial{L}_{#1}}
\newcommand{\stableRoot}{\fourierShift_{\textnormal{s}}}
\newcommand{\unstableRoot}{\fourierShift_{\textnormal{u}}}
\newcommand{\productRootsDOneQThree}{\Pi}

\newcommand{\coefficientStableSolution}{C_{\textnormal{s}}}
\newcommand{\coefficientZero}{C_{0}}
\newcommand{\eigenvectorLetter}{\varphi}

\newcommand{\residue}[2]{\text{Res}_{#2}\left [#1\right ]}
\newcommand{\groupVelocity}{{V}_{\textnormal{g}}}
\newcommand{\dimension}[1]{\textnormal{dim}(#1)}
\newcommand{\spanSpace}[1]{\textnormal{span}\{ #1 \}}
\newcommand{\finalTime}{T}
\usepackage{dsfont}
\newcommand{\indicatorFunction}[1]{\mathds{1}_{#1}}

\newcommand{\puncturedPlane}{\complex^{*}}
\newcommand{\unitCircle}{\mathbb{S}}
\newcommand{\unitDisk}{\mathbb{D}}
\newcommand{\closedUnitDisk}{\overline{\mathbb{D}}}
\newcommand{\neighborhoodInfinity}{\mathbb{U}}
\newcommand{\closedNeighborhoodInfinity}{\overline{\mathbb{U}}}
\newcommand{\frequency}{\xi}
\newcommand{\kreissLopatinskiiDet}{\Delta_{\textnormal{KL}}}

\newcommand{\kreissLopatinskiiMatrix}{\matricial{B}_{\textnormal{KL}}}
\newcommand{\scalarFactorFromAdjugate}{\Sigma_{\textnormal{bulk}}}
\newcommand{\positiveVelocityIndex}{\pi}
\newcommand{\kreissLopatinskiiScalarProduct}{\langle \textnormal{KL} \rangle}

\newcommand{\threeboxes}[3]{%
  \tikz[baseline]{%
    \def\boxsize{\ht\strutbox} 
    \node (base) at (0, 0.5*\boxsize) {};
    \foreach \i/\letter in {0/#1,1/#2,2/#3} {
      \draw[] (\i*\boxsize, 0) rectangle ++(\boxsize, \boxsize);
      \node at (\i*\boxsize + 0.5*\boxsize, 0.5*\boxsize) {\fontsize{\boxsize}{\boxsize}\selectfont\bfseries \letter};
    }
  }%
}

\newcommand{\notContinuousExtensionMark}{$\times$}
\newcommand{\continuousExtensionMark}{$\square$}
\newcommand{\eigenvalueMark}{$\odot$}
\newcommand{\noEigenvalueMark}{\scalebox{0.75}{$\bigcirc$}}
\newcommand{\zeroKL}{$0$}
\newcommand{\nonZeroKL}{$\star$}
\newcommand{\infiniteKL}{{\tiny $\infty$}}

\theoremstyle{plain}
\newtheorem{proposition}{Proposition}%
\newtheorem{corollary}{Corollary}%
\newtheorem{lemma}{Lemma}%
\newtheorem{theorem}{Theorem}%
\newtheorem*{theorem*}{Theorem}
\theoremstyle{definition}
\newtheorem{definition}{Definition}%
\theoremstyle{remark}
\newtheorem{remark}{Remark}%
\newtheorem{example}{Example}%

\usepackage[dvipsnames]{xcolor}
\usepackage[normalem]{ulem}

\providecommand{\keywords}[1]{\small \textbf{{Keywords: }} {#1}}
\providecommand{\amsCat}[1]{\small \textbf{{MSC: }} #1}

\title{\textsc{Stability of lattice Boltzmann schemes for initial boundary value problems in raw formulation}}

\author{\textsc{Thomas Bellotti}\footnote{Université Paris-Saclay, CNRS, CentraleSupélec, Laboratoire EM2C \& Fédération de Mathématiques de CentraleSupélec, 91190, Gif-sur-Yvette, France}}

\allowdisplaybreaks
\begin{document}


\maketitle



\begin{abstract}
	We study the stability of one-dimensional linear lattice Boltzmann schemes for scalar hyperbolic equations with respect to boundary data.
	Our approach is based on the original raw algorithm on several unknowns, thereby avoiding the need for a transformation into an equivalent scalar formulation---a challenging process in presence of boundaries. 
	To address different behaviors exhibited by the numerical scheme, we introduce appropriate notions of strong stability.
	They account for the potential absence of a continuous extension of the stable vector bundle associated with the bulk scheme on the unit circle for certain components.
	Rather than developing a general theory, complicated by the fact that discrete boundaries in lattice Boltzmann schemes are inherently characteristic, we focus on strong stability--instability for methods whose characteristic equations have stencils of breadth one to the left. In this context, we study three representative schemes. These are endowed with various boundary conditions drawn from the literature, and our theoretical results are supported by numerical simulations.
\end{abstract}


\amsCat{65M12, 76M28, 65M06, 35L50, 35L65}

\section{Introduction}

The major difficulty in the numerical analysis of \lbm{} schemes lies in the presence of \strong{merely numerical unknowns}, which---despite being used in the implementation---do not approximate the solution of the target PDEs.
A way to circumvent this problem when the domain is infinite or periodic boundary conditions are enforced is to algebraically \strong{eliminate} these numerical unknowns from the \strong{raw} scheme, see \cite{fuvcik2021equivalent, bellotti2022finite, bellotti2023truncation}, keeping only the unknowns of interest---approximating the solution of the target PDEs.
This is also tightly linked with the notion of \strong{observability} \cite{bellotti2024initialisation}, which determines, \strong{inter alia}, what ``modes'' on the numerical unknowns impact the unknowns of interest, and those which do not, \strong{e.g.}, constant or checkerboard modes.

In a previous work \cite{bellotti:hal-04630735}, we have followed the leitmotiv consisting in eliminating numerical unknowns from the scheme to study \strong{boundary conditions for \lbm{} schemes}.
The significant drawback of this approach is its lack of generality. 
Specifically, the elimination does not rely on a general algebraic result---the Cayley-Hamilton theorem---as with infinite or periodic domains. 
Therefore, this study only applies to a simple two-unknowns scheme supplemented with specific boundary conditions. 
It is unlikely that this approach could be pursued to analyze more involved situations. 
Even in the context of this simple two-unknowns scheme, we are not sure that the elimination could be conducted regardless of the boundary conditions at hand.

The present contribution aims at investigating the stability of \lbm{} schemes for initial boundary value problems, adapting the well-known GKS (for Gustafsson, Kreiss, and Sundstr{\"o}m) stability analysis \cite{gustafsson1972stability} for Finite Difference methods, also known as ``\strong{strong stability}''---\strong{without} having to eliminate numerical unknowns.
This procedure on the original raw scheme without transformation is quite attractive, for it can be applied in principle to any \lbm{} scheme.
Needless to say, computations become more and more involved as the complexity of the scheme grows, making the exploration of certified numerical tools to check stability \cite{boutin2024stability, boutin2023stability} very appealing, although not the object of the present paper.

In order to assess stability for \lbm{} schemes in presence of boundary conditions, we see them as one-step algorithms on a vector of unknowns, which do not all fulfill a PDE at leading order. 
One has to think at this rather in terms of a PDAE (Partial Differential Algebraic Equation) problem.
Along the way, two tightly linked difficulties must be taken into account, and prevent---at the current stage---from developing a very general theory as in \cite{coulombel2009stability}.
\begin{enumerate}
    \item The so-called \strong{non-characteristic discrete boundary} assumption---customary in the study of Finite Difference schemes, \strong{e.g.} \cite[Assumption 5.5]{gustafsson1972stability}, \cite[Assumption 1]{coulombel2009stability}, \cite[Assumption 2.1]{coulombel2011semigroup}, and \cite[Assumption 1]{coulombel2011stability}---is not fulfilled by \lbm{} methods.
    The discrete boundary is non-characteristic when the leftmost and rightmost matrices of the numerical scheme are non-singular.
    Since \lbm{} schemes are inherently linked to relaxation systems \cite{graille2014approximation, caetano2024result}, which are often characteristic \cite{boutin2020stiffly}, this feature is not surprising.
    This prevents us---\strong{inter alia}---from directly using a classical result \cite[Lemma 5.2]{gustafsson1972stability}, sometimes called ``Hersh lemma'' \cite{boutin2024stability}, that classifies the roots of the scheme's characteristic equation with respect to the unit circle. Finding the roots inside the unit disk is one of the two fundamental step when checking strong stability for a given boundary condition.
	Roughly speaking, the Hersh lemma for the lattice Boltzmann schemes that we consider looks like the one of a Finite Difference scheme tackling a scalar problem rather than a system of equations.
	
	More importantly, dealing with characteristic boundaries prevents from recasting the scheme, using a block companion matrix, under the form of an explicit recurrence relation.
	This has a serious impact on the \strong{eigenvectors} (more precisely, eigenspaces) of the problem. 
	These eigenspaces may lack (at least in some components) a continuous extension on the unit circle.

    \item Boundary conditions can be stable on the unknown of interest but generate instabilities on numerical unknowns. 
    This can be explained by the fact that the unstable modes developing on the numerical unknowns belong to the \strong{unobservable subspace} \cite{bellotti2024initialisation} relative to the unknown of interest, thus the latter is not influenced by the former.
    This calls for a finer understanding of the way instabilities develop, for simplistically claim that the scheme is (overall) unstable does not do justice to the fact that is remains stable on the unknown we are interested in. 
	This finer understanding comes from the study of the eigenvectors of the problem, which indicate how the instabilities distribute between the unknowns of the system, and whose structure is now more involved than in the non-characteristic setting.
	
	More extremely, there are cases where there exist eigenvalues for the coupled bulk--boundary system and yet, boundary conditions are strongly stable for all unknowns.
	This comes, once again, from the fact that the structure of the eigenvectors must be taken into consideration.
	Specifically, the eigenvectors of the bulk matrix are not necessarily eigenvectors for the boundary matrix, even when sharing an eigenvalue.
	Even in strongly stable cases, this does not conceal the usefulness of knowing the eigenvalues for the coupled bulk--boundary system, since they make up the components of the discrete solution featuring the larger amplitude, and are easily discernible in numerical simulations. 
\end{enumerate}

Although the reader will immediately notice that we aim at adopting---while analyzing lattice Boltzmann schemes---a standpoint as close as possible to the GKS analysis of Finite Differences, let us review existing works concerning stability of lattice Boltzmann schemes in presence of boundaries with no connection with the GKS-theory.
An undeniable virtue of these works lies in their ability to deal with multi-dimensional problems. 
A major approach relies on the so-called ``\strong{stability structure}'' \cite{junk2009weighted, junk2009convergence}, based on the construction of a weighted norm with respect to which both phases of the lattice Boltzmann algorithm, namely relaxation and  transport, are contractions. 
This notion of stability is intrinsically linear (although having been adapted to quadratic non-linearities to yield convergence with bounce-back boundary conditions \cite{junk2009convergence} towards the solution of the incompressible Navier-Stokes system under smoothness assumptions) and does not consider relaxation and transport jointly, thus producing only sufficient stability conditions.
This approach, which is able to handle bounce-back and anti-bounce-back boundary conditions for parabolic problems, is unfit do to so for the hyperbolic equation we are interested in, see the conclusions in \cite{rheinlander2010stability}. Conversely, the approach that we develop in the present contribution encompasses but is not limited to the two previously-mentioned boundary conditions and is specifically aimed at hyperbolic problems.
Finally, let us mention that a more recent way of analyzing stability for boundary conditions built using equilibria, with no link with the GKS-theory, has been developed relying on \strong{monotonicity} \cite{aregba2025equilibrium}, and applies to weak solutions of scalar non-linear equations.

The paper is organized as follows.
In \Cref{sec:problemAndNumerical}, we introduce the continuous problem and the numerical schemes to address it.
In \Cref{sec:analysis}, we clarify different notions of stability for the scheme without boundaries, and introduce the notion of strong stability when boundaries are present.
In a general setting, we also pave the way to \Cref{sec:strongCertainSchemes}, where strong stability is studied for three representative schemes with various boundary conditions.
In this section, numerical simulations to illustrate and support theoretical results are also provided.
\Cref{sec:conclusions} draws general conclusions and illustrates research perspectives. 


\section{Problem and numerical schemes}\label{sec:problemAndNumerical}

\subsection{Continuous problem}

The problem we are concerned with is the \strong{one-dimensional transport equation} at velocity $\advectionVelocity\neq 0$:
\begin{equation}\label{eq:targetEquation}
    \left\lbrace\,
    \begin{array}{@{}l@{\quad}l@{}l@{}}
        \partial_{\timeVariable}\solutionCauchyProblem(\timeVariable, \spaceVariable) + \advectionVelocity \partial_{\spaceVariable}\solutionCauchyProblem(\timeVariable, \spaceVariable) = 0, \qquad &\timeVariable> 0, \quad &\spaceVariable> 0, \\
        \solutionCauchyProblem(0, \spaceVariable) = \solutionCauchyProblem^{\initial}(\spaceVariable), \qquad & &\spaceVariable> 0,\\
        \solutionCauchyProblem(\timeVariable, 0) = \traceDirichlet(\timeVariable), \qquad &\timeVariable> 0,  &\qquad  \qquad \qquad (\text{if}\quad \advectionVelocity>0).
    \end{array}
    \right.
\end{equation}

\subsection{Time-space discretization}

The  spatial domain is discretized with grid-points $\spaceGridPoint{\indexSpace} \definitionEquality \indexSpace \spaceStep$, with $\indexSpace \in \naturals$ and $\spaceStep > 0$.
We include the boundary point $\spaceVariable = 0$ in the computational domain using $\spaceGridPoint{0} = 0$.
Moreover, the definition of $\spaceGridPoint{\indexSpace}$ works for $\indexSpace \in \relatives$, and thus we can define ghost points.

Time discretization employs grid-points $\timeGridPoint{\indexTime} \definitionEquality \indexTime\timeStep$ with $\indexTime \in \naturals$.
The time step $\timeStep$ is linked to the space step $\spaceStep$ by $\timeStep \definitionEquality \spaceStep/\latticeVelocity$ with $\latticeVelocity > 0$ called ``lattice velocity''.
This scaling between space and time discretization, known as ``acoustic scaling'', is relevant in this context where information travels at \strong{finite speed} and \strong{time-explicit} numerical methods are used.

As this is frequently used in what follows, we introduce the Courant number $\courantNumber\definitionEquality\advectionVelocity/\latticeVelocity$, which quantifies the strength of the transport velocity $\advectionVelocity$ with respect to the velocity of propagation of information by the numerical scheme.
In the paper, we always consider $\courantNumber\neq 0$, both to ensure the physical significance of \eqref{eq:targetEquation}, and to avoid \strong{glancing} wave packets, \confer{} \cite{coulombel2015fully}.

\subsection{Numerical schemes}\label{sec:numericalSchemes}

We consider a quite general class of lattice Boltzmann schemes, the so-called ``multiple relaxation-times'' (MRT) schemes.
We see these schemes merely as given algorithms with a precise structure which involves a certain number of pieces detailed below.
Clues on how to select such pieces according to the consistency with \eqref{eq:targetEquation} are eventually given.
We consider:
\begin{itemize}
	\item A set of $\numberVelocities = 2, 3, \dots$ dimensionless discrete velocities $\dimensionlessDiscreteVelocityLetter_1, \dots, \dimensionlessDiscreteVelocityLetter_{\numberVelocities} \in \relatives$, with associated discrete distribution functions $\distributionFunctionDiscrete_1, \dots, \distributionFunctionDiscrete_{\numberVelocities}$.
	\item A moment matrix $\momentMatrix \in \linearGroup{\numberVelocities}{\reals}$.
	This non-singular matrix yields a change of basis between discrete distribution functions and the so-called moments $\momentDiscrete_1, \dots, \momentDiscrete_{\numberVelocities}$, given by $\transpose{(\momentDiscrete_1, \dots, \momentDiscrete_{\numberVelocities})} = \momentMatrix \transpose{(\distributionFunctionDiscrete_1, \dots, \distributionFunctionDiscrete_{\numberVelocities})}$.
	\item The equilibria $\vectorial{\momentLetter}^{\atEquilibrium} : \baseField \to \baseField^{\numberVelocities}$, fulfilling the conservation constraint $\momentLetter^{\atEquilibrium}_{1}(\momentDiscrete_1) = \momentDiscrete_1$. We therefore assume that $\momentDiscrete_1$ approximates $\solutionCauchyProblem$, thus is the unknown we are interested in.
	In all \lbm{} we are aware of, the moment $\momentDiscrete_1$ is the zeroth-order moment of the distribution functions, hence $\momentDiscrete_1 = \sum_{\indexVelocity=1}^{\numberVelocities}\distributionFunctionDiscrete_{\indexVelocity}$.
	\item The relaxation parameters $\relaxationParameterLetter_1, \dots, \relaxationParameterLetter_{\numberVelocities}$, where $\relaxationParameterLetter_1 \in \baseField$ can be chosen freely, whereas $\relaxationParameterLetter_{2}, \dots, \relaxationParameterLetter_{\numberVelocities} \in [0, 2]$.
	This last constraint is justified in \Cref{sec:stabilityCauchy} for stability reasons, \confer{} \Cref{rem:reasonSChoice}.
	Other explanations of the same constraint can be found, for example, in \cite{dubois2008equivalent, dubois2022nonlinear}.
\end{itemize}
We introduce $\stencilLeft \geq 0$ (respectively $\stencilRight \geq 0$), the amplitude of the space-stencil to the left (respectively, to the right) generated by the discrete velocities, given by 
\begin{equation*}
	\stencilLeft \definitionEquality \max_{\indexVelocity \in \integerInterval{1}{\numberVelocities}} \dimensionlessDiscreteVelocityLetter_{\indexVelocity} \qquad\text{and}\qquad \stencilRight \definitionEquality - \min_{\indexVelocity \in \integerInterval{1}{\numberVelocities}} \dimensionlessDiscreteVelocityLetter_{\indexVelocity}.
\end{equation*}

\subsubsection{Bulk scheme}

With these pieces being given, the \lbm{} algorithm is made up of two phases.
\begin{enumerate}
	\item A \strong{relaxation} phase: a linear relaxation on the moments given by 
	\begin{equation}\label{eq:relaxation}
		\vectorial{\momentDiscrete}_{\indexSpace}^{\indexTime\collided} \definitionEquality (\identityMatrix{\numberVelocities} - \diagonalMatrix(\relaxationParameterLetter_1, \dots, \relaxationParameterLetter_{\numberVelocities})) \vectorial{\momentDiscrete}_{\indexSpace}^{\indexTime} + \diagonalMatrix(\relaxationParameterLetter_1, \dots, \relaxationParameterLetter_{\numberVelocities}) \vectorial{\momentLetter}^{\atEquilibrium} (\momentDiscrete_{1, \indexSpace}^{\indexTime}), \qquad \indexSpace \in \naturals.
	\end{equation}
	Notice that, thanks to the fact that the relaxation is local, every point of the discrete lattice performs this operation.
	\item A \strong{transport} phase, which shifts the distribution functions in the upwind direction. 
	This reads, for $\indexVelocity \in \integerInterval{1}{\numberVelocities}$
	\begin{equation}\label{eq:transportBulk} 
		\distributionFunctionDiscrete_{\indexVelocity, \indexSpace}^{\indexTime + 1} = \distributionFunctionDiscrete_{\indexVelocity, \indexSpace - \dimensionlessDiscreteVelocityLetter_{\indexVelocity}}^{\indexTime \collided}, \qquad \indexSpace \geq \max(0, \dimensionlessDiscreteVelocityLetter_{\indexVelocity}).
	\end{equation}
	Notice that at the boundary points indexed by $\indexSpace \in\integerInterval{0}{\stencilLeft-1}$, some distribution functions at time $\timeGridPoint{\indexTime + 1}$ are not defined using \eqref{eq:transportBulk}, because incoming information comes from outside the computational domain.
	Fixing this lack of information is the aim of \strong{boundary conditions}.
\end{enumerate}

As far as \strong{consistency} with the target equation \eqref{eq:targetEquation} is concerned, one can see from \cite{bellotti2023truncation} that a sufficient condition is
\begin{equation}\label{eq:consistencyEquation}
    \latticeVelocity \transpose{\canonicalBasisVector{1}}\momentMatrix \diagonalMatrix(\dimensionlessDiscreteVelocityLetter_1, \dots, \dimensionlessDiscreteVelocityLetter_{\numberVelocities}) \momentMatrix^{-1} \vectorial{\momentLetter}^{\atEquilibrium}(\momentDiscrete_1) = \advectionVelocity \momentDiscrete_1.
\end{equation}
Since the target problem is linear, and strong stability for Finite Difference schemes handles only linear ones, we assume that equilibria are linear functions of the conserved moment. Therefore, there exists $\equilibriumVector \in \baseField^{\numberVelocities}$ such that $\equilibriumVectorLetter_1 = 1$, and $\vectorial{\momentLetter}^{\atEquilibrium}(\momentDiscrete_1) = \equilibriumVector \momentDiscrete_1$.
Then, \eqref{eq:consistencyEquation} becomes $\transpose{\canonicalBasisVector{1}}\momentMatrix \diagonalMatrix(\dimensionlessDiscreteVelocityLetter_1, \dots, \dimensionlessDiscreteVelocityLetter_{\numberVelocities}) \momentMatrix^{-1} \equilibriumVector = \courantNumber$, where $\courantNumber$ is the Courant number.
\begin{remark}[On $\relaxationParameterLetter_2, \dots, \relaxationParameterLetter_{\numberVelocities} = 0$]
	In what follows, we do not allow any of the relaxation parameters $\relaxationParameterLetter_2, \dots, \relaxationParameterLetter_{\numberVelocities}$ be zero.
	This is to prevent possible lacks of consistency with the target equation \eqref{eq:targetEquation}, because despite \eqref{eq:consistencyEquation} being true, the relaxed moments may indeed not depend on $\advectionVelocity$, see \eqref{eq:relaxation}.
	This does not prevent the scheme from being stable when all or some relaxation parameters are zero.
\end{remark}

\subsubsection{Boundary schemes: kinetic boundary conditions}

\newcommand{\squareTikzTwoColors}[2]{%
	\fill[color=RoyalBlue] (#1-0.2,#2-0.2) -- (#1-0.2,#2+0.2) -- (#1,#2+0.2) -- (#1,#2-0.2) -- cycle;
	\fill[color=OrangeRed] (#1,#2-0.2) -- (#1,#2+0.2) -- (#1+0.2,#2+0.2) -- (#1+0.2,#2-0.2) -- cycle;
  	\draw[black] (#1-0.2,#2-0.2) rectangle (#1+0.2,#2+0.2);
}

\newcommand{\squareTikzLeftColors}[2]{%
	\fill[color=RoyalBlue] (#1-0.2,#2-0.2) -- (#1-0.2,#2+0.2) -- (#1,#2+0.2) -- (#1,#2-0.2) -- cycle;
  	\draw[black] (#1-0.2,#2-0.2) rectangle (#1+0.2,#2+0.2);
}

\newcommand{\doubleCurlyArrows}[2]{%
	\draw[thick, ->, color=RoyalBlue] plot [smooth, tension=0.7] coordinates { (#1-0.1,#2+0.3) (#1-0.1, #2+0.6) (#1+0.1, #2+1) (#1-0.1, #2+1.4) (#1-0.1, #2+1.7) };
	\draw[thick, ->, color=OrangeRed] plot [smooth, tension=0.7] coordinates { (#1+0.1,#2+0.3) (#1+0.1, #2+0.6) (#1-0.1, #2+1) (#1+0.1, #2+1.4) (#1+0.1, #2+1.7) };
}

\usetikzlibrary{decorations.pathreplacing} 

\begin{figure}[h]
	\begin{center}
		\begin{tikzpicture}

			\draw[->] (-4,0) -- (5,0) node[right] {\(x\)};
			\draw[->] (-3.8,-0.2) -- (-3.8,6) node[above] {\(t\)};

			\draw(-2*1.4,0) circle (2pt);

			\foreach \x in {-1,...,2}
				\fill[color=black](\x*1.4,0) circle (2pt);

			\node at (-2*1.4, -0.5) {$\spaceGridPoint{-1}$};
			\node at (-1*1.4, -0.5) {$\spaceGridPoint{0} = 0$};
			\node at (0*1.4, -0.5) {$\spaceGridPoint{1}$};
			\node at (1*1.4, -0.5) {$\spaceGridPoint{2}$};
            \node at (2*1.4, -0.5) {$\spaceGridPoint{3}$};
			\node at (3*1.4, -0.5) {$\cdots$};

			\draw (-3.9, 1) node[left] {$\timeGridPoint{\indexTime}$} -- (-3.7, 1);
			\draw (-3.9, 3) node[left] {($\timeGridPoint{\indexTime\collided}$)} -- (-3.7, 3);
			\draw (-3.9, 5) node[left] {$\timeGridPoint{\indexTime + 1}$} -- (-3.7, 5);

			\squareTikzTwoColors{-1.4}{1}
			\doubleCurlyArrows{-1.4}{1}
			
			\squareTikzTwoColors{0.}{1}
			\doubleCurlyArrows{0.}{1}

			\squareTikzTwoColors{1.4}{1}
			\doubleCurlyArrows{1.4}{1}

			\squareTikzTwoColors{2.8}{1}
			\doubleCurlyArrows{2.8}{1}
			\node at (3*1.4, 1) {$\cdots$};

			\squareTikzLeftColors{-2.8}{3}
			\squareTikzTwoColors{-1.4}{3}
			\squareTikzTwoColors{0.}{3}
			\squareTikzTwoColors{1.4}{3}
			\squareTikzTwoColors{2.8}{3}
			\node at (3*1.4, 3) {$\cdots$};
			\foreach \x in {-2,...,1}
				\draw[thick, ->, color=RoyalBlue] (\x*1.4-0.1, 3.3) -- (\x*1.4 + 1.3, 4.7);

			\foreach \x in {-1,...,1}
				\draw[thick, ->, color=OrangeRed] (\x*1.4+1.5, 3.3) -- (\x*1.4 + 0.1, 4.7);

			\squareTikzTwoColors{-1.4}{5}
			\squareTikzTwoColors{0.}{5}
			\squareTikzTwoColors{1.4}{5}
			\squareTikzTwoColors{2.8}{5}
			\node at (3*1.4, 5) {$\cdots$};

			\draw[thick, ->, dash pattern=on 2pt off 2pt, color=black!80] (-1.5, 2.75) to[out=-120,in=-90] node[midway, above] {\tiny $\boundaryCoefficient_{\textcolor{RoyalBlue}{1}, \textcolor{RoyalBlue}{1}, 1, 0}$} (-2.9, 2.75) ;
			\draw[thick, ->, dash pattern=on 2pt off 2pt, color=black!80]  (-0.1, 2.75) to[out=-120,in=-90] node[pos=0.75, below] {\tiny $\boundaryCoefficient_{\textcolor{RoyalBlue}{1}, \textcolor{RoyalBlue}{1}, 1, 1}$} (-2.9, 2.75) ;

			\draw[thick, ->, dash pattern=on 2pt off 2pt, color=black!80]  (-1.3, 3.25) to[out=90,in=60] node[midway, below] {\tiny $\boundaryCoefficient_{\textcolor{RoyalBlue}{1}, \textcolor{OrangeRed}{2}, 1, 0}$} (-2.9, 3.25) ;
			\draw[thick, ->, dash pattern=on 2pt off 2pt, color=black!80]  (0.1, 3.25) to[out=90,in=60] node[pos=0.57, below] {\tiny $\boundaryCoefficient_{\textcolor{RoyalBlue}{1}, \textcolor{OrangeRed}{2}, 1, 1}$} (-2.9, 3.25) ;

			\draw[decorate, decoration={brace, mirror}, thick] (4*1.4-0.5,1.1) -- (4*1.4-0.5,2.9) node[midway, right=2pt] {Relaxation \eqref{eq:relaxation}};
			\draw[decorate, decoration={brace, mirror}, thick] (4*1.4-0.5,3.1) -- (4*1.4-0.5,4.9) node[midway, right=2pt] {Transport \eqref{eq:transportBulk}--\eqref{eq:boundaryConditions}};

		\end{tikzpicture}
	\end{center}\caption{\label{fig:descriptionScheme}Sketch explaining the way schemes work, with $\numberVelocities = 2$ discrete velocities, $\dimensionlessDiscreteVelocityLetter_1 = 1$ (with associated distribution function indicated in blue), and $\dimensionlessDiscreteVelocityLetter_2 = -1$ (with distribution function indicated in red).}
\end{figure}

We analyze \strong{kinetic} linear boundary conditions, namely we adopt the same transport rule as in the bulk of the domain, \confer{} \eqref{eq:transportBulk}.
This boils down to preparing data at the ghost points $\spaceGridPoint{-\stencilLeft}, \dots, \spaceGridPoint{-1}$ using linear combinations of post-relaxation data, and is similar to what is done in \cite{boutin2024stability} for upwind schemes. However, ghost points are not part of the computational domain, for here, knowledge of the solution is only partial, and thus prevents from defining the relaxation phase on these grid-points.
The one we make is not the only possible choice, but it adheres the most to the \lbm{} leitmotiv encompassing traditional boundary conditions (\strong{e.g.}, bounce-back \cite{dubois2015taylor} and anti-bounce-back \cite{dubois2020anti}).
We therefore write, for $\indexVelocity \in \integerInterval{1}{\numberVelocities}$
\begin{equation}\label{eq:boundaryConditions}
	\distributionFunctionDiscrete_{\indexVelocity, \indexSpace}^{\indexTime + 1} = \distributionFunctionDiscrete_{\indexVelocity, \indexSpace - \dimensionlessDiscreteVelocityLetter_{\indexVelocity}}^{\indexTime \collided}, \qquad \indexSpace\in \naturals, \qquad \text{with} \qquad \distributionFunctionDiscrete_{\indexVelocity, -\indexSpace}^{\indexTime \collided} \definitionEquality \sum_{\indexFreeOne = 1}^{\numberVelocities} \sum_{\indexFreeTwo\in\naturals} \boundaryCoefficient_{\indexVelocity, \indexFreeOne, \indexSpace, \indexFreeTwo} \distributionFunctionDiscrete_{\indexFreeOne, \indexFreeTwo}^{\indexTime \collided}  + \boundarySourceTerm_{\indexVelocity, -\indexSpace}^{\indexTime}, \qquad \indexSpace \in \integerInterval{1}{\max(0, \dimensionlessDiscreteVelocityLetter_{\indexVelocity})}.
\end{equation}
The weights $\boundaryCoefficient_{\indexVelocity, \indexFreeOne, \indexSpace, \indexFreeTwo} \in\reals$ are compactly supported with respect to the last index.
A weight $\boundaryCoefficient_{\indexVelocity, \indexFreeOne, \indexSpace, \indexFreeTwo}$ affects the $\indexVelocity$-th distribution function, to be stored at the ghost point $\spaceGridPoint{-\indexSpace}$, using the $\indexFreeOne$-th distribution function located at the inner point $\spaceGridPoint{\indexFreeTwo}$.
In what follows, weights with indices outside significant ranges are understood to be zero.
A sketch of the situation is provided in \Cref{fig:descriptionScheme}.
We can define the (maximal) stencil of the boundary condition to the right inside the computational domain $\stencilBoundaryCondition \definitionEquality \max_{\indexVelocity, \indexFreeOne, \indexSpace, \indexFreeTwo } \{\indexFreeTwo ~:~\boundaryCoefficient_{\indexVelocity, \indexFreeOne, \indexSpace, \indexFreeTwo}\neq 0 \}$.
Finally, we have introduced given boundary source terms/data $\boundarySourceTerm_{\indexVelocity, -\indexSpace}^{\indexTime} \in \reals$.
They can be linked to the boundary datum $\traceDirichlet$ in \eqref{eq:targetEquation} when $\advectionVelocity>0$, even if this is pointless in the present work.

\subsubsection{Initial conditions}
Despite focusing here on schemes with zero initial data $\vectorial{\distributionFunctionDiscrete}^0_{\indexSpace}\equiv \zeroMatrix{\numberVelocities}$ (hence $\vectorial{\momentDiscrete}^0_{\indexSpace}\equiv \zeroMatrix{\numberVelocities}$), one must keep in mind that since $\numberVelocities>1$, there is an infinite number of ways to link $\vectorial{\distributionFunctionDiscrete}^0_{\indexSpace} \in\reals^{\numberVelocities}$ (or $\vectorial{\momentDiscrete}^0_{\indexSpace}\in\reals^{\numberVelocities}$) to $\solutionCauchyProblem^{\initial}(\spaceGridPoint{\indexSpace})\in\reals$, when $\solutionCauchyProblem^{\initial}$ is smooth enough to be evaluated pointwise.

\section{Analysis}\label{sec:analysis}

For future use, we introduce the following notations:
\begin{align*}
	\puncturedPlane\definitionEquality\complex\smallsetminus\{0\}, \qquad \unitCircle&\definitionEquality\{\timeShiftOperator\in\complex\quad\text{s.t.}\quad |\timeShiftOperator|=1\}, \\
	\unitDisk &\definitionEquality \{\timeShiftOperator\in\complex\quad\text{s.t.}\quad |\timeShiftOperator|<1\}, \qquad \closedUnitDisk\definitionEquality \unitDisk\cup\unitCircle, \\
	\neighborhoodInfinity &\definitionEquality \{\timeShiftOperator\in\complex\quad\text{s.t.}\quad |\timeShiftOperator|>1\}, \qquad \closedNeighborhoodInfinity\definitionEquality \neighborhoodInfinity\cup\unitCircle.
\end{align*}

We consider the scheme with \strong{zero initial data} written on the moments, which reads, for $\indexTime\in\naturals$:
\begin{numcases}{}
    \vectorial{\momentDiscrete}_{\indexSpace}^{\indexTime + 1} = \schemeMatrixBulk \vectorial{\momentDiscrete}_{\indexSpace}^{\indexTime}, \qquad &$\indexSpace\geq \stencilLeft,$ \label{eq:bulkScheme}\\
    \vectorial{\momentDiscrete}_{\indexSpace}^{\indexTime + 1} = \schemeMatrixBoundary{\indexSpace} \vectorial{\momentDiscrete}_{\indexSpace}^{\indexTime} + \vectorial{\boundarySourceTermMoments}_{\indexSpace}^{\indexTime}, \qquad &$\indexSpace \in \integerInterval{0}{\stencilLeft - 1}, $\label{eq:boundaryScheme}\\
    \vectorial{\momentDiscrete}_{\indexSpace}^{0} = \vectorial{0}, \qquad &$\indexSpace\in\naturals.$\label{eq:zeroInitialData}
\end{numcases}
Here, the relaxation matrix on the moments reads $\collisionMatrix\definitionEquality\identityMatrix{\numberVelocities}  + \diagonalMatrix(\relaxationParameterLetter_1, \dots, \relaxationParameterLetter_{\numberVelocities}) (\equilibriumVector\transpose{\canonicalBasisVector{1}} - \identityMatrix{\numberVelocities})$, and 
\begin{align}	
	\schemeMatrixBulk &\definitionEquality \momentMatrix \Bigl ( \sum_{\indexVelocity = 1}^{\numberVelocities} \canonicalBasisVector{\indexVelocity}\transpose{\canonicalBasisVector{\indexVelocity}} \forwardSpaceShift^{-\dimensionlessDiscreteVelocityLetter_{\indexVelocity}} \Bigr )\momentMatrix^{-1} \collisionMatrix \in \matrixSpace{\numberVelocities}{\ring{\baseField}{\forwardSpaceShift^{\pm}}},\\
    \schemeMatrixBoundary{\indexSpace} &\definitionEquality \momentMatrix 
	\Bigl ( 
	\underbrace{\sum_{\dimensionlessDiscreteVelocityLetter_{\indexVelocity}\leq \indexSpace} \canonicalBasisVector{\indexVelocity}\transpose{\canonicalBasisVector{\indexVelocity}} \forwardSpaceShift^{-\dimensionlessDiscreteVelocityLetter_{\indexVelocity}}}_{\text{data inside domain}} + 
	\underbrace{\sum_{\dimensionlessDiscreteVelocityLetter_{\indexVelocity}> \indexSpace} \sum_{\indexFreeOne = 1}^{\numberVelocities}  \canonicalBasisVector{\indexVelocity} \transpose{\canonicalBasisVector{\indexFreeOne}} \sum_{\indexFreeTwo = 0}^{\stencilBoundaryCondition}\boundaryCoefficient_{\indexVelocity, \indexFreeOne, \dimensionlessDiscreteVelocityLetter_{\indexVelocity}-\indexSpace, \indexFreeTwo}\forwardSpaceShift^{-\indexSpace+\indexFreeTwo}}_{\text{ghosts filled w. data inside the domain}}
	\Bigr )
	\momentMatrix^{-1} \collisionMatrix \in \matrixSpace{\numberVelocities}{\ring{\baseField}{\forwardSpaceShift^{\pm}}},\\
    \vectorial{\boundarySourceTermMoments}_{\indexSpace}^{\indexTime}&\definitionEquality\momentMatrix  \sum_{\dimensionlessDiscreteVelocityLetter_{\indexVelocity}> \indexSpace} \canonicalBasisVector{\indexVelocity}\boundarySourceTerm_{\indexVelocity, \indexSpace - \dimensionlessDiscreteVelocityLetter_{\indexVelocity}}^{\indexTime}\in\reals^{\numberVelocities},\label{eq:fromKineticToMomentSources}
\end{align}
with $\indexSpace\in\integerInterval{0}{\stencilLeft - 1}$.
We have introduced the forward space shift operator $\forwardSpaceShift$ such that $\forwardSpaceShift \momentDiscrete_{\indexSpace} = \momentDiscrete_{\indexSpace + 1}$.
Furthermore, $\ring{\baseField}{\forwardSpaceShift^{\pm}}$ is the ring of Laurent polynomial on the base field $\baseField$ in the indeterminates $\forwardSpaceShift$ and $\forwardSpaceShift^{-1}$.
The scheme can also be written on the distribution functions as $\vectorial{\distributionFunctionDiscrete}_{\indexSpace}^{\indexTime + 1} = \momentMatrix^{-1} \schemeMatrixBulk \momentMatrix \vectorial{\distributionFunctionDiscrete}_{\indexSpace}^{\indexTime}$ but, since the conserved moment $\momentDiscrete_1$ is  physically significant as an approximation of $\solutionCauchyProblem$, we adopt the moment point-of-view.

We write $\schemeMatrixBulk$ and $\schemeMatrixBoundary{\indexSpace}$ rather as elements of $\ring{(\matrixSpace{\numberVelocities}{\reals})}{\forwardSpaceShift^{\pm}}$ than of $\matrixSpace{\numberVelocities}{\ring{\baseField}{\forwardSpaceShift^{\pm}}}$:
\begin{align}
	\schemeMatrixBulk &= \sum_{\indexSpace = -\stencilLeft}^{\stencilRight} \schemeMatrixBulkByPower_{\indexSpace}\forwardSpaceShift^{\indexSpace}, \qquad \text{where} \qquad \schemeMatrixBulkByPower_{\indexSpace}  =\sum_{\substack{\indexVelocity \in \integerInterval{1}{\numberVelocities} \\ \text{s.t.}~\dimensionlessDiscreteVelocityLetter_{\indexVelocity} = -\indexSpace}} \momentMatrix  \canonicalBasisVector{\indexVelocity}\transpose{\canonicalBasisVector{\indexVelocity}}\momentMatrix^{-1} \collisionMatrix \in\matrixSpace{\numberVelocities}{\baseField},\label{eq:matrixByPowerBulk}\\
	\schemeMatrixBoundary{\indexSpace} &= \sum_{\indexFreeOne = -\indexSpace}^{\stencilBoundaryCondition}\schemeMatrixBoundaryByPower{\indexSpace}{\indexFreeOne}\forwardSpaceShift^{\indexFreeOne}, \qquad \text{where}\qquad \schemeMatrixBoundaryByPower{\indexSpace}{\indexFreeOne}\in\matrixSpace{\numberVelocities}{\baseField}.\label{eq:matrixByPowerBoundary}
\end{align} 
Equation \eqref{eq:matrixByPowerBulk} explicitly shows that in general $\schemeMatrixBulkByPower_{\indexSpace} \not\in \linearGroup{\numberVelocities}{\baseField}$.
More precisely, if the discrete velocities $\dimensionlessDiscreteVelocityLetter_1, \dots, \dimensionlessDiscreteVelocityLetter_{\numberVelocities}$ are all distinct,  $\rank(\schemeMatrixBulkByPower_{\indexSpace}) \leq 1 < \numberVelocities$ for all $\indexSpace\in\integerInterval{-\stencilLeft}{\stencilRight}$ and  means that the \strong{non-characteristic assumption is not fulfilled}.

We introduce the notation $\schemeMatrixBulkFourier(\fourierShift) \definitionEquality \sum_{\indexFreeOne = -\stencilLeft}^{\indexFreeOne = \stencilRight} \schemeMatrixBulkByPower_{\indexFreeOne}\fourierShift^{\indexFreeOne}$, which corresponds to a Fourier transform in space whenever $\fourierShift=e^{i\frequency\spaceStep}\in\unitCircle$ for $\frequency\in[-\pi/\spaceStep, \pi/\spaceStep]$, in which case $\schemeMatrixBulkFourier(\fourierShift) $ is the amplification matrix of the scheme. 
Upon formally multiplying by $\fourierShift^{\stencilLeft}$, we obtain a matrix polynomial (or a $\fourierShift$-matrix) $\fourierShift^{\stencilLeft}\schemeMatrixBulkFourier(\fourierShift)$, see \cite{matrixpoly09}, which can also be seen as a matrix pencil of degree (at most) $\stencilLeft + \stencilRight$. 
Matrix pencils generate so-called polynomial eigenvalue problems, which generalize the standard eigenvalue problem: find $(\fourierShift, \vectorial{\eigenvectorLetter})$ such that $(\fourierShift\identityMatrix{} - \matricial{A} ) \vectorial{\eigenvectorLetter} = \zeroMatrix{}$.
Observe that $\schemeMatrixBulkFourier(1) = \collisionMatrix \neq \identityMatrix{\numberVelocities}$, which is different from \cite[Equation (2.13)]{coulombel00616497}, and indicates that the \lbm{} scheme as a whole is not ``consistent'' with a system of $\numberVelocities$ first-order-in-time PDEs, but rather with a differential-algebraic system \cite{martinson2003index, debrabant2005convergence}.

\subsection{Stability of the schemes for the Cauchy problem}\label{sec:stabilityCauchy}

Before addressing the stability of boundary conditions, we would like to explain the notion of stability for schemes \strong{without boundaries} and clarify the relations with analogous concepts for Finite Difference schemes.

\begin{definition}[Simple \strong{von Neumann} polynomial]
	A polynomial with complex coefficients is said to be a ``simple \strong{von Neumann}'' polynomial if all its roots are in $\closedUnitDisk$ and those on $\unitCircle$ are simple.
\end{definition}

\begin{definition}[\emph{von Neumann} stability]\label{def:vonNeumannStability}
	The \lbm{} scheme on the infinite domain (\idEst{} \eqref{eq:bulkScheme} for $\indexSpace\in\relatives$) is said to be stable according to \emph{von Neumann} if, for all $\waveNumber\in[-\pi, \pi]$ we have that $\spectrum(\schemeMatrixBulkFourier(e^{i\waveNumber}))\subset \closedUnitDisk$.
\end{definition}

We propose the following definition based on \cite[Definition 1]{coulombel00616497}, which applies to one-step schemes for systems.
\begin{definition}[Stability of the \lbm{} scheme]\label{def:stabLBM}
	The \lbm{} scheme on the infinite domain (\idEst{} \eqref{eq:bulkScheme} for $\indexSpace\in\relatives$) with non-zero initial data is said to be $L^2$ stable if there exists a constant $C>0$ such that for all $\spaceStep>0$, for all initial condition $(\vectorial{\momentDiscrete}_{\indexSpace}^0)_{\indexSpace\in\relatives} \in (\ell^2_{\spaceStep})^{\numberVelocities}$, and for all $\indexTime \in \naturals$
	\begin{equation*}
		\sum_{\indexSpace\in\relatives} \spaceStep|\vectorial{\momentDiscrete}_{\indexSpace}^{\indexTime}|^2 \leq C \sum_{\indexSpace\in\relatives} \spaceStep|\vectorial{\momentDiscrete}^0_{\indexSpace}|^2.
	\end{equation*}
    Here, $|\cdot|$ denotes any norm on $\reals^{\numberVelocities}$.
\end{definition}
Notice that requesting a control from the initial data written on the moments is equivalent to a control with data written on the distribution function, since $\momentMatrix$ is invertible.
We recall \cite[Proposition 1]{coulombel00616497}.
\begin{proposition}[Stability of the \lbm{} scheme]\label{prop:stabLBM}
	The \lbm{} scheme on the infinite domain is $L^2$ stable, according to \Cref{def:stabLBM}, if and only if 
		\begin{equation}\label{eq:L2PowerBoundedness}
			\forall \indexTime \in \naturals, \qquad \forall \waveNumber\in[-\pi, \pi], \qquad |\schemeMatrixBulkFourier(e^{i\waveNumber})^{\indexTime}| \leq \sqrt{C},
		\end{equation}
		with $C$ the same constant as in \Cref{def:stabLBM}.
\end{proposition}

\begin{proposition}[Necessary conditions for the stability of the \lbm{} scheme]\label{prop:necConditions}
	The following are necessary conditions for the stability of the \lbm{} scheme according to \Cref{def:stabLBM}, with the second one being stronger.
	\begin{enumerate}
		\item \Cref{def:vonNeumannStability} is fulfilled, \idEst{} the scheme is \strong{von Neumann}--stable.
		\item For all $\waveNumber\in[-\pi, \pi]$, the eigenvalues of $\schemeMatrixBulkFourier(e^{i\waveNumber})$ are in $\closedUnitDisk$ and those on $\unitCircle$ are semi-simple. This is equivalent to the fact that, for all $\waveNumber\in[-\pi, \pi]$, the minimal polynomial of $\schemeMatrixBulkFourier(e^{i\waveNumber})$ is a simple \strong{von Neumann} polynomial.
	\end{enumerate}
\end{proposition}
\begin{remark}[On a reason why $\relaxationParameterLetter_{\indexVelocity} \in {[0, 2]}$ for $\indexVelocity\in\integerInterval{2}{\numberVelocities}$]\label{rem:reasonSChoice}
	We have $\spectrum(\schemeMatrixBulkFourier(1)) = \{1, 1-\relaxationParameterLetter_2, \dots, 1-\relaxationParameterLetter_{\numberVelocities}\}$, thus the first necessary condition from \Cref{prop:necConditions} prescribes $\relaxationParameterLetter_{\indexVelocity}\in[0, 2]$ for $\indexVelocity\in\integerInterval{2}{\numberVelocities}$, as previously proposed without justification in \Cref{sec:numericalSchemes}.
	Otherwise, schemes would be unstable even for space-constant initial data.
\end{remark}
We first state the following useful result, whose proof is in \Cref{app:lemma:semiSimple}, before proving \Cref{prop:necConditions}.
\begin{lemma}\label{lemma:semiSimple}
	Let $N>1$, $\matricial{A}\in\matrixSpace{N}{\complex}$, and $\timeShiftOperator\in\complex$ be an eigenvalue of $\matricial{A}$.
	Then 
	\begin{equation*}
		\timeShiftOperator\text{ is semi-simple} \qquad \Longleftrightarrow\qquad \text{the multiplicity of }\timeShiftOperator\text{ as zero of the minimal polynomial of }\matricial{A}\text{ is one.}
	\end{equation*}
\end{lemma}

\begin{proof}[Proof of \Cref{prop:necConditions}]
	The first condition is obvious and is \cite[Corollary 1]{coulombel00616497}.
	The second condition is \cite[Lemma 3]{coulombel00616497}, which ensures power boundedness frequency-wise. 
	However, this is not sufficient since \eqref{eq:L2PowerBoundedness} must hold with a uniform constant.
	The equivalence to the condition featuring the minimal polynomial follows from \Cref{lemma:semiSimple}.
\end{proof}

Adapting \cite[Definition 2]{coulombel00616497} to multi-step scalar schemes, we have the following.
\begin{definition}[Stability of the corresponding Finite Difference scheme generated by the characteristic polynomial]\label{def:stabFD}
	Let $\determinant(\timeShiftOperator\identityMatrix{\numberVelocities} - \schemeMatrixBulk) = \timeShiftOperator^{\numberVelocities} + \sum_{\indexFreeOne = \numberVelocities-1-s}^{\numberVelocities-1}\gamma_{\indexFreeOne}\timeShiftOperator^{\indexFreeOne}$ where $s\definitionEquality\cardinality{\{\relaxationParameterLetter_{\indexFreeOne}\, : \,\indexFreeOne\in\integerInterval{2}{\numberVelocities}\,\text{and} \, \relaxationParameterLetter_{\indexFreeOne}\neq 1\}}$ counts the number of moments that do not relax on their equilibrium during \eqref{eq:relaxation}.
    Consider the corresponding Finite Difference scheme generated by $\determinant(\timeShiftOperator\identityMatrix{\numberVelocities} - \schemeMatrixBulk)$, under the form 
    \begin{equation}\label{eq:FDScheme}
        \solutionCauchyProblem_{\indexSpace}^{\indexTime+1} = -\sum_{\indexFreeOne = 0}^{s}\gamma_{\numberVelocities - \indexFreeOne - 1}\solutionCauchyProblem_{\indexSpace}^{\indexTime-\indexFreeOne}, \qquad \indexTime\geq s, \quad \indexSpace\in\relatives.
    \end{equation}
    This scheme is said to be $L^2$ stable if there exists a constant $C>0$ such that for all $\spaceStep>0$, for all initial condition $(\solutionCauchyProblem_{\indexSpace}^{0})_{\indexSpace\in\relatives}, \dots, (\solutionCauchyProblem_{\indexSpace}^{s})_{\indexSpace\in\relatives} \in \ell^2_{\spaceStep}$, and for all $\indexTime \in \naturals$
	\begin{equation*}
		\sum_{\indexSpace\in\relatives} \spaceStep(\solutionCauchyProblem_{\indexSpace}^{\indexTime})^2 \leq C \sum_{\indexFreeOne = 0}^{s} \sum_{\indexSpace\in\relatives} \spaceStep (\solutionCauchyProblem^{\indexFreeOne}_{\indexSpace})^2.
	\end{equation*}
\end{definition}
Notice that if $\momentDiscrete_1$ is obtained by the \lbm{} scheme, it also fulfills \eqref{eq:FDScheme} (up to machine-precision).
Merging \cite[Proposition 2]{coulombel00616497} with \cite[Theorem 4.2.1]{strikwerda2004finite}, we have the following.
\begin{proposition}[Stability of the corresponding Finite Difference scheme generated by the characteristic polynomial]\label{prop:stabFD}
    Denoting $\hat{\matricial{C}}(\fourierShift)$ the companion matrix of $\timeShiftOperator^{s - \numberVelocities + 1}\determinant(\timeShiftOperator\identityMatrix{\numberVelocities} - \schemeMatrixBulkFourier(\fourierShift))$, the corresponding Finite Difference scheme generated by the characteristic polynomial is $L^2$ stable according to \Cref{def:stabFD} if and only if there exists $\tilde{C}>0$ such that 
    \begin{equation*}
        \forall \indexTime \in \naturals, \qquad \forall \waveNumber\in[-\pi, \pi], \qquad |\hat{\matricial{C}}(e^{i\waveNumber})^{\indexTime}| \leq \tilde{C}.
    \end{equation*}
	This is also equivalent to the fact that, for all $\waveNumber\in[-\pi, \pi]$, the characteristic polynomial $\determinant(\timeShiftOperator\identityMatrix{\numberVelocities} - \schemeMatrixBulkFourier(e^{i\waveNumber}))$ be a simple von Neumann polynomial.
\end{proposition}

\begin{proposition}[Stable Finite Difference entails stable \lbm{}]\label{prop:stableFDDoncStableLBM}
	If the corresponding Finite Difference scheme generated by $\determinant(\timeShiftOperator\identityMatrix{\numberVelocities} - \schemeMatrixBulk)$ is $L^2$ stable according to \Cref{def:stabFD}, then the \lbm{} scheme is $L^2$ stable according to \Cref{def:stabLBM}.
	The converse is false.
\end{proposition}
\begin{proof}
	Preliminarly note that if the corresponding Finite Difference scheme generated by $\determinant(\timeShiftOperator\identityMatrix{\numberVelocities} - \schemeMatrixBulk)$ is $L^2$ stable, then necessarily \Cref{def:vonNeumannStability} holds for the \lbm{} scheme.
	We want to use \cite[Proposition 3]{coulombel00616497} to conclude.
	To this end, we want to show that $\schemeMatrixBulk$ is ``geometrically regular'' according to \cite[Definition 3]{coulombel00616497}.
	Precisely, $\schemeMatrixBulk$ is geometrically regular if for $\targetFourierShift\in\unitCircle$ and $\targetEigenvalue\in\spectrum(\schemeMatrixBulkFourier(\targetFourierShift))\cap \unitCircle$ with algebraic multiplicity $\alpha$, there exist $\timeShiftOperator_1(\fourierShift), \dots, \timeShiftOperator_{\alpha}(\fourierShift)$ holomorphic in a complex neighborhood $\mathscr{W}$ of $\targetFourierShift$ such that $\timeShiftOperator_1(\targetFourierShift) = \dots = \timeShiftOperator_{\alpha}(\targetFourierShift) = \targetEigenvalue$ and $\determinant(\timeShiftOperator\identityMatrix{\numberVelocities} - \schemeMatrixBulkFourier(\targetFourierShift) = \vartheta(\timeShiftOperator, \fourierShift)\prod_{\indexFreeOne = 1}^{\alpha}(\timeShiftOperator - \timeShiftOperator_{\indexFreeOne}(\fourierShift))$ with $(\timeShiftOperator, \fourierShift)\mapsto \vartheta(\timeShiftOperator, \fourierShift)$ holomorphic in a neighborhood of $(\targetEigenvalue, \targetFourierShift)$ and $\vartheta(\targetEigenvalue, \targetFourierShift)\neq 0$, and there exist vectors $\vectorial{\eigenvectorLetter}_1(\fourierShift), \dots, \vectorial{\eigenvectorLetter}_{\alpha}(\fourierShift)\in\complex^{\numberVelocities}$ holomorphic and linearly independent in $\mathscr{W}$ satisfying $\schemeMatrixBulkFourier (\fourierShift)\vectorial{\eigenvectorLetter}_{\indexFreeOne}(\fourierShift) = \timeShiftOperator_{\indexFreeOne}(\fourierShift) \vectorial{\eigenvectorLetter}_{\indexFreeOne}(\fourierShift)$ for $\indexFreeOne\in\integerInterval{1}{\alpha}$ and $\fourierShift\in\mathscr{W}$.

	Let $\targetFourierShift\in\unitCircle$ and $\targetEigenvalue\in\spectrum(\schemeMatrixBulkFourier(\targetFourierShift))\cap \unitCircle$.
	By \Cref{prop:stabFD}, $\targetEigenvalue$ is simple so we take $\alpha = 1$.
	Since $\schemeMatrixBulkFourier(\fourierShift)$ is analytic on $\puncturedPlane$, \cite[Theorem 1]{greenbaum2020first} ensures that it exists $\timeShiftOperator_1(\fourierShift)$ holomorphic in a neighborhood $\mathscr{W}$ of $\targetFourierShift$, such that $\timeShiftOperator_1(\targetFourierShift) = \targetEigenvalue$.
	Moreover, \cite[Theorem 2]{greenbaum2020first} entails that it exists a holomorphic vector-valued function $\vectorial{\eigenvectorLetter}_1(\fourierShift)$ on $\mathscr{W}$ such that $\schemeMatrixBulkFourier(\fourierShift)\vectorial{\eigenvectorLetter}_1(\fourierShift) = \timeShiftOperator_1(\fourierShift)\vectorial{\eigenvectorLetter}_1(\fourierShift)$ on $\mathscr{W}$.
	Now let $\timeShiftOperator_2, \dots, \timeShiftOperator_{\numberVelocities}\in \spectrum(\schemeMatrixBulkFourier(\targetFourierShift))\smallsetminus\{\targetEigenvalue \}$. Again by \cite[Theorem 1]{greenbaum2020first}, there exist $\timeShiftOperator_2(\fourierShift), \dots, \timeShiftOperator_{\numberVelocities}(\fourierShift)$ holomorphic functions in $\mathscr{W}$ (one can shrink $\mathscr{W}$ around $\targetFourierShift$ if needed) such that $\timeShiftOperator_{\indexFreeOne}(\targetFourierShift) = \timeShiftOperator_{\indexFreeOne}$ for $\indexFreeOne\in\integerInterval{2}{\numberVelocities}$.
	We thus set 
	\begin{equation*}
		\vartheta(\timeShiftOperator, \fourierShift)= \prod_{\indexFreeOne = 2}^{\numberVelocities} (\timeShiftOperator - \timeShiftOperator_{\indexFreeOne}(\fourierShift))
	\end{equation*}
	which is holomorphic in a neighborhood of $(\targetEigenvalue, \targetFourierShift)$ and such that $\vartheta(\targetEigenvalue, \targetFourierShift) = \prod_{\indexFreeOne = 2}^{\numberVelocities} (\targetEigenvalue - \timeShiftOperator_{\indexFreeOne})\neq 0$.
	The function $\vartheta(\timeShiftOperator, \fourierShift)$ does the job at establishing the identity $\determinant(\timeShiftOperator\identityMatrix{\numberVelocities} - \schemeMatrixBulkFourier(\fourierShift)) = \vartheta(\timeShiftOperator, \fourierShift) (\timeShiftOperator -\timeShiftOperator_1(\fourierShift))$ since polynomials with complex coefficients split on $\complex$.
	All these facts show that $\schemeMatrixBulk$ is geometrically regular and allow concluding.

	The converse is false, \strong{e.g.} the scheme presented in \cite{bellotti:hal-04358349} and \Cref{sec:D1Q3FourthOrder}.
	Issues are generally due to the fact that semi-simple eigenvalues of $\schemeMatrixBulk$ are not such for the companion matrix of $\determinant(\timeShiftOperator\identityMatrix{\numberVelocities} - \schemeMatrixBulk)$.
\end{proof}

\subsection{Characteristic equation and its roots}\label{sec:charEqAndRoots}

We shall see that a pivotal concept in the study of the boundary value problem is the ``\strong{characteristic equation}''
\begin{equation}\label{eq:characteristicEquation}
	\determinant(\timeShiftOperator\identityMatrix{\numberVelocities} - \schemeMatrixBulkFourier(\fourierShift)) = \sum_{\indexFreeOne = -\stencilLeftCharacteristic}^{\stencilRightCharacteristic} \coefficientCharEquationInFourier_{\indexFreeOne}(\timeShiftOperator) \fourierShift^{\indexFreeOne}= \timeShiftOperator^{\numberVelocities} - \trace(\schemeMatrixBulkFourier(\fourierShift))\timeShiftOperator^{\numberVelocities - 1} + \dots + (-1)^{\numberVelocities}\determinant(\schemeMatrixBulkFourier(\fourierShift)) = 0,
\end{equation}
which can be seen as an equation on $\fourierShift$ with $\timeShiftOperator$ as a parameter.
The coefficients $\coefficientCharEquationInFourier_{\indexFreeOne}(\timeShiftOperator)$ are polynomials in $\timeShiftOperator$ of degree at most $\numberVelocities$ for $\indexFreeOne = 0$ and $\numberVelocities - 1$ otherwise.
Here, $\stencilLeftCharacteristic$ (respectively $\stencilRightCharacteristic$) is the stencil of the characteristic equation to the left (respectively, to the right), and we assume that $\coefficientCharEquationInFourier_{-\stencilLeftCharacteristic} \not \equiv 0$ and $\coefficientCharEquationInFourier_{\stencilRightCharacteristic} \not \equiv 0$.
In general, $\stencilLeftCharacteristic\neq \stencilLeft$ and $\stencilRightCharacteristic\neq \stencilRight$, even if we shall focus on cases where equality holds.
Notwithstanding, $\stencilLeftCharacteristic$ and $\stencilRightCharacteristic$ are difficult to determine abstractly due to the \strong{characteristic} (or \strong{non-monic} from the standpoint of matrix polynomials, \confer{} \cite{matrixpoly09}) nature of the problem.
	\begin{example}\label{ex:D1Q2Char}
		To provide an example of \eqref{eq:characteristicEquation}, consider the \lbmScheme{1}{2} scheme to be introduced in \Cref{sec:D1Q2}, with $\relaxationParameterLetter_2 = \tfrac{3}{2}$ and $\courantNumber=-\tfrac{1}{2}$.
		In this case, we have $\stencilLeftCharacteristic=\stencilRightCharacteristic = 1$, and \eqref{eq:characteristicEquation} reads $\tfrac{\timeShiftOperator}{8}\fourierShift^{-1} - \tfrac{1}{2}+\timeShiftOperator^2 - \tfrac{5\timeShiftOperator}{8}\fourierShift = 0$.
		For $\relaxationParameterLetter_2 = \tfrac{4}{3}$ and $\courantNumber=-\tfrac{1}{2}$, we have $\stencilLeftCharacteristic = 0$ and $\stencilRightCharacteristic = 1$, with  \eqref{eq:characteristicEquation} given by $- \tfrac{1}{3}+\timeShiftOperator^2 - \tfrac{2\timeShiftOperator}{3}\fourierShift = 0$.
	\end{example}

Using the matrix-determinant lemma \cite[Equation (0.8.5.11)]{horn2012matrix} and the formula for the adjugate of a diagonal matrix, we obtain 
\begin{equation}\label{eq:traceAndDeterminant}
	\trace(\schemeMatrixBulkFourier(\fourierShift)) = \sum_{\indexFreeOne = -\stencilLeft}^{\stencilRight} \trace(\schemeMatrixBulkByPower_{\indexFreeOne})\fourierShift^{\indexFreeOne} \qquad\text{and}\qquad 
	\determinant(\schemeMatrixBulkFourier(\fourierShift)) = \Bigl ( \prod_{\indexFreeOne = 2}^{\numberVelocities}(1-\relaxationParameterLetter_{\indexFreeOne})\Bigr ) \Bigl ( \prod_{\indexVelocity = 1}^{\numberVelocities}\fourierShift^{\indexVelocity}\Bigr ).
\end{equation}
This allows to grasp that, except in the pathological cases where $\trace(\schemeMatrixBulkByPower_{-\stencilLeft})$ or $\trace(\schemeMatrixBulkByPower_{\stencilRight})$ vanish, one normally has $\stencilLeftCharacteristic \geq \stencilLeft$ and $\stencilRightCharacteristic\geq \stencilRight$.

\begin{proposition}\label{prop:maximumStencilCharacteristics}
	Let $\numberVelocities = 2, 3, \dots$, $\momentMatrix \in \linearGroup{\numberVelocities}{\baseField}$, $\matricial{K} \in \matrixSpace{\numberVelocities}{\baseField}$, and $\delta_1, \dots, \delta_{\numberVelocities} \in \relatives$.
	Then, $L_{\timeShiftOperator}(\fourierShift)=\determinant(\timeShiftOperator\identityMatrix{\numberVelocities} -\momentMatrix \diagonalMatrix(\fourierShift^{\delta_1}, \dots, \fourierShift^{\delta_{\numberVelocities}})\momentMatrix^{-1}\matricial{K})$ is a Laurent polynomial in $\fourierShift$ with 
	\begin{equation*}
		\min\degree(L_{\timeShiftOperator}) \geq \sum_{\delta_{\indexVelocity}<0}\delta_{\indexVelocity} \qquad \text{and} \qquad \max\degree(L_{\timeShiftOperator}) \leq  \sum_{\delta_{\indexVelocity}>0}\delta_{\indexVelocity}.
	\end{equation*}
\end{proposition}
\begin{remark}[Sharpness of \Cref{prop:maximumStencilCharacteristics}]
	Take $\momentMatrix  = \matricial{K} = \identityMatrix{\numberVelocities}$, so that we obtain 
	\begin{align*}
		L_{\timeShiftOperator}(\fourierShift) &= \Bigl ( \prod_{\delta_{\indexVelocity}<0} (\timeShiftOperator - \fourierShift^{\delta_{\indexVelocity}})\Bigr ) \Bigl ( \prod_{\delta_{\indexVelocity}=0} (\timeShiftOperator - 1)\Bigr ) \Bigl ( \prod_{\delta_{\indexVelocity}>0} (\timeShiftOperator - \fourierShift^{\delta_{\indexVelocity}})\Bigr )\\
		&= (-1)^{\cardinality{ \{ \delta_{\indexVelocity}>0 \} }} \timeShiftOperator^{\cardinality{ \{ \delta_{\indexVelocity}<0 \} }} \Bigl ( \prod_{\delta_{\indexVelocity}=0} (\timeShiftOperator - 1)\Bigr ) \fourierShift^{\sum_{\delta_{\indexVelocity}>0}\delta_{\indexVelocity}} + \text{lower-orders in }\fourierShift.
	\end{align*}
	This shows that the bounds  of \Cref{prop:maximumStencilCharacteristics} are sharp.
\end{remark}

Taking $\delta_{\indexVelocity} = -\dimensionlessDiscreteVelocityLetter_{\indexVelocity}$ for $\indexVelocity \in \integerInterval{1}{\numberVelocities}$ and $\matricial{K} = \identityMatrix{\numberVelocities}  + \diagonalMatrix(\relaxationParameterLetter_1, \dots, \relaxationParameterLetter_{\numberVelocities}) (\equilibriumVector\transpose{\canonicalBasisVector{1}} - \identityMatrix{\numberVelocities})$ as the relaxation matrix in \Cref{prop:maximumStencilCharacteristics} (which is proved in \Cref{app:prop:maximumStencilCharacteristics}) gives the following result.
\begin{corollary}[Degrees of the characteristic equation in $\fourierShift$, \idEst{} in space]\label{cor:stencilFDBounds}
	For the minimal (respectively maximal) degree $\stencilLeftCharacteristic$ (respectively $\stencilRightCharacteristic$) of \eqref{eq:characteristicEquation} in $\fourierShift$, we have the bounds 
	\begin{equation*}
		\stencilLeftCharacteristic \leq \sum_{\dimensionlessDiscreteVelocityLetter_{\indexVelocity}>0} \dimensionlessDiscreteVelocityLetter_{\indexVelocity} \qquad \text{and}\qquad \stencilRightCharacteristic \leq - \sum_{\dimensionlessDiscreteVelocityLetter_{\indexVelocity}<0} \dimensionlessDiscreteVelocityLetter_{\indexVelocity}.
	\end{equation*}
\end{corollary}
This result proves a fact that, although systematically observed in \cite{bellotti2022finite, bellotti2023truncation, bellotti2024initialisation, bellotti:hal-04358349}, has never been rigorously quantified and proved: once turning \lbm{} scheme into Finite Difference using the characteristic polynomial (\confer{} \Cref{def:stabFD}), the spatial stencil (degree in $\fourierShift$ in \eqref{eq:characteristicEquation}) of the scheme remains ``compact'' despite a potentially large time stencils (degree in $\timeShiftOperator$ in \eqref{eq:characteristicEquation}) $\propto \numberVelocities$, see \Cref{fig:domainOfDependence}.
The domain of dependence is in general not a triangle as misleadingly depicted in \cite{bellotti2022finite} and the space stencil not proportional to $\numberVelocities\times\max_{\indexVelocity\in\integerInterval{1}{\numberVelocities}}|\dimensionlessDiscreteVelocityLetter_{\indexVelocity}|$.
Indeed, notice that trivially $\sum_{\dimensionlessDiscreteVelocityLetter_{\indexVelocity}>0} \dimensionlessDiscreteVelocityLetter_{\indexVelocity}\leq \numberVelocities\times\max_{\indexVelocity\in\integerInterval{1}{\numberVelocities}}|\dimensionlessDiscreteVelocityLetter_{\indexVelocity}|$, and in general this is a strict inequality.
\Cref{cor:stencilFDBounds} would be false---and stencils of the Finite Difference counterpart potentially very large---if the transport phase of the \lbm{} scheme were not diagonal in some basis (the one of the distribution function) independent of $\fourierShift$.

\begin{figure}[h]
	\begin{center}
		\begin{tikzpicture}
			\draw[->] (-4,0) -- (4,0) node[right] {$\spaceVariable$};
			\draw[->] (0,-0.2) -- (0,4) node[above] {$\timeVariable$};

			\draw[dashed] (-4*0.7,-0.25) node[below] {$-\stencilLeftCharacteristic = -4$} -- (-4*0.7,3.5);
			\draw[dashed] ( 4*0.7,-0.25) node[below] {$\stencilRightCharacteristic = 4$}  -- ( 4*0.7,3.5);

			\foreach \x in {-5, ...,5}
			  \draw (\x*0.7,0.1) -- (\x*0.7,-0.1);
	
			\node at (-1*0.7, -0.5) {$\spaceGridPoint{\indexSpace-1}$};
			\node at (0, -0.5) {$\spaceGridPoint{\indexSpace}$};
			\node at ( 1*0.7, -0.5) {$\spaceGridPoint{\indexSpace+1}$};
	
			\foreach \y in {1,...,8}
			  \draw (0.1,\y*0.4) -- (-0.1,\y*0.4);
	
			\foreach \x in {-4,...,4}
				\foreach \y in {2,...,6}
					\fill[color=RoyalBlue](\x*0.7,\y*0.4) circle (2pt);
			\foreach \x in {-2,...,2}
				\fill[color=RoyalBlue](\x*0.7,7*0.4) circle (2pt);
			\fill[color=RoyalBlue](0*0.7,1*0.4) circle (2pt) node[right, color=black] {$\timeGridPoint{\indexTime - 6}$};
			\fill[color=OrangeRed](0*0.7,8*0.4) circle (2pt) node[right, color=black] {$\timeGridPoint{\indexTime + 1}$};
	
			\draw[->] (5*0.7,7*0.4) node[right] {``$\trace(\schemeMatrixBulkFourier(\fourierShift))$'' step, \confer{} \eqref{eq:characteristicEquation}} -- (2.3*0.7,7*0.4) ;
			\draw[->] (5*0.7,1*0.4) node[right] {``$\determinant(\schemeMatrixBulkFourier(\fourierShift))$'' step, \confer{} \eqref{eq:characteristicEquation}} -- (2.3*0.7,1*0.4) ;
	
		  \end{tikzpicture}
	\end{center}\caption{\label{fig:domainOfDependence}Domain of dependence (blue dots) in the characteristic equation for a (science fiction) scheme with $\numberVelocities = 7$ discrete velocities equal to $0, \pm 1, \pm 1$, and $\pm 2$.}
\end{figure}

Note that $\stencilLeftCharacteristic$ and $\stencilRightCharacteristic$ might depend on the values of the relaxation parameters--equilibria, see \Cref{ex:D1Q2Char} and \Cref{sec:D1Q2}.
This illustrates a peculiar feature of \lbm{} schemes: the number of grid points where boundary conditions are needed at kinetic level may not match the width of the stencil in the characteristic equation, because the latter also takes the relaxation phase into account.
Under the assumption that $\coefficientCharEquationInFourier_{-\stencilLeftCharacteristic} \not \equiv 0$ and $\coefficientCharEquationInFourier_{\stencilRightCharacteristic} \not \equiv 0$, we can re-use \cite[Lemma 1]{coulombel2009stability} applied to \eqref{eq:characteristicEquation} with some modifications.
These modifications have to take into account the fact that $\coefficientCharEquationInFourier_{-\stencilLeftCharacteristic}(\timeShiftOperator)$ or $\coefficientCharEquationInFourier_{\stencilRightCharacteristic}(\timeShiftOperator)$  may vanish for some $\timeShiftOperator\in\closedNeighborhoodInfinity$, going against \cite[Assumption 1]{coulombel2009stability} applied to these two coefficients.
\begin{lemma}[Hersh]\label{lemma:Hersh}
	Assume that $\coefficientCharEquationInFourier_{-\stencilLeftCharacteristic} \not \equiv 0$ and $\coefficientCharEquationInFourier_{\stencilRightCharacteristic} \not \equiv 0$.
	Let the \lbm{} scheme be \emph{von Neumann} stable according to \Cref{def:vonNeumannStability}, then no root $\fourierShift (\timeShiftOperator)$ of \eqref{eq:characteristicEquation} is such that $\fourierShift (\timeShiftOperator)\in\unitCircle$ for $\timeShiftOperator\in\neighborhoodInfinity$.
	Moreover, the number of roots $\fourierShift (\timeShiftOperator)$ of \eqref{eq:characteristicEquation} in $\unitDisk$, called ``\strong{stable}'' (respectively in $\neighborhoodInfinity$, called ``\strong{unstable}'') for $\timeShiftOperator\in\neighborhoodInfinity$  is equal to $\stencilLeftCharacteristic$ (respectively $\stencilRightCharacteristic$), except at a finite number of isolated values of $\timeShiftOperator$, where some stable roots can be thought to equal zero and some unstable roots to equal $\infty$, by continuity in $\timeShiftOperator$.

	Furthermore, assuming that $\stencilLeft\geq 1$ (the scheme has at least one strictly right-going discrete velocity) and that we face a symmetric velocity set (for every $\dimensionlessDiscreteVelocityLetter_{\indexVelocity}\neq 0$, there exists $\tilde{\indexVelocity}\in\integerInterval{1}{\numberVelocities}$ such that $\dimensionlessDiscreteVelocityLetter_{\tilde{\indexVelocity}} = - \dimensionlessDiscreteVelocityLetter_{\indexVelocity}$), the number of these isolated (singular) points is at most $\numberVelocities-2$ if $\stencilLeftCharacteristic = \stencilLeft$, and at most $\numberVelocities-3$ if $\stencilLeftCharacteristic > \stencilLeft$.
\end{lemma}
\begin{proof}
	Notice that $\coefficientCharEquationInFourier_{\stencilLeftCharacteristic}(\timeShiftOperator)$ and $\coefficientCharEquationInFourier_{\stencilRightCharacteristic}(\timeShiftOperator)$ are polynomials in $\timeShiftOperator$. They thus have a finite number of zeros, whose number equals their degree.
	The first part of this result is \cite[Lemma 1]{coulombel2009stability} applied to a scalar Finite Difference scheme.
	Let us now consider the second part of the claim.
	Thanks to the symmetric velocity set, from \eqref{eq:traceAndDeterminant}, we see that $\determinant(\schemeMatrixBulkFourier(\fourierShift))$ does not depend on $\fourierShift$. Furthermore, we have that $\min\degree(\trace(\schemeMatrixBulkFourier(\fourierShift)))\geq -\stencilLeft$ and $\max\degree(\trace(\schemeMatrixBulkFourier(\fourierShift)))\leq \stencilRight$.
	\begin{itemize}
		\item If $\stencilLeftCharacteristic = \stencilLeft$ and $\stencilRightCharacteristic = \stencilRight$, then inspecting \eqref{eq:characteristicEquation}, we see that $\coefficientCharEquationInFourier_{-\stencilLeftCharacteristic}(\timeShiftOperator)$ and $\coefficientCharEquationInFourier_{\stencilRightCharacteristic}(\timeShiftOperator)$ are polynomials with terms in the range $\timeShiftOperator, \dots, \timeShiftOperator^{\numberVelocities-1}$, which therefore have at most $\numberVelocities - 2$ roots.
		\item If $\stencilLeftCharacteristic > \stencilLeft$ and $\stencilRightCharacteristic > \stencilRight$, the same way of proceeding unveils that $\coefficientCharEquationInFourier_{-\stencilLeftCharacteristic}(\timeShiftOperator)$ and $\coefficientCharEquationInFourier_{\stencilRightCharacteristic}(\timeShiftOperator)$ are polynomials with terms in the range $\timeShiftOperator, \dots, \timeShiftOperator^{\numberVelocities-2}$, which therefore have at most $\numberVelocities - 3$ roots outside the unit disk each.
	\end{itemize}
\end{proof}

\subsection{Strong stability and boundary-value problem as a recurrence relation in space}

\subsubsection{Strong stability}

Well-posedness of numerical schemes is built on three conditions. The first two, namely existence and uniqueness of the solution, are automatically met for the explicit schemes we deal with.
The third condition concerns \strong{continuity with respect to data}, the one we now discuss.
Since in \eqref{eq:bulkScheme}, \eqref{eq:boundaryScheme}, and \eqref{eq:zeroInitialData}, the only datum is the boundary one, we would like the solution to depend continuously on this.

\begin{definition}[Strong stability]\label{def:strongStability}
	We say that the lattice Boltzmann scheme \eqref{eq:bulkScheme}--\eqref{eq:boundaryScheme}--\eqref{eq:zeroInitialData} is \strong{strongly stable} (SS) if there exists a constant $C>0$ such that for all $\alpha>0$, for all $(\boundarySourceTerm_{\indexVelocity, -\indexSpace}^{\indexTime})_{\indexTime\in\naturals}$ with $\indexVelocity$ such that $\dimensionlessDiscreteVelocityLetter_{\indexVelocity}>0$ and $\indexSpace\in\integerInterval{1}{\dimensionlessDiscreteVelocityLetter_{\indexVelocity}}$, and for all $\spaceStep>0$ (recall that $\timeStep = \spaceStep / \latticeVelocity$)
	\begin{equation}\label{eq:strongStability}
		\frac{\alpha}{1+\alpha\timeStep}  \sum_{\indexTime\in\naturals} \sum_{\indexSpace\in\naturals} \timeStep \spaceStep \, e^{-2\alpha \indexTime\timeStep} |\vectorial{\momentDiscrete}_{\indexSpace}^{\indexTime}|^2 \leq C \sum_{\indexTime\in\naturals} \sum_{\dimensionlessDiscreteVelocityLetter_{\indexVelocity}>0}\sum_{\indexSpace = 1}^{\dimensionlessDiscreteVelocityLetter_{\indexVelocity}} \timeStep \, e^{-2\alpha \indexTime\timeStep}(\boundarySourceTerm_{\indexVelocity, -\indexSpace}^{\indexTime})^2.
	\end{equation}
	We say that the scheme is \strong{strongly stable on the observed output} (SSOO) if 
	\begin{equation}\label{eq:strongStabilityObserved}
		\frac{\alpha}{1+\alpha\timeStep}  \sum_{\indexTime\in\naturals} \sum_{\indexSpace\in\naturals}\timeStep\spaceStep \, e^{-2\alpha \indexTime\timeStep} ({\momentDiscrete}_{1, \indexSpace}^{\indexTime})^2 \leq  C \sum_{\indexTime\in\naturals} \sum_{\dimensionlessDiscreteVelocityLetter_{\indexVelocity}>0}\sum_{\indexSpace = 1}^{\dimensionlessDiscreteVelocityLetter_{\indexVelocity}} \timeStep \, e^{-2\alpha \indexTime\timeStep}(\boundarySourceTerm_{\indexVelocity, -\indexSpace}^{\indexTime})^2.
	\end{equation}
\end{definition}

Let us first comment on the name ``\strong{strongly stable on the observed output}''.
As we are interested only in $\momentDiscrete_1$ which approximates the solution of the target PDE, this is the only output that we \strong{observe} from the whole state $\vectorial{\momentDiscrete}$. 
\strong{Strong stability} imposes a control on the whole vector $\vectorial{\momentDiscrete} \in\reals^{\numberVelocities}$ instead of the first component $\momentDiscrete_1$ for the \strong{strong stability on the observed output}: the former trivially implies the latter.
Both \eqref{eq:strongStability} and \eqref{eq:strongStabilityObserved} mean a control of the output of the system (left-hand side) by its input (right-hand side), and the vector space to whom this input belongs is of dimension $\sum_{\dimensionlessDiscreteVelocityLetter_{\indexVelocity}>0} \dimensionlessDiscreteVelocityLetter_{\indexVelocity}$, the upper bound for $\stencilLeftCharacteristic$ by \Cref{cor:stencilFDBounds}.
The reader willing to acquire a more visual and practical understanding of strong stability can jump straight to \Cref{rem:remarkStrongStab}, which however needs a preliminary  assimilation of numerical experiments, \confer{} \Cref{sec:numericalSimD1Q2}.

\subsubsection{$\timeShiftOperator$-transform}

As said in the introduction, we would like to recast the scheme under the form of a recurrence relation in space.
This relation is not explicit, due to the characteristic nature of the boundary.
The final preparatory step consists in seeing time as a complex parameter as already done in \Cref{sec:charEqAndRoots}, which is rigorously achieved by the $\timeShiftOperator$-transform.
We follow \cite[Appendix B]{lebarbenchon:tel-04214887}.
\begin{proposition}[$\timeShiftOperator$-transform, its inverse, and Parseval's identity]\label{prop:ZTransform}
	Let $(\solutionCauchyProblem^{\indexTime})_{\indexTime\in\naturals}\in\ell^2$. We define its $\timeShiftOperator$-transform $\laplaceTransformed{\solutionCauchyProblem}$ by 
	\begin{equation}\label{eq:definitionZTransform}
		\laplaceTransformed{\solutionCauchyProblem}(\timeShiftOperator)\definitionEquality\sum_{\indexTime\in\naturals}\timeShiftOperator^{-\indexTime}\solutionCauchyProblem^{\indexTime}, \qquad \text{for}\quad \timeShiftOperator\in\neighborhoodInfinity.
	\end{equation}
	The inverse transform is given by
	\begin{equation}\label{eq:inverseZTransform}
		\solutionCauchyProblem^{\indexTime}=\frac{1}{2\pi i}\oint_{\mathcal{C}(0, R)}\timeShiftOperator^{\indexTime-1}\laplaceTransformed{\solutionCauchyProblem}(\timeShiftOperator)\differential{\timeShiftOperator} = \frac{1}{2\pi}\int_0^{2\pi} \laplaceTransformed{\solutionCauchyProblem}(R e^{i\vartheta})R^{\indexTime}e^{i\indexTime\vartheta}\differential{\vartheta},
	\end{equation}
	for any $R>1$ and $\indexTime\in\naturals$.
	Finally, the Parseval's identity reads 
	\begin{equation*}
		\sum_{\indexTime\in\naturals} R^{-2\indexTime} |\solutionCauchyProblem^{\indexTime}|^2 = \frac{1}{2\pi} \int_0^{2\pi}|\laplaceTransformed{\solutionCauchyProblem}(R e^{i\vartheta})|^2\differential{\vartheta}.
	\end{equation*}
\end{proposition}
\begin{remark}
	The definition of $\timeShiftOperator$-transform by \eqref{eq:definitionZTransform} works also for sequences not in $\ell^2$ (for example, exponentially diverging), upon considering $ |\timeShiftOperator|>R\geq 1$, where $R$ is the radius of convergence.
	The other formul\ae{} in \Cref{prop:ZTransform} can be adapted as well, see \cite[Appendix]{ehrhardt2001discrete}.
\end{remark}

The complex parameter $\timeShiftOperator$ has to be thought as a (potential) time-eigenvalue. 
As such,
\begin{itemize}
	\item if $\timeShiftOperator\in\unitDisk$, the mode is damped in time;
	\item if $\timeShiftOperator\in\unitCircle$, the mode is time-steady, hence temporally persists in the numerical solution;
	\item if $\timeShiftOperator\in\neighborhoodInfinity$, the mode blows up geometrically, generating so-called \strong{Godunov-Ryabenkii} instabilities.
\end{itemize}

\subsubsection{Recurrence relation in space}

Taking the $\timeShiftOperator$-transform of \eqref{eq:bulkScheme}--\eqref{eq:boundaryScheme} and using the fact that we consider zero initial data, \confer{} \eqref{eq:zeroInitialData}, we obtain
\begin{numcases}{}
    \timeShiftOperator \laplaceTransformed{\vectorial{\momentDiscrete}}_{\indexSpace}(\timeShiftOperator) = \schemeMatrixBulk \laplaceTransformed{\vectorial{\momentDiscrete}}_{\indexSpace}(\timeShiftOperator), \qquad &$\indexSpace\geq \stencilLeft, $\label{eq:zTransBulk} \\
    \timeShiftOperator\laplaceTransformed{\vectorial{\momentDiscrete}}_{\indexSpace}(\timeShiftOperator) = \schemeMatrixBoundary{\indexSpace} \laplaceTransformed{\vectorial{\momentDiscrete}}_{\indexSpace}(\timeShiftOperator) + \laplaceTransformed{\vectorial{\boundarySourceTermMoments}}_{\indexSpace}(\timeShiftOperator), \qquad &$\indexSpace \in \integerInterval{0}{\stencilLeft - 1}.$\label{eq:zTransBoundary}
\end{numcases}
Using the expressions in \eqref{eq:matrixByPowerBulk} and \eqref{eq:matrixByPowerBoundary} and performing changes in the indices, we have 
\begin{numcases}{}
	\timeShiftOperator \laplaceTransformed{\vectorial{\momentDiscrete}}_{\indexSpace + \stencilLeft}(\timeShiftOperator) - \sum_{\indexFreeOne=0}^{\stencilLeft + \stencilRight}\schemeMatrixBulkByPower_{\indexFreeOne - \stencilLeft}\laplaceTransformed{\vectorial{\momentDiscrete}}_{\indexSpace + \indexFreeOne} (\timeShiftOperator) = \zeroMatrix{\numberVelocities}, \qquad &$\indexSpace\in\naturals, $\label{eq:resolventBulk}\\
	\timeShiftOperator\laplaceTransformed{\vectorial{\momentDiscrete}}_{\indexSpace}(\timeShiftOperator) - \sum_{\indexFreeOne = -\indexSpace}^{\stencilBoundaryCondition} \schemeMatrixBoundaryByPower{\indexSpace}{\indexFreeOne} \laplaceTransformed{\vectorial{\momentDiscrete}}_{\indexSpace + \indexFreeOne}(\timeShiftOperator) = \laplaceTransformed{\vectorial{\boundarySourceTermMoments}}_{\indexSpace}(\timeShiftOperator), \qquad &$\indexSpace \in \integerInterval{0}{\stencilLeft - 1},$\label{eq:resolventBoundary}
\end{numcases}
which is sometimes called ``\strong{resolvent equation}''.
Notice that \eqref{eq:resolventBulk} is a vectorial homogeneous linear difference equation in $\indexSpace$ parametrized by $\timeShiftOperator$. 
On the other hand, \eqref{eq:resolventBoundary} yields the initialization in space for this difference equation and makes the solution inside the domain depend on the boundary data.

We denote the vector space of the solutions of \eqref{eq:resolventBulk} in $\ell^2$ by $\stableSubspace(\timeShiftOperator)$, for $\timeShiftOperator\in\neighborhoodInfinity$ (and possibly extended to $\unitCircle$).
Following \cite[Chapter 8]{matrixpoly09}, the study of \eqref{eq:resolventBulk} is conducted introducing the matrix polynomial
\begin{equation}\label{eq:matrixPolynomialGohberg}
	\matrixPolynomialBulk{\timeShiftOperator}(\fourierShift)\definitionEquality \sum_{\indexFreeOne=0}^{\stencilLeft + \stencilRight}\schemeMatrixBulkByPowerIncludingShift_{\indexFreeOne}(\timeShiftOperator)\fourierShift^{\indexFreeOne}, \qquad \text{where}\qquad 
	\schemeMatrixBulkByPowerIncludingShift_{\indexFreeOne}(\timeShiftOperator) \definitionEquality \timeShiftOperator\identityMatrix{\numberVelocities}\delta_{\indexFreeOne\stencilLeft}-\schemeMatrixBulkByPower_{\indexFreeOne-\stencilLeft}.
\end{equation}
The general solution to \eqref{eq:resolventBulk} and thus $\stableSubspace(\timeShiftOperator)$ are constructed using spectral information regarding $\matrixPolynomialBulk{\timeShiftOperator}(\fourierShift)$.
For the complete spectral theory of matrix polynomials, the interested reader can reach \cite[Chapter 7]{matrixpoly09}.
In particular, the \strong{eigenvalues} of $\matrixPolynomialBulk{\timeShiftOperator}(\fourierShift)$ are the $\fourierShift=\fourierShift(\timeShiftOperator)\in\complex$ such that $\determinant(\matrixPolynomialBulk{\timeShiftOperator}(\fourierShift)) = 0$, which must be thought as eigenvalues in space and classified as follows. 
Considering that solutions are the form $\fourierShift(\timeShiftOperator)^{\indexSpace}$,
\begin{itemize}
	\item if $\fourierShift\in\unitDisk$, the mode is geometrically damped in space, hence localized on the boundary where it forms a boundary layer.
	\item If $\fourierShift\in\unitCircle$, the mode extends inside the domain and, if the corresponding $\timeShiftOperator\in\unitCircle$, can be propagated with a group velocity (see \cite{trefethen1984instability} or \cite{coulombel2015fully}).
	\item If $\fourierShift\in\neighborhoodInfinity$, the mode is geometrically explosive in space and does not participate to $\stableSubspace(\timeShiftOperator)$,
\end{itemize}
as the stable subspace $\stableSubspace(\timeShiftOperator)$ is constructed using eigenvalues $\fourierShift(\timeShiftOperator)\in\closedUnitDisk$. 
Notice however that $\matrixPolynomialBulk{\timeShiftOperator}(\fourierShift)$ is non-monic, for the leading-order matrix $\schemeMatrixBulkByPowerIncludingShift_{\stencilLeft + \stencilRight}(\timeShiftOperator) \notin \linearGroup{\numberVelocities}{\complex}$.
This corresponds to the fact that the problem is characteristic, and has at least two important consequences. 

First, $\degree(\determinant(\matrixPolynomialBulk{\timeShiftOperator}(\fourierShift))) < \numberVelocities (\stencilLeft + \stencilRight)$, hence the number of eigenvalues is smaller than $\numberVelocities (\stencilLeft + \stencilRight)$.
To keep the total number of eigenvalues equal to $\numberVelocities (\stencilLeft + \stencilRight)$ as in the monic case, one says that there are $\numberVelocities (\stencilLeft + \stencilRight) - \degree(\determinant(\matrixPolynomialBulk{\timeShiftOperator}(\fourierShift)))$ \strong{infinite} eigenvalues, which are the zero eigenvalues of the reflected matrix polynomial $\fourierShift^{\stencilLeft+\stencilRight}\matrixPolynomialBulk{\timeShiftOperator}(\fourierShift^{-1})$.
Still, there is no easy way to determine $\degree(\determinant(\matrixPolynomialBulk{\timeShiftOperator}(\fourierShift)))$ \emph{a priori}\footnote{Which boils down to computing the associated invariant polynomials/elementary divisors of $\matrixPolynomialBulk{\timeShiftOperator}(\fourierShift)$, see \cite[Proposition S1.4]{matrixpoly09}.}, which directly comes from the difficulty of computing $\stencilLeftCharacteristic$ and $\stencilRightCharacteristic$ in \eqref{eq:characteristicEquation}.

The second important consequence is that we cannot recast \eqref{eq:resolventBulk} into an explicit recurrence using a block-companion matrix. 
Something which can be done is considering a companion matrix polynomial of degree one in $\fourierShift$, see the following easily-provable result.
\begin{lemma}
	Let $\companionPolynomial(\fourierShift)$ be the companion matrix polynomial to $\matrixPolynomialBulk{\timeShiftOperator}(\fourierShift)$ defined by 
	\begin{equation*}
		\companionPolynomial(\fourierShift) \definitionEquality 
		\fourierShift
		\left [
			\begin{array}{c|c}
				\schemeMatrixBulkByPowerIncludingShift_{\stencilLeft + \stencilRight}(\timeShiftOperator) & \zeroMatrix{\numberVelocities \times \numberVelocities(\stencilLeft+\stencilRight - 1)} \\
				\hline
				\zeroMatrix{\numberVelocities(\stencilLeft+\stencilRight - 1)\times \numberVelocities} & \identityMatrix{\numberVelocities(\stencilLeft+\stencilRight - 1)}
			\end{array}
		\right ]
		-
		\left [
		\begin{array}{c|c|c|c|c}
			-\schemeMatrixBulkByPowerIncludingShift_{\stencilLeft + \stencilRight-1}(\timeShiftOperator) & -\schemeMatrixBulkByPowerIncludingShift_{\stencilLeft + \stencilRight-2}(\timeShiftOperator) & \cdots & -\schemeMatrixBulkByPowerIncludingShift_{1}(\timeShiftOperator)  & -\schemeMatrixBulkByPowerIncludingShift_{0}(\timeShiftOperator) \\
			\hline 
			\identityMatrix{\numberVelocities} & \zeroMatrix{\numberVelocities} & \cdots & \zeroMatrix{\numberVelocities} & \zeroMatrix{\numberVelocities} \\
			\hline 
			\zeroMatrix{\numberVelocities} & \identityMatrix{\numberVelocities} & \cdots & \zeroMatrix{\numberVelocities}  & \zeroMatrix{\numberVelocities} \\
			\hline 
			\vdots & \vdots & \ddots & \vdots & \vdots\\
			\hline
			\zeroMatrix{\numberVelocities}  & \zeroMatrix{\numberVelocities}  & \cdots & \identityMatrix{\numberVelocities} & \zeroMatrix{\numberVelocities}
		\end{array}
		\right ].
	\end{equation*}
	Then, formally (that is for $\fourierShift\in\puncturedPlane$) we have 
	\begin{equation}\label{eq:equalityCharEquations}
		\determinant(\matrixPolynomialBulk{\timeShiftOperator}(\fourierShift)) = \determinant (\companionPolynomial(\fourierShift)) = \fourierShift^{\numberVelocities\stencilLeft} \determinant(\timeShiftOperator \identityMatrix{\numberVelocities}-\schemeMatrixBulkFourier(\fourierShift)).
	\end{equation}
\end{lemma}
The fact that the non-zero finite eigenvalues of $\matrixPolynomialBulk{\timeShiftOperator}(\fourierShift)$ are the roots of $\determinant(\timeShiftOperator \identityMatrix{\numberVelocities}-\schemeMatrixBulkFourier(\fourierShift))$ now explains why we have discussed the characteristic equation \eqref{eq:characteristicEquation} in depth in \Cref{sec:charEqAndRoots}.
This fosters the explicit construction of $\stableSubspace(\timeShiftOperator)$.
	\begin{example}
		Coming back to the setting of \Cref{ex:D1Q2Char} with $\relaxationParameterLetter_2 = \tfrac{3}{2}$ and $\courantNumber=-\tfrac{1}{2}$, we have 
		\begin{equation*}
			\matrixPolynomialBulk{\timeShiftOperator}(\fourierShift) = 
			\begin{pmatrix}
				-\tfrac{7}{8}\fourierShift^2 + \timeShiftOperator\fourierShift-\tfrac{1}{8} & -\tfrac{1}{4}\fourierShift^2 + \tfrac{1}{4}\\
				\tfrac{7}{8}\fourierShift^2 - \tfrac{1}{8} & \tfrac{1}{4}\fourierShift^2 + \timeShiftOperator\fourierShift+\tfrac{1}{4}
			\end{pmatrix}
			\quad \text{and}\quad
			\companionPolynomial(\fourierShift) = \fourierShift
			\begin{pmatrix}
				-\tfrac{7}{8} & -\tfrac{1}{4} & 0 & 0\\
				\tfrac{7}{8} & \tfrac{1}{4} & 0 & 0 \\
				0 & 0 & 1 & 0 \\
				0 & 0 & 0 & 1
			\end{pmatrix} - 
			\begin{pmatrix}
				-\timeShiftOperator & 0 & \tfrac{1}{8} & -\tfrac{1}{4}\\ 
				0 & -\timeShiftOperator & \tfrac{1}{8} & -\tfrac{1}{4}\\
				1 & 0 & 0 & 0\\
				0 & 1 & 0 & 0
			\end{pmatrix}.
		\end{equation*}
	\end{example}

\subsubsection{Schemes with characteristic equations with stencil breadth to the left equal to one}\label{sec:recurrenceSpace}

The assumptions of the following proposition apply (outside exceptional cases) to all the examples that we consider in \Cref{sec:strongCertainSchemes}, namely schemes with symmetric velocity sets and two or three discrete velocities.
\begin{proposition}\label{prop:noDiracBoundary}
	Assume that the scheme is stable according to \strong{von Neumann}, \confer{} \Cref{def:vonNeumannStability}, discrete velocities are distinct, and $\stencilLeft = \stencilLeftCharacteristic = 1$, so that $\coefficientCharEquationInFourier_{-1}(\timeShiftOperator)\not\equiv 0$.
	Denote the stable root of \eqref{eq:characteristicEquation}, whose existence comes from \Cref{lemma:Hersh}, by $\stableRoot(\timeShiftOperator)$ and by $\vectorial{\eigenvectorLetter}_{\textnormal{s}}(\timeShiftOperator)\in\kernel(\matrixPolynomialBulk{\timeShiftOperator}(\stableRoot(\timeShiftOperator)))$ a corresponding eigenvector.

	Then, $\fourierShift\equiv 0$ is an eigenvalue of $\matrixPolynomialBulk{\timeShiftOperator}(\fourierShift)$ with algebraic and geometric ($\dimension{\kernel(\matrixPolynomialBulk{\timeShiftOperator}(0))}$) multiplicities both equal to $\numberVelocities - 1$.
	Thus, denoting $\kernel(\matrixPolynomialBulk{\timeShiftOperator}(0))=\spanSpace{\vectorial{\eigenvectorLetter}_0^1, \dots,\vectorial{\eigenvectorLetter}_0^{\numberVelocities-1}}$ where $\vectorial{\eigenvectorLetter}_0^1, \dots,\vectorial{\eigenvectorLetter}_0^{\numberVelocities-1}$ can be chosen independent of $\timeShiftOperator$, the solution of \eqref{eq:resolventBulk} $\laplaceTransformed{\vectorial{\momentDiscrete}}_{\indexSpace}(\timeShiftOperator)\in\stableSubspace(\timeShiftOperator)$ for $\timeShiftOperator\in\neighborhoodInfinity$ reads
	\begin{equation}\label{eq:generalStableSolutionAllSchemes}
		\laplaceTransformed{\vectorial{\momentDiscrete}}_{\indexSpace}(\timeShiftOperator) = \coefficientStableSolution(\timeShiftOperator) \vectorial{\eigenvectorLetter}_{\textnormal{s}}(\timeShiftOperator) \stableRoot(\timeShiftOperator)^{\indexSpace} + \sum_{\indexFreeOne = 1}^{\numberVelocities - 1}\coefficientZero^{\indexFreeOne}(\timeShiftOperator)\vectorial{\eigenvectorLetter}_0^{\indexFreeOne}\delta_{0\indexSpace}, \qquad \indexSpace\in\naturals,
	\end{equation}
	where, in order to also fulfill \eqref{eq:resolventBoundary}, the coefficients $\coefficientStableSolution(\timeShiftOperator), \coefficientZero^{1}(\timeShiftOperator), \dots, \coefficientZero^{\numberVelocities-1}(\timeShiftOperator)$ satisfy 
	\begin{multline}\label{eq:lopatiskiiSystem}
		\kreissLopatinskiiMatrix(\timeShiftOperator) \transpose{(\coefficientStableSolution(\timeShiftOperator), \coefficientZero^{1}(\timeShiftOperator), \dots, \coefficientZero^{\numberVelocities-1}(\timeShiftOperator))} = \laplaceTransformed{\vectorial{\boundarySourceTermMoments}}_0(\timeShiftOperator)\\
		\text{with}\quad  
		\kreissLopatinskiiMatrix(\timeShiftOperator)\definitionEquality 
		\begin{bmatrix}
			(\timeShiftOperator\identityMatrix{\numberVelocities} - \sum_{\indexFreeOne = 0}^{\stencilBoundaryCondition}\schemeMatrixBoundaryByPower{0}{\indexFreeOne}\stableRoot(\timeShiftOperator)^{\indexFreeOne})\vectorial{\eigenvectorLetter}_{\textnormal{s}}(\timeShiftOperator) \, | \,(\timeShiftOperator\identityMatrix{\numberVelocities} - \schemeMatrixBoundaryByPower{0}{0})  \vectorial{\eigenvectorLetter}_{0}^1 \, | \, \cdots \, | \, (\timeShiftOperator\identityMatrix{\numberVelocities} - \schemeMatrixBoundaryByPower{0}{0})  \vectorial{\eigenvectorLetter}_{0}^{\numberVelocities - 1}
		\end{bmatrix}.
	\end{multline}
	Furthermore, denoting $\kreissLopatinskiiDet(\timeShiftOperator)\definitionEquality\determinant(\kreissLopatinskiiMatrix(\timeShiftOperator))$ the so-called ``\strong{Kreiss-Lopatinskii determinant}'', and $\positiveVelocityIndex\in\integerInterval{1}{\numberVelocities}$ the unique index such that $\dimensionlessDiscreteVelocityLetter_{\positiveVelocityIndex} = 1$, the coefficients $\coefficientStableSolution(\timeShiftOperator), \coefficientZero^{1}(\timeShiftOperator), \dots, \coefficientZero^{\numberVelocities-1}(\timeShiftOperator)$ satisfy 
	\begin{equation}\label{eq:coefficientWithKreissLopDet}
		\kreissLopatinskiiDet(\timeShiftOperator) \coefficientStableSolution(\timeShiftOperator) = \scalarFactorFromAdjugate(\timeShiftOperator)\laplaceTransformed{\boundarySourceTerm}_{\positiveVelocityIndex, -1}(\timeShiftOperator) \qquad \text{and} \qquad \kreissLopatinskiiDet(\timeShiftOperator) \coefficientZero^{\indexFreeOne}(\timeShiftOperator) = 0, \quad \indexFreeOne \in\integerInterval{1}{\numberVelocities - 1},
	\end{equation}
	where the scalar function $\scalarFactorFromAdjugate(\timeShiftOperator)$ reads 
	\begin{equation}\label{eq:scalarFactorFromAdjugate}
		\scalarFactorFromAdjugate(\timeShiftOperator) \definitionEquality 
		\determinant
		\begin{bmatrix}
			\momentMatrix\canonicalBasisVector{\positiveVelocityIndex}\, | \,(\timeShiftOperator\identityMatrix{\numberVelocities} - \momentMatrix \sum_{\dimensionlessDiscreteVelocityLetter_{\indexVelocity}= 0 }\canonicalBasisVector{\indexVelocity}\transpose{\canonicalBasisVector{\indexVelocity}} \momentMatrix^{-1}\collisionMatrix )  \vectorial{\eigenvectorLetter}_{0}^1 \, | \, \cdots \, | \, (\timeShiftOperator\identityMatrix{\numberVelocities} - \momentMatrix \sum_{\dimensionlessDiscreteVelocityLetter_{\indexVelocity}= 0 }\canonicalBasisVector{\indexVelocity}\transpose{\canonicalBasisVector{\indexVelocity}} \momentMatrix^{-1}\collisionMatrix )  \vectorial{\eigenvectorLetter}_{0}^{\numberVelocities - 1}
		\end{bmatrix} 
		= c \times  \coefficientCharEquationInFourier_{-1}(\timeShiftOperator),
	\end{equation}
	is a polynomial function in $\timeShiftOperator$, and depends only on the bulk scheme, with $c=c(\vectorial{\eigenvectorLetter}_0^1, \dots, \vectorial{\eigenvectorLetter}_0^{\numberVelocities-1})\neq 0$.
	Moreover, 
	\begin{equation}\label{eq:decompositionKreissLopatinskiiDet}
		\kreissLopatinskiiDet(\timeShiftOperator) 
		=\Bigl (  \transpose{\canonicalBasisVector{\positiveVelocityIndex}}\stableRoot(\timeShiftOperator)^{-1} -\underbrace{\sum_{\indexFreeOne = 0}^{\stencilBoundaryCondition}\sum_{\indexFreeTwo = 1}^{\numberVelocities}\transpose{\canonicalBasisVector{\indexFreeTwo}}\boundaryCoefficient_{\positiveVelocityIndex, \indexFreeTwo, 1, \indexFreeOne}\stableRoot(\timeShiftOperator)^{\indexFreeOne}}_{\text{boundary dependent part}} \Bigr ) \underbrace{\momentMatrix^{-1}\collisionMatrix\vectorial{\eigenvectorLetter}_{\textnormal{s}}(\timeShiftOperator)\scalarFactorFromAdjugate(\timeShiftOperator)}_{\text{bulk dependent part}}.
	\end{equation}
	Finally, setting $\kreissLopatinskiiScalarProduct (\timeShiftOperator)\definitionEquality (  \transpose{\canonicalBasisVector{\positiveVelocityIndex}}\stableRoot(\timeShiftOperator)^{-1} -\sum_{\indexFreeOne = 0}^{\stencilBoundaryCondition}\sum_{\indexFreeTwo = 1}^{\numberVelocities}\transpose{\canonicalBasisVector{\indexFreeTwo}}\boundaryCoefficient_{\positiveVelocityIndex, \indexFreeTwo, 1, \indexFreeOne}\stableRoot(\timeShiftOperator)^{\indexFreeOne}  )\momentMatrix^{-1}\collisionMatrix\vectorial{\eigenvectorLetter}_{\textnormal{s}}(\timeShiftOperator)$, \eqref{eq:coefficientWithKreissLopDet} formally (in the sense, by dividing both sides by the polynomial $\coefficientCharEquationInFourier_{-1}(\timeShiftOperator)$) becomes 
	\begin{equation}\label{eq:coefficientWithKreissLopDetFinal}
		\kreissLopatinskiiScalarProduct(\timeShiftOperator) \coefficientStableSolution(\timeShiftOperator) = \laplaceTransformed{\boundarySourceTerm}_{\positiveVelocityIndex, -1}(\timeShiftOperator) \qquad \text{and} \qquad \kreissLopatinskiiScalarProduct(\timeShiftOperator) \coefficientZero^{\indexFreeOne}(\timeShiftOperator) = 0, \quad \indexFreeOne \in\integerInterval{1}{\numberVelocities - 1}.
	\end{equation}
\end{proposition}
A first take-home message of this result, whose proof---an exercise of linear algebra---is in \Cref{app:prop:noDiracBoundary}, is that the boundary-stitched term proportional to $\delta_{0\indexSpace}$ in \eqref{eq:generalStableSolutionAllSchemes} is zero whenever $\kreissLopatinskiiDet(\timeShiftOperator)\neq 0$.
This indicates that the eigenvalue $\fourierShift\equiv 0$ introduced in the transition from \eqref{eq:zTransBulk} to \eqref{eq:resolventBulk} plays no role.
Moreover, $\coefficientStableSolution(\timeShiftOperator)$ depends on the boundary condition only through the Kreiss-Lopatinskii determinant $\kreissLopatinskiiDet(\timeShiftOperator)$ or the Kreiss-Lopatinskii ``scalar product'' $\kreissLopatinskiiScalarProduct(\timeShiftOperator)$.
This latter is the scalar product of $\transpose{\collisionMatrix}\momentMatrix^{-\mathsf{T}}(  {\canonicalBasisVector{\positiveVelocityIndex}}\stableRoot(\timeShiftOperator)^{-1} -\sum_{\indexFreeOne = 0}^{\stencilBoundaryCondition}\sum_{\indexFreeTwo = 1}^{\numberVelocities}{\canonicalBasisVector{\indexFreeTwo}}\boundaryCoefficient_{\positiveVelocityIndex, \indexFreeTwo, 1, \indexFreeOne}\stableRoot(\timeShiftOperator)^{\indexFreeOne} )$ and $\vectorial{\eigenvectorLetter}_{\textnormal{s}}(\timeShiftOperator)$.
A vanishing Kreiss-Lopatinskii determinant can be therefore interpreted as coming from the orthogonality between these two vectors.
Moreover, $\kreissLopatinskiiScalarProduct(\timeShiftOperator)$ (and so the solution of \eqref{eq:coefficientWithKreissLopDetFinal}) is independent of the choice of eigenvectors $\vectorial{\eigenvectorLetter}_0^1, \dots, \vectorial{\eigenvectorLetter}_0^{\numberVelocities - 1}$.
According to \eqref{eq:decompositionKreissLopatinskiiDet}, considering that $\stableRoot(\timeShiftOperator)$ can be continuously extended to $\unitCircle$, the only obstruction to continuously extend $\kreissLopatinskiiDet(\timeShiftOperator)$--$\kreissLopatinskiiScalarProduct(\timeShiftOperator)$ to $\unitCircle$ comes from  a possible lack of this property as far as  $\vectorial{\eigenvectorLetter}_{\textnormal{s}}(\timeShiftOperator)$ is concerned.
For a given moment (unknown) of the numerical scheme, \Cref{def:strongStability} prescribes that the solution depend continuously on boundary data, which is obtained (\confer{} \eqref{eq:generalStableSolutionAllSchemes}) upon $\coefficientStableSolution(\timeShiftOperator)$ depending continuously--uniformly in $\neighborhoodInfinity$ (thus up to $\unitCircle$) on the boundary datum, as well as from a continuous extension of the relevant component of $\vectorial{\eigenvectorLetter}_{\textnormal{s}}(\timeShiftOperator)$.
The first condition, whenever violated due to zeros of $\kreissLopatinskiiScalarProduct$, yields instabilities on all components of the scheme; whereas the second one, depending only on the bulk scheme, generates instabilities only on some components.
	\begin{remark}[On the assumption of stencil breadth one to the left]
		If the assumption $\stencilLeft = 1$ and/or $\stencilLeftCharacteristic = 1$ in \Cref{prop:noDiracBoundary} is not fulfilled due to higher-breadth stencils and/or repeated velocities, we face several stable roots of the characteristic equation, see \Cref{cor:stencilFDBounds} and \Cref{lemma:Hersh}.
    	It would thus be difficult---if not impossible---to solve this equation analytically, which is also needed to compute the eigenvectors. 
		Additionally, as remarked by \cite{boutin2024stability}, issues linked to the multiplicity of the roots, which might coincide or cross for certain values of $\timeShiftOperator$, must be carefully handled in this setting.
    	Finally, note that the large number of free parameters---whence difficulties in the theoretical understanding of the scheme---and computational cost that comes from 1D schemes with more velocities than a \lbmScheme{1}{3} are actually unnecessary to approximate the solution of a scalar hyperbolic equation  like \eqref{eq:targetEquation} satisfactorily.
	\end{remark}

Let the assumptions of \Cref{prop:noDiracBoundary} be met.
Instabilities arise as ``\strong{resonances}'' between boundary and bulk schemes, hence a necessary condition for them to happen (rigorously proved  in the classification at the end of the section) is that these two schemes share an eigenvalue $(\timeShiftOperator, \fourierShift)$.
To identify them, we insert the ansatz $\laplaceTransformed{\vectorial{\momentDiscrete}}_{\indexSpace} = \vectorial{\varphi}\fourierShift^{\indexSpace}$, where $\vectorial{\varphi}=\vectorial{\varphi}(\timeShiftOperator)\in\complex^{\numberVelocities}$ plays the role of an eigenvector, and $\fourierShift=\fourierShift(\timeShiftOperator)\in\complex$ the role of an eigenvalue in space, into the resolvent equation \eqref{eq:resolventBulk}-\eqref{eq:resolventBoundary} without boundary source term.
This yields two systems
\begin{equation*}
	\matrixPolynomialBulk{\timeShiftOperator}(\fourierShift)\vectorial{\varphi}_{\text{bulk}}(\timeShiftOperator) = \zeroMatrix{\numberVelocities}\qquad \text{and} \qquad \Bigl ( \timeShiftOperator\identityMatrix{\numberVelocities} - \sum_{\indexFreeOne= 0}^{\stencilBoundaryCondition}\schemeMatrixBoundaryByPower{0}{\indexFreeOne}\fourierShift^{\indexFreeOne}\Bigr )\vectorial{\varphi}_{\text{boundary}}(\timeShiftOperator) = \zeroMatrix{\numberVelocities},
\end{equation*}
and potential issues arise from non-trivial eigenvectors $\vectorial{\varphi}_{\text{bulk}}$ and $\vectorial{\varphi}_{\text{boundary}}$.
In general, the subspaces spanned by $\vectorial{\varphi}_{\text{bulk}}$ and $\vectorial{\varphi}_{\text{boundary}}$ are different. 
The worst-case scenario is the one where they overlap.
Whatever case we are in, which shall be analyzed in what follows, this commands the study of the two-unknowns--two-equations system 
\begin{equation}\label{eq:toSolveBoundaryBulkD1Q2}
	\determinant(\matrixPolynomialBulk{\timeShiftOperator}(\fourierShift)) = 0 \qquad \text{and} \qquad \determinant\Bigl ( \timeShiftOperator\identityMatrix{\numberVelocities} - \sum_{\indexFreeOne = 0}^{\stencilBoundaryCondition}\schemeMatrixBoundaryByPower{0}{\indexFreeOne}\fourierShift^{\indexFreeOne}\Bigr ) = 0.
\end{equation}

\begin{definition}[Shared eigenvalue between bulk and boundary scheme]\label{def:eigenvalue}
	Let $(\targetEigenvalue, \targetFourierShift)\in\neighborhoodInfinity\times \unitDisk$ be a solution of \eqref{eq:toSolveBoundaryBulkD1Q2} or $\targetEigenvalue\in\unitCircle$ such that $\lim_{\timeShiftOperator\to\targetEigenvalue}\stableRoot(\timeShiftOperator) = \stableRoot(\targetEigenvalue)=\targetFourierShift$ and $(\targetEigenvalue, \targetFourierShift)\in\unitCircle\times \closedUnitDisk$ be a solution \eqref{eq:toSolveBoundaryBulkD1Q2}.
	We then say that $(\targetEigenvalue, \targetFourierShift)$ is a ``shared eigenvalue between bulk and boundary scheme'' (sometimes just ``shared eigenvalue'' or ``eigenvalue'').
\end{definition}


In general, one is mainly concerned with the analysis of $\targetEigenvalue\in\unitCircle$.
Consider the case where we can divide both sides of \eqref{eq:coefficientWithKreissLopDet} by the polynomial $\scalarFactorFromAdjugate(\timeShiftOperator)$, so that we focus on $\kreissLopatinskiiScalarProduct(\timeShiftOperator)$ instead of $\kreissLopatinskiiDet(\timeShiftOperator)$.
A computation similar to the ones in the proof of \Cref{prop:noDiracBoundary} yields
			\begin{equation*}
				\determinant\Bigl ( \timeShiftOperator\identityMatrix{\numberVelocities} - \sum_{\indexFreeOne = 0}^{\stencilBoundaryCondition}\schemeMatrixBoundaryByPower{0}{\indexFreeOne}\stableRoot(\timeShiftOperator)^{\indexFreeOne}\Bigr ) 
				= \Bigl (  \transpose{\canonicalBasisVector{\positiveVelocityIndex}}\stableRoot(\timeShiftOperator)^{-1} -\sum_{\indexFreeOne = 0}^{\stencilBoundaryCondition}\sum_{\indexFreeTwo = 1}^{\numberVelocities}\transpose{\canonicalBasisVector{\indexFreeTwo}}\boundaryCoefficient_{\positiveVelocityIndex, \indexFreeTwo, 1, \indexFreeOne}\stableRoot(\timeShiftOperator)^{\indexFreeOne} \Bigr ) \momentMatrix^{-1}\collisionMatrix\adjugate(\timeShiftOperator\identityMatrix{\numberVelocities}-\schemeMatrixBulkFourier(\stableRoot(\timeShiftOperator)))\momentMatrix\canonicalBasisVector{\positiveVelocityIndex}.
			\end{equation*}
As $\stableRoot(\timeShiftOperator)$ for $\timeShiftOperator\in\neighborhoodInfinity$ is simple, we deduce that $\rank(\timeShiftOperator\identityMatrix{\numberVelocities}-\schemeMatrixBulkFourier(\stableRoot(\timeShiftOperator))) = \numberVelocities - 1$ for $\timeShiftOperator\in\neighborhoodInfinity$. This entails that
\begin{equation}\label{eq:adjugateAsRankOneMatrix}
	\adjugate(\timeShiftOperator\identityMatrix{\numberVelocities}-\schemeMatrixBulkFourier(\stableRoot(\timeShiftOperator))) = \vectorial{\eigenvectorLetter}_{\textnormal{s}}(\timeShiftOperator)\transpose{\tilde{\vectorial{\eigenvectorLetter}}_{\textnormal{s}}(\timeShiftOperator)},\qquad \timeShiftOperator\in\neighborhoodInfinity,
\end{equation}
where $\kernel(\timeShiftOperator\identityMatrix{\numberVelocities}-\schemeMatrixBulkFourier(\stableRoot(\timeShiftOperator))) = \spanSpace{\vectorial{\eigenvectorLetter}_{\textnormal{s}}(\timeShiftOperator)}$, with the normalization of $\vectorial{\eigenvectorLetter}_{\textnormal{s}}(\timeShiftOperator)$ that has been fixed (we consistently take $\eigenvectorLetter_{\textnormal{s}, 1}(\timeShiftOperator)\equiv 1$ in the paper), and $\kernel(\timeShiftOperator\identityMatrix{\numberVelocities}-\transpose{\schemeMatrixBulkFourier(\stableRoot(\timeShiftOperator))}) = \spanSpace{\tilde{\vectorial{\eigenvectorLetter}}_{\textnormal{s}}(\timeShiftOperator)}$.
Notice that \eqref{eq:adjugateAsRankOneMatrix} imposes some kind of normalization on $\tilde{\vectorial{\eigenvectorLetter}}_{\textnormal{s}}(\timeShiftOperator)$ coming from the fact that the adjugate is a polynomial function of its argument, and the argument on the right-hand side of \eqref{eq:adjugateAsRankOneMatrix} can be continuously extended to $\targetEigenvalue$.
This ensures that, on the one hand, if $\vectorial{\eigenvectorLetter}_{\textnormal{s}}(\timeShiftOperator)$ can be continuously extended to $\targetEigenvalue$ (\idEst{}, no component has a pole), so does $\tilde{\vectorial{\eigenvectorLetter}}_{\textnormal{s}}(\timeShiftOperator)$.
On the other hand, if $\vectorial{\eigenvectorLetter}_{\textnormal{s}}(\timeShiftOperator)$ cannot be continuously extended to $\targetEigenvalue$ (\idEst{}, some component has a pole, whose maximal order we denote $\sigma$), every component of $\tilde{\vectorial{\eigenvectorLetter}}_{\textnormal{s}}(\timeShiftOperator)$ has zeros of order at least $\sigma$. When all zeros are of order strictly larger than $\sigma$, this entails that $\adjugate(\targetEigenvalue\identityMatrix{\numberVelocities}-\schemeMatrixBulkFourier(\stableRoot(\targetEigenvalue)))=\zeroMatrix{\numberVelocities\times\numberVelocities}$, so $\rank(\targetEigenvalue\identityMatrix{\numberVelocities}-\schemeMatrixBulkFourier(\stableRoot(\targetEigenvalue))) < \numberVelocities - 1$, and $\targetEigenvalue$ is a multiple eigenvalue.

\begin{remark}[On the normalization of $\vectorial{\eigenvectorLetter}_{\textnormal{s}}(\timeShiftOperator)$]
	Eigenvectors are defined up to non-zero scale factors.
	For instance, we could argue that the chosen $\vectorial{\eigenvectorLetter}_{\textnormal{s}}(\timeShiftOperator)$ can be renormalized \strong{via} a pre-multiplication by $\alpha(\timeShiftOperator)$.
	The multilinearity of the determinant entails that $\kreissLopatinskiiDet(\timeShiftOperator)$ and $\kreissLopatinskiiScalarProduct(\timeShiftOperator)$ would also be multiplied by $\alpha(\timeShiftOperator)$.
	For this reason, $\coefficientStableSolution(\timeShiftOperator)$ would be multiplied by $\alpha(\timeShiftOperator)^{-1}$ which simplifies with $\alpha(\timeShiftOperator)$ in $\alpha(\timeShiftOperator) \vectorial{\eigenvectorLetter}_{\textnormal{s}}(\timeShiftOperator)$, \confer{} \eqref{eq:generalStableSolutionAllSchemes}. 
	This proves that renormalizations, although they can (linearly) change $\kreissLopatinskiiDet(\timeShiftOperator)$ and $\kreissLopatinskiiScalarProduct(\timeShiftOperator)$, do not modify the outcome on the $\timeShiftOperator$-transformed numerical solution, and thus our eventual conclusions. 
\end{remark}

From \eqref{eq:adjugateAsRankOneMatrix}, we obtain that
\begin{equation}\label{eq:factorKL}
	\determinant\Bigl ( \timeShiftOperator\identityMatrix{\numberVelocities} - \sum_{\indexFreeOne = 0}^{\stencilBoundaryCondition}\schemeMatrixBoundaryByPower{0}{\indexFreeOne}\stableRoot(\timeShiftOperator)^{\indexFreeOne}\Bigr ) 
	= \bigl (\transpose{\tilde{\vectorial{\eigenvectorLetter}}_{\textnormal{s}}(\timeShiftOperator)}\momentMatrix\canonicalBasisVector{\positiveVelocityIndex} \bigr )\kreissLopatinskiiScalarProduct(\timeShiftOperator),\qquad \timeShiftOperator\in\neighborhoodInfinity.
\end{equation}

As instabilities arise either from zeros in the Kreiss-Lopatinskii determinant or from poles in the eigenvector, the following lemma tells us that we have to study shared eigenvalues between bulk and boundary scheme in order to understand the former kind of instability.
\begin{lemma}
	Let $\lim_{\timeShiftOperator\to\targetEigenvalue}\kreissLopatinskiiScalarProduct(\timeShiftOperator) = 0$ (we often write $\kreissLopatinskiiScalarProduct(\targetEigenvalue) = 0$), then $(\targetEigenvalue, \stableRoot(\targetEigenvalue))$ is a shared eigenvalue between bulk and boundary scheme.
\end{lemma}
\begin{proof}
	As $\momentMatrix$ is non-singular, we have that $\momentMatrix\canonicalBasisVector{\positiveVelocityIndex}\neq \zeroMatrix{\numberVelocities}$. Thus $\transpose{\tilde{\vectorial{\eigenvectorLetter}}_{\textnormal{s}}(\timeShiftOperator)} \momentMatrix\canonicalBasisVector{\positiveVelocityIndex}$ cannot have a pole, as $\tilde{\vectorial{\eigenvectorLetter}}_{\textnormal{s}}(\timeShiftOperator)$ does not have poles by the previous discussion. Equation \eqref{eq:factorKL} gives 
	\begin{equation}\label{eq:isEigenvalue}
		\lim_{\timeShiftOperator\to\targetEigenvalue} \determinant\Bigl ( \timeShiftOperator\identityMatrix{\numberVelocities} - \sum_{\indexFreeOne = 0}^{\stencilBoundaryCondition}\schemeMatrixBoundaryByPower{0}{\indexFreeOne}\stableRoot(\timeShiftOperator)^{\indexFreeOne}\Bigr ) = \determinant\Bigl ( \targetEigenvalue\identityMatrix{\numberVelocities} - \sum_{\indexFreeOne = 0}^{\stencilBoundaryCondition}\schemeMatrixBoundaryByPower{0}{\indexFreeOne}\stableRoot(\targetEigenvalue)^{\indexFreeOne}\Bigr ) = 0, 
	\end{equation}
	hence the claim.
\end{proof}

We now classify expected behaviors for $\kreissLopatinskiiScalarProduct(\timeShiftOperator)$ and $\vectorial{\eigenvectorLetter}_{\textnormal{s}}(\timeShiftOperator)$ as $\timeShiftOperator\to\targetEigenvalue$, meanwhile showing that the latter kind of instability---due to poles in the eigenvector---manifests only for shared eigenvalues.
\begin{itemize}
	\item Case \threeboxes{\notContinuousExtensionMark}{}{}. $\stableSubspace(\timeShiftOperator)$ (\idEst{} $\vectorial{\eigenvectorLetter}_{\textnormal{s}}(\timeShiftOperator)$) cannot be continuously extended as $\timeShiftOperator\to\targetEigenvalue$.
	Still, at least one component of $\vectorial{\eigenvectorLetter}_{\textnormal{s}}(\timeShiftOperator)$ can be chosen independent of $\timeShiftOperator$ (hence continuous).
	\begin{itemize}
		\item Case \threeboxes{\notContinuousExtensionMark}{\eigenvalueMark}{}. $\targetEigenvalue$ is a shared eigenvalue between bulk and boundary scheme, \idEst{} \eqref{eq:isEigenvalue} holds.
		\begin{itemize}
			\item Case \threeboxes{\notContinuousExtensionMark}{\eigenvalueMark}{\zeroKL}. $\kreissLopatinskiiScalarProduct(\targetEigenvalue) = 0$. This means that  bulk and boundary scheme also share the eigenvector.
			This entails instability on every component of the system, with some worse than others due to principal parts in some components of $\vectorial{\eigenvectorLetter}_{\textnormal{s}}(\timeShiftOperator)$.
			\item Case \threeboxes{\notContinuousExtensionMark}{\eigenvalueMark}{\nonZeroKL}. $\kreissLopatinskiiScalarProduct(\targetEigenvalue) \neq 0$ and finite.
			This means that bulk and boundary scheme do not share the eigenvector.
			The scheme can be strongly stable on certain components (hopefully, $\momentDiscrete_1$) and is not strongly stable on others, due to principal parts in $\vectorial{\eigenvectorLetter}_{\textnormal{s}}(\timeShiftOperator)$.
			In particular $\coefficientStableSolution(\timeShiftOperator) = (\kreissLopatinskiiScalarProduct(\targetEigenvalue)^{-1} + \bigO{\timeShiftOperator-\targetEigenvalue})\laplaceTransformed{\boundarySourceTerm}_{\positiveVelocityIndex, -1}(\timeShiftOperator)$ depends continuously on the data, so if $\eigenvectorLetter_{\textnormal{s}, \indexVelocity}(\timeShiftOperator)$ can be continuously extended, $\laplaceTransformed{\momentDiscrete}_{\indexVelocity}(\timeShiftOperator)$ depends continuously on the data.
			\item Case \threeboxes{\notContinuousExtensionMark}{\eigenvalueMark}{\infiniteKL}. $\lim_{\timeShiftOperator\to\targetEigenvalue}\kreissLopatinskiiScalarProduct(\timeShiftOperator)=\infty$.
			When the maximum order of the pole in $\vectorial{\eigenvectorLetter}_{\textnormal{s}}(\timeShiftOperator)$ is $\sigma = 1$, the scheme can be strongly stable on every component without further verification on this mode. 
			Indeed, the order of the pole in $\kreissLopatinskiiScalarProduct(\timeShiftOperator)$ is less or equal than $\sigma$, hence exactly $1$.
			This entails by \eqref{eq:coefficientWithKreissLopDetFinal} that $\coefficientStableSolution(\timeShiftOperator) = \bigO{\timeShiftOperator-\targetEigenvalue} \times \laplaceTransformed{\boundarySourceTerm}_{\positiveVelocityIndex, -1}(\timeShiftOperator)$, which back into \eqref{eq:generalStableSolutionAllSchemes} compensates the poles of order $1$ of $\vectorial{\eigenvectorLetter}_{\textnormal{s}}(\timeShiftOperator)$ to yield a continuous dependence of any component of the solution on $\laplaceTransformed{\boundarySourceTerm}_{\positiveVelocityIndex, -1}(\timeShiftOperator)$.
		\end{itemize}
		\item Case \threeboxes{\notContinuousExtensionMark}{\noEigenvalueMark}{}. $\targetEigenvalue$ is not a shared eigenvalue between bulk and boundary scheme:
		\begin{equation}\label{eq:tmpNew}
			\lim_{\timeShiftOperator\to\targetEigenvalue}\determinant \Bigl ( \timeShiftOperator\identityMatrix{\numberVelocities} - \sum_{\indexFreeOne = 0}^{\stencilBoundaryCondition}\schemeMatrixBoundaryByPower{0}{\indexFreeOne}\stableRoot(\timeShiftOperator)^{\indexFreeOne} \Bigr ) = \determinant \Bigl ( \targetEigenvalue\identityMatrix{\numberVelocities} - \sum_{\indexFreeOne = 0}^{\stencilBoundaryCondition}\schemeMatrixBoundaryByPower{0}{\indexFreeOne}\stableRoot(\targetEigenvalue)^{\indexFreeOne} \Bigr )\neq 0.
		\end{equation}
		By the previous discussion, $\transpose{\tilde{\vectorial{\eigenvectorLetter}}_{\textnormal{s}}(\timeShiftOperator)}\momentMatrix\canonicalBasisVector{\positiveVelocityIndex}$ has a zero.
		Therefore, we deduce from \eqref{eq:factorKL} that $\kreissLopatinskiiScalarProduct(\timeShiftOperator)$ must have a pole to ensure \eqref{eq:tmpNew}.
		\begin{itemize}
			\item Case \threeboxes{\notContinuousExtensionMark}{\noEigenvalueMark}{\infiniteKL}. $\lim_{\timeShiftOperator\to\targetEigenvalue}\kreissLopatinskiiScalarProduct(\timeShiftOperator)=\infty$.
			The scheme can be strongly stable on every component without further verification on this mode. 
			Indeed, $\transpose{\tilde{\vectorial{\eigenvectorLetter}}_{\textnormal{s}}(\timeShiftOperator)}\momentMatrix\canonicalBasisVector{\positiveVelocityIndex}$ has a zero of order $\tilde{\sigma}\geq \sigma$.
			As the limit by \eqref{eq:tmpNew} exists and is different from zero, \eqref{eq:factorKL} gives that $\kreissLopatinskiiScalarProduct(\timeShiftOperator)$ has a pole of order $\tilde{\sigma}$.
			This entails that $\kreissLopatinskiiScalarProduct(\timeShiftOperator)^{-1}$ has a zero of order $\tilde{\sigma}\geq \sigma$, so by \eqref{eq:coefficientWithKreissLopDetFinal} that $\coefficientStableSolution(\timeShiftOperator) = \bigO{(\timeShiftOperator-\targetEigenvalue)^{\tilde{\sigma}}} \times \laplaceTransformed{\boundarySourceTerm}_{\positiveVelocityIndex, -1}(\timeShiftOperator)$, which back into \eqref{eq:generalStableSolutionAllSchemes} compensates the poles of order at most $\sigma$ of $\vectorial{\eigenvectorLetter}_{\textnormal{s}}(\timeShiftOperator)$ to yield a continuous dependence of any component of the solution on $\laplaceTransformed{\boundarySourceTerm}_{\positiveVelocityIndex, -1}(\timeShiftOperator)$.
		\end{itemize}
	\end{itemize}
	\item Case \threeboxes{\continuousExtensionMark}{}{}. $\stableSubspace(\timeShiftOperator)$ (\idEst{} $\vectorial{\eigenvectorLetter}_{\textnormal{s}}(\timeShiftOperator)$) can be continuously extended as $\timeShiftOperator\to\targetEigenvalue$ (so also $\kreissLopatinskiiScalarProduct(\timeShiftOperator)$).

	\begin{itemize}
		\item Case \threeboxes{\continuousExtensionMark}{\eigenvalueMark}{}. $\targetEigenvalue$ is a shared eigenvalue between bulk and boundary scheme, \idEst{} \eqref{eq:isEigenvalue} holds.
		From \eqref{eq:factorKL}, we deduce that either $\kreissLopatinskiiScalarProduct(\targetEigenvalue) = 0$ or $\transpose{\tilde{\vectorial{\eigenvectorLetter}}_{\textnormal{s}}(\targetEigenvalue)}\momentMatrix\canonicalBasisVector{\positiveVelocityIndex} = 0$.
		\begin{itemize}
			\item Case \threeboxes{\continuousExtensionMark}{\eigenvalueMark}{\zeroKL}. $\kreissLopatinskiiScalarProduct(\targetEigenvalue) = 0$. This means that  bulk and boundary scheme also share the eigenvector.
			This entails instability on every component of the system, all with the same degree of severity as far as this mode is concerned. 
			\item Case \threeboxes{\continuousExtensionMark}{\eigenvalueMark}{\nonZeroKL}. $\kreissLopatinskiiScalarProduct(\targetEigenvalue) \neq 0$ and finite. 
			We could neither find an example of this situation, nor prove that it cannot happen. 
			In this case, $\transpose{\tilde{\vectorial{\eigenvectorLetter}}_{\textnormal{s}}(\targetEigenvalue)}\momentMatrix\canonicalBasisVector{\positiveVelocityIndex} = 0$, so $\tilde{\vectorial{\eigenvectorLetter}}_{\textnormal{s}}(\targetEigenvalue)\in\kernel(\targetEigenvalue\identityMatrix{\numberVelocities} - \sum_{\indexFreeOne = 0}^{\stencilRight}\transpose{\schemeMatrixBulkByPower_{\indexFreeOne}}\stableRoot(\targetEigenvalue)^{\indexFreeOne})$ in addition to $\tilde{\vectorial{\eigenvectorLetter}}_{\textnormal{s}}(\targetEigenvalue)\in\kernel(\targetEigenvalue\identityMatrix{\numberVelocities} - \sum_{\indexFreeOne = -1}^{\stencilRight}\transpose{\schemeMatrixBulkByPower_{\indexFreeOne}}\stableRoot(\targetEigenvalue)^{\indexFreeOne})$, which is very unlike to happen.
			Still, these modes are harmless since they do not include any source of instability, either through a vanishing Kreiss-Lopatinskii determinant or poles in the eigenvector.
		\end{itemize}
		\item Case \threeboxes{\continuousExtensionMark}{\noEigenvalueMark}{}. $\targetEigenvalue$ is not a shared eigenvalue between bulk and boundary scheme.
		Both $\kreissLopatinskiiScalarProduct(\targetEigenvalue) \neq 0$ and $\transpose{\tilde{\vectorial{\eigenvectorLetter}}_{\textnormal{s}}(\targetEigenvalue)}\momentMatrix\canonicalBasisVector{\positiveVelocityIndex} \neq 0$ and finite.
		\begin{itemize}
			\item Case \threeboxes{\continuousExtensionMark}{\noEigenvalueMark}{\nonZeroKL}. $\kreissLopatinskiiScalarProduct(\targetEigenvalue) \neq 0$ and finite.
			These modes are harmless as they do not trigger one of the two sources of instability. 
			We do not encounter them in the future as they do not appear while searching shared eigenvalues or looking for poles in the eigenvector.
		\end{itemize}
	\end{itemize}
\end{itemize}
This discussion makes it clear that we have to \strong{concentrate only on modes which are eigenvalues}.
Examples of these situations, recalled with the \threeboxes{}{}{} notation, are in \Cref{sec:strongCertainSchemes}.

\section{Strong stability of some schemes and description of the instabilities}\label{sec:strongCertainSchemes}

We now consider a quite complete set of examples of schemes for \eqref{eq:targetEquation}.
We start by introducing a notation: in all examples, we have (outside exceptional cases) $\stencilLeftCharacteristic = \stencilRightCharacteristic = 1$, hence \eqref{eq:characteristicEquation} is a quadratic equation in $\fourierShift$, which reads (for $\fourierShift\in\puncturedPlane$) 
\begin{equation*}
	\coefficientCharEquationInFourier_1(\timeShiftOperator)\fourierShift^2 + \coefficientCharEquationInFourier_0(\timeShiftOperator)\fourierShift + \coefficientCharEquationInFourier_{-1}(\timeShiftOperator) = 0.
\end{equation*}
The product of the two roots equals $\productRootsDOneQThree(\timeShiftOperator)\definitionEquality \frac{\coefficientCharEquationInFourier_{-1}(\timeShiftOperator)}{\coefficientCharEquationInFourier_{1}(\timeShiftOperator)}$.
As we need actual parametrizations of these roots as functions of $\timeShiftOperator$, we introduce the notation 
\begin{equation}\label{eq:rootSecondOrder}
	\fourierShift_{\pm}(\timeShiftOperator)\definitionEquality \frac{- \coefficientCharEquationInFourier_0(\timeShiftOperator)\pm\sqrt{\coefficientCharEquationInFourier_0(\timeShiftOperator)^2 - 4\coefficientCharEquationInFourier_1(\timeShiftOperator)\coefficientCharEquationInFourier_{-1}(\timeShiftOperator)} }{2 \coefficientCharEquationInFourier_1(\timeShiftOperator)},
\end{equation}
where $\sqrt{\cdot}$ is the principal square root.
In what follows, many computations are cumbersome: we therefore take advantage of the computer algebra system \texttt{SageMath 10.0} to perform them.
For the sake of reproducibility, codes are available at \href{https://github.com/thomasbellotti/GKS-LBM}{https://github.com/thomasbellotti/GKS-LBM}.

\subsection{\lbmScheme{1}{2} scheme}\label{sec:D1Q2}

We consider the scheme by \cite{graille2014approximation} with the following components:
\begin{equation*}
	\numberVelocities = 2, \qquad \dimensionlessDiscreteVelocityLetter_1 = -\dimensionlessDiscreteVelocityLetter_2 = 1, \qquad \momentMatrix = 
	\begin{pmatrix}
		1 & 1 \\
		1 & - 1
	\end{pmatrix}, \qquad 
	\equilibriumVectorLetter_2 = \courantNumber,
\end{equation*}
where we recall that $\courantNumber$ is the Courant number.
One can show that for $\relaxationParameterLetter_2\in(0, 2)$, the scheme is first-order accurate, whereas if $\relaxationParameterLetter_2 = 2$, it is second-order accurate.
We do not consider $\relaxationParameterLetter_2 = 0$, as in this case the scheme is not consistent with \eqref{eq:targetEquation}.
As far as stability is concerned, we have the following.
\begin{proposition}[Stability of the \lbmScheme{1}{2} scheme for the Cauchy problem]\label{prop:stabD1Q2}
	Assume that $\relaxationParameterLetter_2\in(0, 2]$.
	The \lbmScheme{1}{2} scheme is stable according to \emph{von Neumann}, \confer{} \Cref{def:vonNeumannStability}, if and only if
	\begin{equation}\label{eq:D1Q2VonNeumann}
		|\courantNumber|\leq 1.
	\end{equation}
	Moreover, stability---both for \Cref{def:stabLBM} and \ref{def:stabFD}---holds if and only if 
	\begin{align}
		\text{when}\quad\relaxationParameterLetter_2\in (0, 2), \qquad \text{then} \qquad &|\courantNumber|\leq 1,\label{eq:sCondD1Q2-1} \\
		\text{or when}\quad \relaxationParameterLetter_2=2, \qquad \text{then} \qquad &|\courantNumber|< 1.\label{eq:sCondD1Q2-2}
	\end{align}
\end{proposition}
\begin{proof}[Proof of \Cref{prop:stabD1Q2}]
	\cite{bellotti2024initialisation} gives necessary and sufficient conditions for \Cref{def:vonNeumannStability} and \ref{def:stabFD} to be met.
	The latter entails that \Cref{def:stabLBM} holds as well thanks to \Cref{prop:stableFDDoncStableLBM}.
	We are left to check the necessary character of \eqref{eq:sCondD1Q2-2} to have \Cref{def:stabLBM}, which by virtue of \Cref{prop:necConditions} has to be checked only when $\relaxationParameterLetter_2 = 2$ and $|\courantNumber|=1$.
	In these cases, we have that for $\fourierShift=\pm i$, then $\timeShiftOperator = \mp \sign(\courantNumber) i $ is a double eigenvalue of $\schemeMatrixBulkFourier(\pm 1)$.
	This is unsurprising since when $\relaxationParameterLetter_2 = 2$, the characteristic equation \eqref{eq:characteristicEquation} coincides with the one of a leap-frog scheme for the transport equation, which features the same issue (\confer{}, \cite[Chapter 4]{strikwerda2004finite}).
	Moreover, $\dimension{\kernel(\mp \sign(\courantNumber) i \identityMatrix{2} -\schemeMatrixBulkFourier(\pm i))} = 1$, hence the scheme is not stable according to \Cref{def:stabLBM}.
\end{proof}

We consider the following boundary conditions set on $\distributionFunctionDiscrete_{1, -1}^{\indexTime\collided}$ under the form by \eqref{eq:boundaryConditions}:
\begin{align}
	\distributionFunctionDiscrete_{1, -1}^{\indexTime\collided} = \distributionFunctionDiscrete_{2, 0}^{\indexTime\collided} + \boundarySourceTerm_{1, -1}^{\indexTime}&\qquad \text{(bounce-back)}, \label{eq:BB}\\
	\distributionFunctionDiscrete_{1, -1}^{\indexTime\collided} = -\distributionFunctionDiscrete_{2, 0}^{\indexTime\collided} + \boundarySourceTerm_{1, -1}^{\indexTime}&\qquad \text{(anti-bounce-back)}, \label{eq:ABB}\\
	\distributionFunctionDiscrete_{1, -1}^{\indexTime\collided} = -\distributionFunctionDiscrete_{2, 1}^{\indexTime\collided} + \boundarySourceTerm_{1, -1}^{\indexTime}&\qquad \text{(two-steps anti-bounce-back)}, \label{eq:TwoABB}\\
	\distributionFunctionDiscrete_{1, -1}^{\indexTime\collided} = \sum_{\indexSpace=0}^{\orderExtrapolation-1}(-1)^{\indexSpace}\binom{\orderExtrapolation}{\indexSpace + 1} \distributionFunctionDiscrete_{1, \indexSpace}^{\indexTime\collided} + \boundarySourceTerm_{1, -1}^{\indexTime}&\qquad \text{(extrapolation of order }\orderExtrapolation\geq 1\text{)}, \label{eq:extrapolation} \\
	\distributionFunctionDiscrete_{1, -1}^{\indexTime\collided} = \boundarySourceTerm_{1, -1}^{\indexTime}&\qquad \text{(kinetic Dirichlet)}, \label{eq:kineticDirichlet}\\
	\distributionFunctionDiscrete_{1, -1}^{\indexTime\collided} = \tfrac{1}{2}(1+\courantNumber)\sum_{\indexSpace=0}^{\orderExtrapolation-1}(-1)^{\indexSpace}\binom{\orderExtrapolation}{\indexSpace + 1} (\distributionFunctionDiscrete_{1, \indexSpace}^{\indexTime\collided}+\distributionFunctionDiscrete_{2, \indexSpace}^{\indexTime\collided}) + \boundarySourceTerm_{1, -1}^{\indexTime}&\qquad \text{(extrapolated equil. of order }\orderExtrapolation\geq 1\text{)}, \label{eq:extrapolatedEquilibrium}\\
	\distributionFunctionDiscrete_{1, -1}^{\indexTime\collided} = \tfrac{1}{2}(1+\courantNumber)(\distributionFunctionDiscrete_{1, 0}^{\indexTime\collided}+\distributionFunctionDiscrete_{2, 2}^{\indexTime\collided}) + \boundarySourceTerm_{1, -1}^{\indexTime} &\qquad\text{(future)},\label{eq:kinrod}\\
	\distributionFunctionDiscrete_{1, -1}^{\indexTime\collided} = \distributionFunctionDiscrete_{1, 0}^{\indexTime\collided} + 2\distributionFunctionDiscrete_{2, 0}^{\indexTime\collided} + \boundarySourceTerm_{1, -1}^{\indexTime} &\qquad\text{(invented)}.\label{eq:godunovRyabenkii}
\end{align}
These conditions need further discussion in order to clarify their origin and role.
Condition \eqref{eq:BB} is generally employed to enforce no-slip boundary conditions when dealing with systems, and not scalar equations.
Equation \eqref{eq:ABB} is classically used to enforce inflow boundary conditions.
Condition \eqref{eq:TwoABB} has been introduced in \cite[Chapter 6]{helie2023schema} to enforce inflow boundary conditions and comes from the fact that, in the homogeneous case, $\momentDiscrete_{1, 0}^{\indexTime+1} = \distributionFunctionDiscrete_{1, 0}^{\indexTime + 1} + \distributionFunctionDiscrete_{2, 0}^{\indexTime + 1} = \distributionFunctionDiscrete_{1, -1}^{\indexTime\collided} + \distributionFunctionDiscrete_{2, 1}^{\indexTime \collided} = 0$, taking into account the transport phase.
Condition \eqref{eq:extrapolation} extrapolates the missing distribution function with the formul\ae{} by \cite{goldberg1977boundary}, and can be employed with outflows \cite{bellotti:hal-04630735}.
Equation \eqref{eq:kineticDirichlet} boils down to setting the missing information to a value independent of the solution inside the computational domain, and has been analyzed in \cite{aregba2025equilibrium} both in the inflow and outflow setting.
Condition \eqref{eq:extrapolatedEquilibrium}, also discussed in \cite{aregba2025equilibrium} for outflows, boils down to extrapolating the value of $\momentDiscrete_{1, -1}^{\indexTime\collided}$ using the formul\ae{} by \cite{goldberg1977boundary}, and then taking the corresponding equilibrium $\distributionFunctionLetter_1^{\atEquilibrium}(\momentLetter_1) = \tfrac{1}{2}(1+\courantNumber)\momentLetter_1$.
Equation \eqref{eq:kinrod} has been tested in \cite{bellotti:hal-04630735} for outflows and comes from setting $\distributionFunctionDiscrete_{1, -1}^{\indexTime\collided} = \distributionFunctionLetter_1^{\atEquilibrium}(\momentDiscrete_{1, 1}^{\indexTime+1})$ and take the transport into account.
Finally, \eqref{eq:godunovRyabenkii} is considered for pure mathematical investigation to generate serious instabilities.

To foster the recapitulation of the results to come, additionally to SS (\eqref{eq:strongStability} holds) and SSOO (\eqref{eq:strongStabilityObserved} but not \eqref{eq:strongStability} holds), we use the following shorthands.
\begin{itemize}
	\item MU-L/E: mildly-unstable (critical eigenvalue $\timeShiftOperator\in\unitCircle$), L for instability localized in space (critical eigenvalue $\fourierShift\in\unitDisk$); E  for instability extended in space (critical eigenvalue $\fourierShift\in\unitCircle$). Neither \eqref{eq:strongStability} nor \eqref{eq:strongStabilityObserved} hold. 
	\item GR-L: Godunov-Ryabenkii instability (critical eigenvalue $\timeShiftOperator\in\neighborhoodInfinity$), instability localized in space (critical eigenvalue $\fourierShift\in\unitDisk$). Neither \eqref{eq:strongStability} nor \eqref{eq:strongStabilityObserved} hold. 
\end{itemize}

\begin{theorem}[Strong stability--instability of the boundary-value \lbmScheme{1}{2} scheme]\label{thm:stabilityBoundaryD1Q2}
	Under the stability conditions \eqref{eq:sCondD1Q2-1}-\eqref{eq:sCondD1Q2-2}, the stability of the boundary-value \lbmScheme{1}{2} scheme according to \Cref{def:strongStability} is as follows.
	\begin{center}\renewcommand{\arraystretch}{1.3}
		\begin{tabular}{|c||c|c|c|}
			\cline{2-4}
			\multicolumn{1}{c||}{} & \multicolumn{2}{c|}{$\courantNumber<0$ (outflow)} & $\courantNumber>0$ (inflow)\\
			\cline{2-4}
			\multicolumn{1}{c||}{} & $\relaxationParameterLetter_2\in(0, 2)$ & $\relaxationParameterLetter_2 = 2$ & $\relaxationParameterLetter_2\in(0, 2]$ \\
			\hline
			\hline
			Bounce-back \eqref{eq:BB} & MU-L & MU-E & SS \\
			Anti-bounce-back \eqref{eq:ABB} & MU-L & MU-E & SS \\
			Two-steps anti-bounce-back \eqref{eq:TwoABB} & SS & SSOO & SS \\
			Extrapolation $\orderExtrapolation = 1$ \eqref{eq:extrapolation} & SS & SSOO & MU-E \\
			Extrapolation $2\leq \orderExtrapolation \leq 4$ \eqref{eq:extrapolation} & SS & MU-E & MU-E \\
			Kinetic Dirichlet \eqref{eq:kineticDirichlet} & SS & SS & SS \\
			\hline
		\end{tabular}
	\end{center}
	For boundary conditions which we  study only for specific values of the parameters $\relaxationParameterLetter_2$ and $\courantNumber$, we have:
	\begin{center}\renewcommand{\arraystretch}{1.3}
		\begin{tabular}{|c||c|c|c|}
			\cline{2-4}
			\multicolumn{1}{c||}{} & \multicolumn{2}{c|}{$\courantNumber=-\tfrac{1}{2}$ (outflow)} & $\courantNumber=\tfrac{1}{2}$ (inflow)\\
			\cline{2-4}
			\multicolumn{1}{c||}{} & $\relaxationParameterLetter_2=\tfrac{3}{2}$ & $\relaxationParameterLetter_2 = 2$ & $\relaxationParameterLetter_2=\tfrac{3}{2}$ or $\relaxationParameterLetter_2=2$ \\
			\hline
			\hline
			Extrapolated equilibrium $\orderExtrapolation = 1$ \eqref{eq:extrapolatedEquilibrium} & SS & SS & MU-E \\
			Extrapolated equilibrium $\orderExtrapolation = 3$ \eqref{eq:extrapolatedEquilibrium} & SS & GR-L & GR-L \\
			Future \eqref{eq:kinrod} & SS& SS & MU-E \\
			Invented \eqref{eq:godunovRyabenkii} & GR-L & GR-L & GR-L \\
			\hline
		\end{tabular}
	\end{center}
\end{theorem}

Let us comment on the results of \Cref{thm:stabilityBoundaryD1Q2} starting from \eqref{eq:extrapolation} with $\orderExtrapolation = 1$ when $\relaxationParameterLetter_2 = 2$ and $\courantNumber<0$, which is an essential reason why this paper has been written down.
On the one hand, in \cite{bellotti:hal-04630735}, we proved that this boundary condition is strongly stable for the scheme recast only on $\momentDiscrete_1$.
On the other hand, in \cite[Chapter 12]{bellotti2023numerical}, instabilities on $\momentDiscrete_2$ where showcased for $\courantNumber=-\tfrac{1}{2}$.
Therefore, the claim of \Cref{thm:stabilityBoundaryD1Q2} is coherent with these findings.
Secondly, we claim that for \eqref{eq:kineticDirichlet}, we are not able to recast the scheme only on  $\momentDiscrete_1$, so that the strategy by \cite{bellotti:hal-04630735} cannot be easily applied to this setting, and makes the interest of the present ``raw'' approach.
This boundary condition is also very peculiar since it yields a sort of \lbm{} analogue of the \strong{Goldberg-Tadmor} lemma \cite{goldberg1981scheme, coulombel2011semigroup}, \idEst{} a boundary condition which remains stable regardless of the fact that it is used at an inflow or at an outflow.
Finally, \Cref{thm:stabilityBoundaryD1Q2} shows why---forgetting about the question of consistency, which is not addressed in the present paper---certain boundary conditions are used only to enforce inflows and others for outflows.

Finally, we have not considered \eqref{eq:extrapolatedEquilibrium}, \eqref{eq:kinrod}, and \eqref{eq:godunovRyabenkii} for generic values of $\relaxationParameterLetter_2$ and $\courantNumber$ due to  difficulties in finding solutions to \eqref{eq:toSolveBoundaryBulkD1Q2} as functions of these two parameters.
We instead study stability for representative values of $\relaxationParameterLetter_2$ and $\courantNumber$, also used in numerical simulations, with the help of \texttt{SageMath}, \confer{} details in \Cref{app:moreBCD1Q2}.

\subsubsection{Numerical simulations}\label{sec:numericalSimD1Q2}

We simulate using a domain with 100 points and at the right endpoint, we employ a homogeneous kinetic Dirichlet boundary condition on $\distributionFunctionDiscrete_2$.
We show the solution at a time equal to one, with zero initial data and boundary source term $\boundarySourceTerm_{1, -1}^{\indexTime} = \delta_{0\indexTime}$, so that $\laplaceTransformed{\boundarySourceTerm}_{1, -1}(\timeShiftOperator) = 1$.

\begin{figure}
	\begin{center}
		\includegraphics[width=1\textwidth]{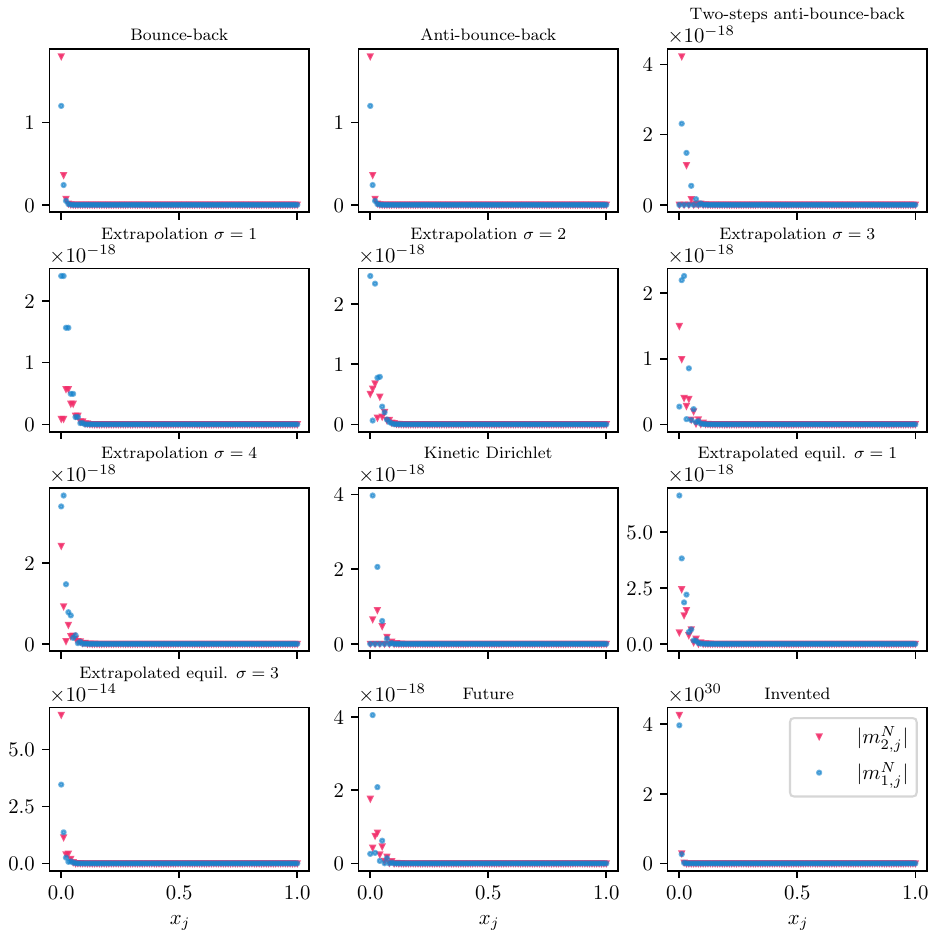}
	\end{center}\caption{\label{fig:D1Q2-s-3_2-C--1_2-revision}Solution at final time for the \lbmScheme{1}{2} scheme under $\relaxationParameterLetter_2 = \tfrac{3}{2}$ and $\courantNumber = -\tfrac{1}{2}$.}
\end{figure}

\begin{figure}
	\begin{center}
		\includegraphics[width=1\textwidth]{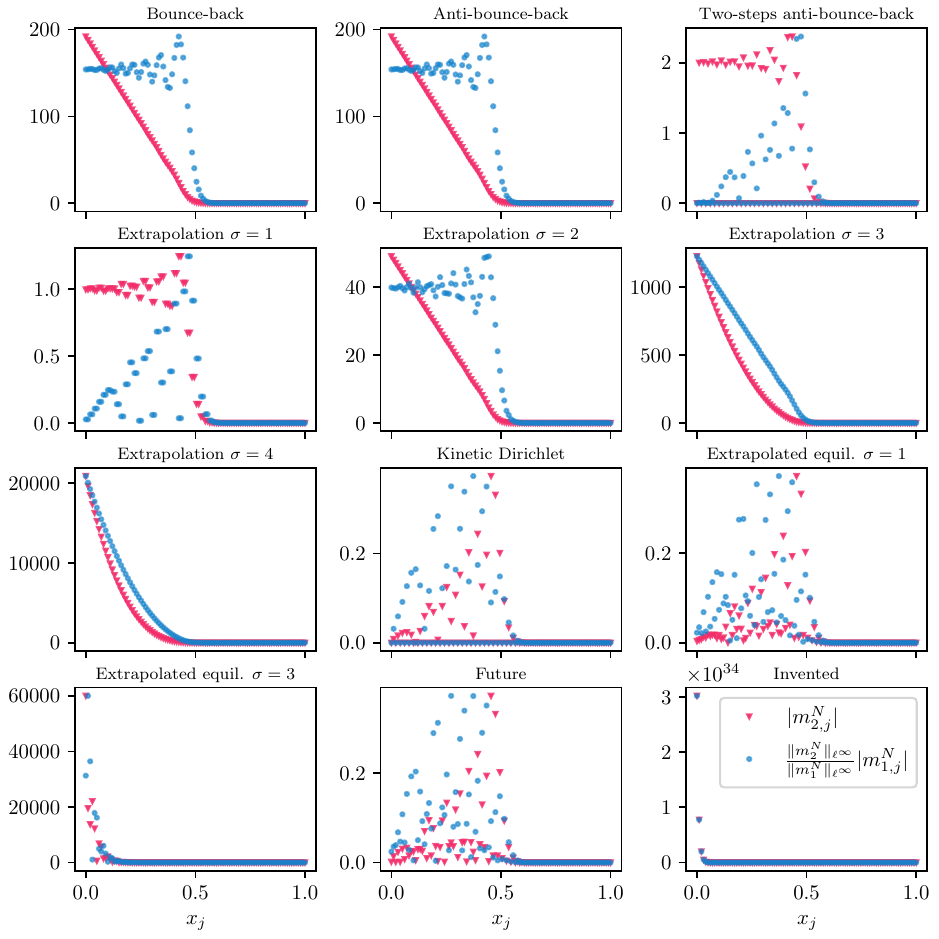}
	\end{center}\caption{\label{fig:D1Q2-s-2-C--1_2-revision}Solution at final time for the \lbmScheme{1}{2} scheme under $\relaxationParameterLetter_2 = 2$ and $\courantNumber = -\tfrac{1}{2}$.}
\end{figure}

\begin{figure}
	\begin{center}
		\includegraphics[width=1\textwidth]{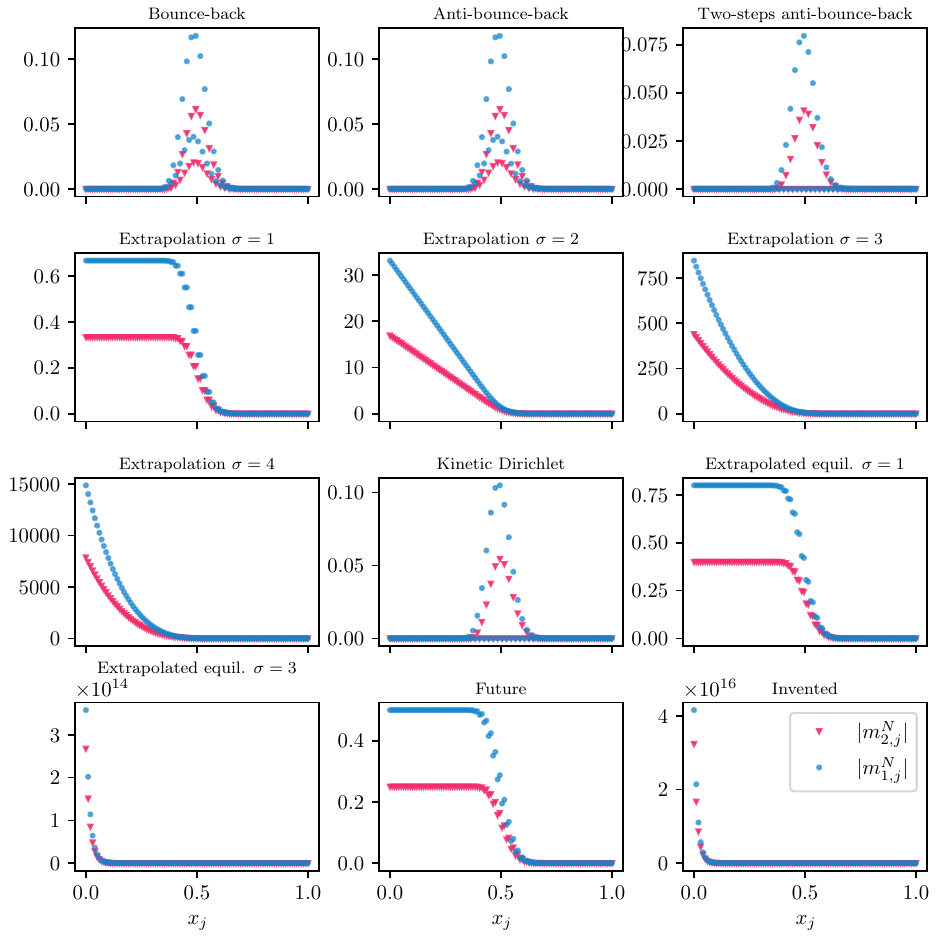}
	\end{center}\caption{\label{fig:D1Q2-s-3_2-C-1_2-revision}Solution at final time for the \lbmScheme{1}{2} scheme under $\relaxationParameterLetter_2 = \tfrac{3}{2}$ and $\courantNumber = \tfrac{1}{2}$.}
\end{figure}

\begin{figure}
	\begin{center}
		\includegraphics[width=1\textwidth]{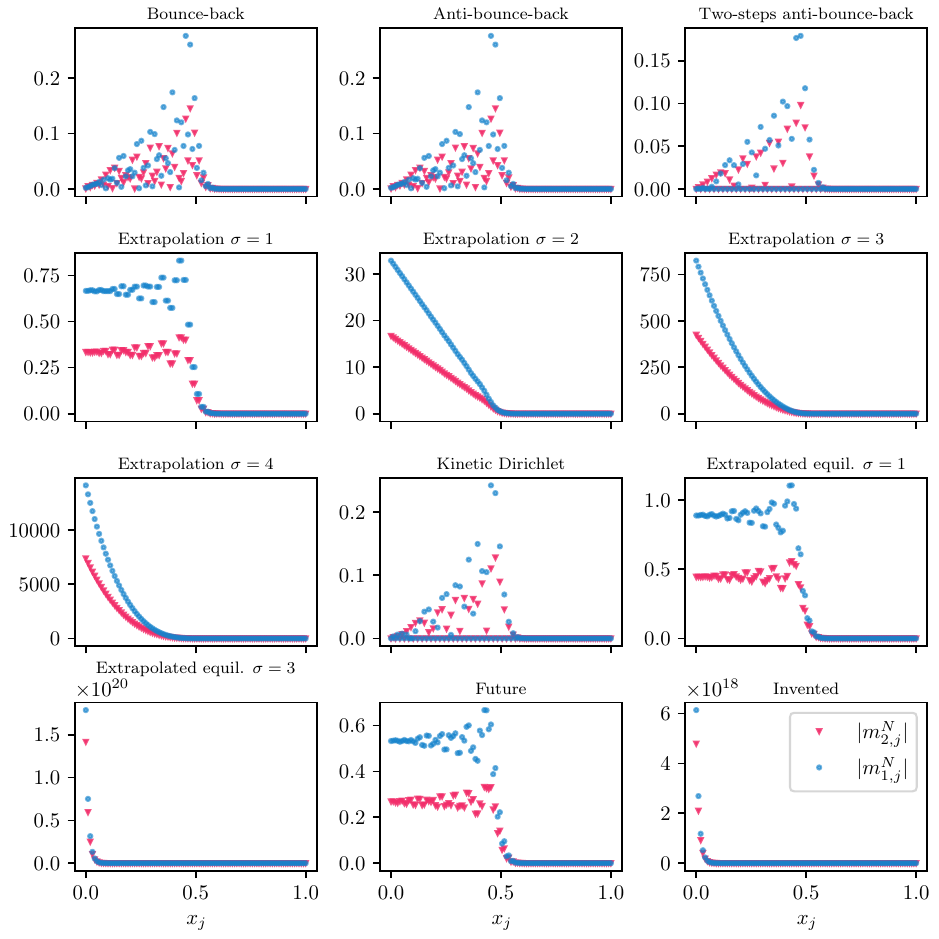}
	\end{center}\caption{\label{fig:D1Q2-s-2-C-1_2-revision}Solution at final time for the \lbmScheme{1}{2} scheme under $\relaxationParameterLetter_2 = 2$ and $\courantNumber = \tfrac{1}{2}$.}
\end{figure}

\begin{itemize}
	\item Let us start considering $\relaxationParameterLetter_2 = \tfrac{3}{2}$ and $\courantNumber = -\tfrac{1}{2}$, whence the boundary is an outflow and the scheme is dissipative.
	The results are in \Cref{fig:D1Q2-s-3_2-C--1_2-revision}.
	\item Now take $\relaxationParameterLetter_2 = 2$ and $\courantNumber = -\tfrac{1}{2}$, so that the boundary is an outflow and the scheme is non-dissipative.
	In the very same setting as the previous case, we observe the results of \Cref{fig:D1Q2-s-2-C--1_2-revision}, where the moment $\momentLetter_1$ has been renormalized by $\momentLetter_2$ to display its spatial structure, despite quite different orders-of-magnitude explained in \Cref{sec:proofStabD1Q2}.
	\item Then, we take $\relaxationParameterLetter_2 = \tfrac{3}{2}$ and $\courantNumber = \tfrac{1}{2}$, so that the boundary is an inflow and the scheme is dissipative.
	Results are shown in \Cref{fig:D1Q2-s-3_2-C-1_2-revision}.
	\item Finally, take $\relaxationParameterLetter_2 = 2$ and $\courantNumber = \tfrac{1}{2}$, so that the boundary is an inflow and the scheme is non-dissipative.
	Results are in \Cref{fig:D1Q2-s-2-C-1_2-revision}.
\end{itemize}
All these results are perfectly coherent with \Cref{thm:stabilityBoundaryD1Q2}, in a way that we detail on the way in \Cref{sec:proofStabD1Q2}, \confer{} \Cref{rem:remarkStrongStab}.

\subsubsection{Proof of \Cref{thm:stabilityBoundaryD1Q2} and comments on the numerical simulations}\label{sec:proofStabD1Q2}

\paragraph{Description of the roots of the characteristic equation}

The characteristic equation \eqref{eq:characteristicEquation} features $\stencilLeftCharacteristic = \stencilRightCharacteristic = 1$ (except in cases detailed below) with 
\begin{equation}
	\coefficientCharEquationInFourier_{-1}(\timeShiftOperator) = -\tfrac{1}{2}(2 + (\courantNumber-1)\relaxationParameterLetter_2) \timeShiftOperator, \qquad \coefficientCharEquationInFourier_{0}(\timeShiftOperator) = \timeShiftOperator^2 + 1-\relaxationParameterLetter_2, \qquad \coefficientCharEquationInFourier_{1}(\timeShiftOperator) = -\tfrac{1}{2}(2 - (\courantNumber+1)\relaxationParameterLetter_2) \timeShiftOperator.
\end{equation}
By \Cref{lemma:Hersh}, except when either $2 + (\courantNumber-1)\relaxationParameterLetter_2 = 0$ or $2 - (\courantNumber+1)\relaxationParameterLetter_2 = 0$, \eqref{eq:characteristicEquation} has one stable root $\stableRoot(\timeShiftOperator)$ and one unstable root $\unstableRoot(\timeShiftOperator)$ for $\timeShiftOperator\in\neighborhoodInfinity$.
They can be extended to $\timeShiftOperator\in\unitCircle$.
We can say that when $2 + (\courantNumber-1)\relaxationParameterLetter_2 = 0$, $\stableRoot(\timeShiftOperator)\equiv 0$ and only $\unstableRoot(\timeShiftOperator)$ is non-trivial.
On the other hand, when $2 - (\courantNumber+1)\relaxationParameterLetter_2 = 0$, we have $\unstableRoot(\timeShiftOperator) \equiv \infty$ and only $\stableRoot(\timeShiftOperator)$ is non-trivial.
The proof of the following result is provided in \cite{bellotti:hal-04630735}.
\begin{proposition}\label{prop:hershD1Q2}
	Under \eqref{eq:D1Q2VonNeumann}, for $\timeShiftOperator\in\neighborhoodInfinity$, the characteristic equation \eqref{eq:characteristicEquation} for the \lbmScheme{1}{2} scheme has two roots, one stable root $\stableRoot(\timeShiftOperator)\in\unitDisk$ (identically equal to $0$ when $2 + (\courantNumber-1)\relaxationParameterLetter_2 = 0$) and one unstable root $\unstableRoot(\timeShiftOperator)\in\neighborhoodInfinity$ (identically equal to $\infty$ when $2 - (\courantNumber+1)\relaxationParameterLetter_2 = 0$). 
	They can be extended to the unit circle, \idEst{} to $\timeShiftOperator\in\unitCircle$, and this extension is still denoted $\stableRoot(\timeShiftOperator)$ and $\unstableRoot(\timeShiftOperator)$.
	These extensions satisfy:
	\begin{equation*}
		\stableRoot(\pm 1) =
		\begin{cases}
			\pm \productRootsDOneQThree(\relaxationParameterLetter_2,\courantNumber), \qquad &\text{if}\quad \courantNumber<0, \\
			\pm 1, \qquad &\text{if}\quad \courantNumber>0.
		\end{cases}\qquad \qquad
		\unstableRoot(\pm 1) =
		\begin{cases}
			\pm 1, \qquad &\text{if}\quad \courantNumber<0, \\
			\pm \productRootsDOneQThree(\relaxationParameterLetter_2,\courantNumber), \qquad &\text{if}\quad \courantNumber>0,
		\end{cases}
	\end{equation*}
	where $\productRootsDOneQThree = \productRootsDOneQThree(\relaxationParameterLetter_2,\courantNumber) = \frac{2 + (\courantNumber-1)\relaxationParameterLetter_2}{2 - (\courantNumber+1)\relaxationParameterLetter_2 }$.
\end{proposition}

\paragraph{Shared eigenvalues}

Computations on a computer algebra system give the following solutions of \eqref{eq:toSolveBoundaryBulkD1Q2}, where we do not list $(\timeShiftOperator, \fourierShift)$ with $\timeShiftOperator = 0$, harmless.
\begin{align*}
	\text{For }\eqref{eq:BB}\qquad &(\timeShiftOperator, \fourierShift) =  (1, \productRootsDOneQThree(\relaxationParameterLetter_2,\courantNumber)), \quad (\timeShiftOperator, \fourierShift) = (\relaxationParameterLetter_2-1, -1), \quad (\timeShiftOperator, \fourierShift) = (\tfrac{1}{2}(1-\courantNumber)\relaxationParameterLetter_2, 0), \\
	\text{For }\eqref{eq:ABB}\qquad &(\timeShiftOperator, \fourierShift) =  (-1, -\productRootsDOneQThree(\relaxationParameterLetter_2,\courantNumber)), \quad (\timeShiftOperator, \fourierShift) = (1-\relaxationParameterLetter_2, 1),  \quad (\timeShiftOperator, \fourierShift) = (\tfrac{1}{2}(\courantNumber-1)\relaxationParameterLetter_2, 0), \\
	\text{For }\eqref{eq:TwoABB}\qquad &(\timeShiftOperator, \fourierShift) =  (\relaxationParameterLetter_2-1, -1), \quad (\timeShiftOperator, \fourierShift) = (1-\relaxationParameterLetter_2, 1),\\
	\text{For }\eqref{eq:extrapolation}\text{, }\orderExtrapolation = 1, 2, 3, 4\qquad &(\timeShiftOperator, \fourierShift) =  (1, 1), \quad (\timeShiftOperator, \fourierShift) = (1-\relaxationParameterLetter_2, 1), \quad (\timeShiftOperator, \fourierShift) = (\sigma (\tfrac{1}{2}(\courantNumber-1)\relaxationParameterLetter_2 + 1), 0), \\
	\text{For }\eqref{eq:kineticDirichlet}\qquad &\text{no }(\timeShiftOperator, \fourierShift).
\end{align*}
\begin{example}[On the help from symbolic computations]
	We used the help of \texttt{SageMath} to solve \eqref{eq:toSolveBoundaryBulkD1Q2}.
	However, this is convenient but not strictly needed---and computations can be done by hand. 
	As example, consider \eqref{eq:toSolveBoundaryBulkD1Q2} for \eqref{eq:TwoABB}:
	\begin{equation*}
		\begin{cases}
			\frac{1}{2}  {\left({\courantNumber} {\fourierShift}^{2} {\relaxationParameterLetter_2} {\timeShiftOperator} + {\fourierShift}^{2} {\relaxationParameterLetter_2} {\timeShiftOperator} - 2  {\fourierShift}^{2} {\timeShiftOperator} - {\courantNumber} {\relaxationParameterLetter_2} {\timeShiftOperator} + 2  {\fourierShift} {\timeShiftOperator}^{2} - 2  {\fourierShift} {\relaxationParameterLetter_2} + {\relaxationParameterLetter_2} {\timeShiftOperator} + 2  {\fourierShift} - 2  {\timeShiftOperator}\right)} {\fourierShift} = 0,\\
			{\left({\fourierShift} {\relaxationParameterLetter_2} - {\fourierShift} + {\timeShiftOperator}\right)} {\timeShiftOperator} = 0.
		\end{cases}
	\end{equation*}
	We have observed that $\timeShiftOperator = 0$ can be discarded. 
	We thus divide the second equation by $\timeShiftOperator$ and solve for $\timeShiftOperator$: $\timeShiftOperator = (1-\relaxationParameterLetter_2)\fourierShift$. 
	In the sequel, we can assume $\relaxationParameterLetter_2\neq 1$, as otherwise $\timeShiftOperator=0$, going against the assumption that $\timeShiftOperator\neq 0$.
	Plugging $\timeShiftOperator$  into the first equation gives $\relaxationParameterLetter_2(1-\relaxationParameterLetter_2)(1-\courantNumber)(\fourierShift+1)(\fourierShift-1)\fourierShift^2 = 0$.
	\begin{itemize}
		\item The case $\relaxationParameterLetter_2 = 0$ does not apply, as we assumed $\relaxationParameterLetter_2\in (0, 2]$.
		\item Consider $\courantNumber = 1$, hence bulk stability requests $\relaxationParameterLetter_2\in(0, 2)$. As $\relaxationParameterLetter_2 \neq 1$, any $(\timeShiftOperator, \fourierShift)$ such that $\timeShiftOperator = (1-\relaxationParameterLetter_2)\fourierShift$ does the job. 
		However, as we consider $\timeShiftOperator\in\closedNeighborhoodInfinity$, since $\relaxationParameterLetter_2\in(0, 2)$, we obtain $\fourierShift\in\neighborhoodInfinity$, thus corresponds to $\unstableRoot$.
		\item Consider $\fourierShift = 0$. This gives $\timeShiftOperator=0$, a contradiction against the assumptions.
	\end{itemize}
	Finally, we have $\fourierShift=\pm 1$, hence $\timeShiftOperator = \pm(1-\relaxationParameterLetter_2)$: the only non-trivial roots.
\end{example}
Notice that problematic $\timeShiftOperator$ are those such that $\timeShiftOperator\in\closedNeighborhoodInfinity$.
Even if the description of the roots of the characteristic equation does not apply to $\timeShiftOperator\in\unitDisk$, these values are harmless since they produce poles associated with geometrically damped modes in $\indexTime$ when processed by the inverse $\timeShiftOperator$-transform, \confer{} \eqref{eq:inverseZTransform}. 
For $\timeShiftOperator\in\closedNeighborhoodInfinity$, one has to understand whether the corresponding $\fourierShift$ fulfills \Cref{def:eigenvalue}, namely is stable (equal to $\stableRoot$, hence contributing to $\stableSubspace$) or unstable (equal to $\unstableRoot$).
Let us discuss this condition-by-condition with the help of \Cref{prop:hershD1Q2}.
\begin{itemize}
	\item Condition \eqref{eq:BB}. The first listed couple is a shared eigenvalue when $\courantNumber<0$, since $\fourierShift=\stableRoot$. Moreover, since in this framework $|\productRootsDOneQThree(\relaxationParameterLetter_2,\courantNumber)|<1$ for $\relaxationParameterLetter_2 < 2$ and $|\productRootsDOneQThree(\relaxationParameterLetter_2,\courantNumber)|=1$ for $\relaxationParameterLetter_2 = 2$, the potential instability from this mode is geometrically damped in space in the former case and spatially extended in the latter.
	The second couple is an eigenvalue when $\relaxationParameterLetter_2 = 2$ and $\courantNumber<0$, and yields spatially extended modes.
	The last couple is harmless since $\timeShiftOperator \in\unitDisk$ under stability conditions.

	\item Condition \eqref{eq:ABB}. The first couple is an eigenvalue when $\courantNumber<0$, with geometric damping in space for $\relaxationParameterLetter_2<2$ and extended spatial pattern for $\relaxationParameterLetter_2 = 2$. 
	The second couple is an eigenvalue when $\relaxationParameterLetter_2 = 2$ and $\courantNumber<0$.
	The last couple is harmless since $\fourierShift \in\unitDisk$ under stability conditions.

	\item Condition \eqref{eq:TwoABB}. The two couples are eigenvalues when $\relaxationParameterLetter_2 = 2$ and $\courantNumber<0$. 

	\item Condition \eqref{eq:extrapolation}. The first couple is an eigenvalue when $\courantNumber > 0$.
	The second couple is such when $\relaxationParameterLetter_2 = 2$ and $\courantNumber<0$.
	The last couple may yield Godunov-Ryabenkii instabilities at first sight for $\orderExtrapolation$ large enough, since $\timeShiftOperator\in\neighborhoodInfinity$. 
	Clearly, $\fourierShift$ does not equal $\stableRoot$ except when $\tfrac{1}{2}(\courantNumber-1)\relaxationParameterLetter_2 + 1 = 0$, which is harmless.
	This value is totally fictitious and coming from the admissibility of $\fourierShift = 0$ as eigenvalue of $\matrixPolynomialBulk{\timeShiftOperator}(\fourierShift)$.
\end{itemize}

\Cref{tab:modesD1Q2} recapitulates eigenvalues shared between bulk and boundary scheme.
When none is found, we can conclude that the scheme is strongly stable, see the classification at the end of \Cref{sec:recurrenceSpace}.

\begin{table}\caption{\label{tab:modesD1Q2}Unstable candidate eigenvalues (shared between bulk and boundary scheme) for the \lbmScheme{1}{2} scheme.}
	\begin{center}
		\begin{tabular}{|c||c|c|c|}
			\cline{2-4}
			\multicolumn{1}{c||}{} & \multicolumn{2}{c|}{$\courantNumber<0$ (outflow)} & $\courantNumber>0$ (inflow)\\
			\cline{2-4}
			\multicolumn{1}{c||}{} & $\relaxationParameterLetter_2\in(0, 2)$ & $\relaxationParameterLetter_2 = 2$ & $\relaxationParameterLetter_2\in(0, 2]$ \\
			\hline
			\hline
			Bounce-back \eqref{eq:BB} & $(1, \productRootsDOneQThree)$ & $(1, -1)$ and $(-1, 1)$ & none \\
			Anti-bounce-back \eqref{eq:ABB} & $(-1, -\productRootsDOneQThree)$ & $(-1, 1)$ & none \\
			Two-steps anti-bounce-back \eqref{eq:TwoABB} & none & $(1, -1)$ and $(-1, 1)$ & none \\
			Extrapolation $\orderExtrapolation \in\integerInterval{1}{4}$ \eqref{eq:extrapolation} & none & $(-1, 1)$ & $(1, 1)$ \\
			Kinetic Dirichlet \eqref{eq:kineticDirichlet} & none & none & none \\
			\hline
		\end{tabular}
	\end{center}
\end{table}

\paragraph{Eigenvectors: conclusion of the proof and description of the instabilities}

We want to understand to what extent stability can hold for boundary conditions where eigenvalues shared by bulk and boundary schemes have been detected, \confer{} \Cref{tab:modesD1Q2}.
For the boundary conditions which are unstable, we give a quantitative description of these instabilities seen in numerical experiments, \confer{} \Cref{sec:numericalSimD1Q2}.

Except when $2+\relaxationParameterLetter_2 (\courantNumber-1) = 0$, that is when $\coefficientCharEquationInFourier_{-1}(\timeShiftOperator)\equiv 0$, $\stableSubspace(\timeShiftOperator)$  can be constructed as in \Cref{prop:noDiracBoundary}.
In this setting, one can take $\vectorial{\eigenvectorLetter}_{0}^1 = \vectorial{\eigenvectorLetter}_{0} = \transpose{(\relaxationParameterLetter_2-1, 1+\courantNumber\relaxationParameterLetter_2)}$.

In the case where  $2+\relaxationParameterLetter_2 (\courantNumber-1) = 0$, that is $\relaxationParameterLetter_2 = \tfrac{2}{1-\courantNumber}$, the algebraic multiplicity of $\fourierShift \equiv 0$ as zero of \eqref{eq:equalityCharEquations} equals two, and we thus have to investigate the presence of a Jordan chain of length two.
This is the problem of finding the first generalized eigenvector $\tilde{\vectorial{\eigenvectorLetter}}_{0}(\timeShiftOperator)$ solution of $\matrixPolynomialBulk{\timeShiftOperator}'(0) \vectorial{\eigenvectorLetter}_{0} + \matrixPolynomialBulk{\timeShiftOperator}(0)\tilde{\vectorial{\eigenvectorLetter}}_{0}(\timeShiftOperator) = \zeroMatrix{2}$, which gives $\tilde{\vectorial{\eigenvectorLetter}}_{0}(\timeShiftOperator) = 
\transpose{(1, 1-2\timeShiftOperator)}$.
In this setting, the solution of \eqref{eq:resolventBulk} $\laplaceTransformed{\vectorial{\momentDiscrete}}_{\indexSpace}(\timeShiftOperator)\in\stableSubspace(\timeShiftOperator)$ for $\timeShiftOperator\in\neighborhoodInfinity$ thus reads
\begin{equation}\label{eq:generalSolutionStable}
	\laplaceTransformed{\vectorial{\momentDiscrete}}_{\indexSpace}(\timeShiftOperator) = 
		\coefficientZero(\timeShiftOperator) \vectorial{\eigenvectorLetter}_{0}\delta_{0\indexSpace} + \tilde{\coefficientZero}(\timeShiftOperator)(\tilde{\vectorial{\eigenvectorLetter}}_{0}(\timeShiftOperator)\delta_{0\indexSpace}+\vectorial{\eigenvectorLetter}_{0}\delta_{1\indexSpace}),
\end{equation}
for $\indexSpace\in\naturals$, where $\coefficientZero(\timeShiftOperator),  \tilde{\coefficientZero}(\timeShiftOperator)\in\complex$ are to be determined by the boundary condition \eqref{eq:resolventBoundary}: 
\begin{equation}\label{eq:lopatiskiiSystemZero}
	\kreissLopatinskiiMatrix(\timeShiftOperator)\transpose{(\coefficientZero(\timeShiftOperator), \tilde{\coefficientZero}(\timeShiftOperator))} = \laplaceTransformed{\vectorial{\boundarySourceTermMoments}}_0(\timeShiftOperator), \quad
	\text{with}\quad \kreissLopatinskiiMatrix(\timeShiftOperator)= 
	\begin{bmatrix}
		 (\timeShiftOperator\identityMatrix{2} - \schemeMatrixBoundaryByPower{0}{0})  \vectorial{\eigenvectorLetter}_{0}\, | \,(\timeShiftOperator\identityMatrix{2} - \schemeMatrixBoundaryByPower{0}{0}) \tilde{\vectorial{\eigenvectorLetter}}_{0}(\timeShiftOperator) - \schemeMatrixBoundaryByPower{0}{1} \vectorial{\eigenvectorLetter}_{0} 
	\end{bmatrix}.
\end{equation}
We set $\kreissLopatinskiiDet(\timeShiftOperator)\definitionEquality\determinant(\kreissLopatinskiiMatrix(\timeShiftOperator))$ and have the following corollary of \Cref{prop:noDiracBoundary}.
\begin{corollary}
	Under \eqref{eq:D1Q2VonNeumann}, if $\relaxationParameterLetter_2 \neq \tfrac{2}{1-\courantNumber}$ (respectively, if $\relaxationParameterLetter_2 = \tfrac{2}{1-\courantNumber}$), the solution of \eqref{eq:resolventBulk} in $\stableSubspace(\timeShiftOperator)$ is given by \eqref{eq:generalStableSolutionAllSchemes} (respectively, \eqref{eq:generalSolutionStable}).
	Enforcing the boundary condition \eqref{eq:resolventBoundary}, the coefficients are given by \eqref{eq:coefficientWithKreissLopDetFinal} (respectively, \eqref{eq:lopatiskiiSystemZero}), which also read
	\begin{align}
		\begin{cases}
			\kreissLopatinskiiScalarProduct(\timeShiftOperator) \coefficientStableSolution(\timeShiftOperator) =  \laplaceTransformed{\boundarySourceTerm}_{1, -1}(\timeShiftOperator), \\
			\kreissLopatinskiiScalarProduct(\timeShiftOperator) \coefficientZero^1(\timeShiftOperator) = 0, 
		\end{cases}
		\qquad &\text{if}\quad \relaxationParameterLetter_2\neq \tfrac{2}{1-\courantNumber}.\label{eq:kreissLopatinskiDependenceBoundaryDatumD1Q2}\\
		\begin{cases}
			\kreissLopatinskiiDet(\timeShiftOperator) \coefficientZero(\timeShiftOperator) =2(\timeShiftOperator^2 + \tfrac{\courantNumber+1}{\courantNumber-1}) \laplaceTransformed{\boundarySourceTerm}_{1, -1}(\timeShiftOperator), \\
			\kreissLopatinskiiDet(\timeShiftOperator) \tilde{\coefficientZero}(\timeShiftOperator) = 0, 
		\end{cases}
		\qquad &\text{if}\quad \relaxationParameterLetter_2= \tfrac{2}{1-\courantNumber}.\label{eq:kreissLopatinskiDependenceBoundaryDatumD1Q2Zero}
	\end{align}
\end{corollary}
\begin{proof}
	The case $\relaxationParameterLetter_2\neq\tfrac{2}{1-\courantNumber}$ fits into the framework of \Cref{prop:noDiracBoundary}. 
	By an explicit computation $\scalarFactorFromAdjugate(\timeShiftOperator) = (2+(\courantNumber-1)\relaxationParameterLetter_2)\timeShiftOperator$, using \eqref{eq:scalarFactorFromAdjugate}. 
	This quantity never vanish for $\timeShiftOperator\in\closedNeighborhoodInfinity$, hence can be safely simplified to give \eqref{eq:coefficientWithKreissLopDetFinal}, hence \eqref{eq:kreissLopatinskiDependenceBoundaryDatumD1Q2}.

	For $\relaxationParameterLetter_2=\tfrac{2}{1-\courantNumber}$, like in the proof of \Cref{prop:noDiracBoundary}, we have $\kreissLopatinskiiDet(\timeShiftOperator) {\coefficientZero}(\timeShiftOperator) = \determinant[ \momentMatrix\canonicalBasisVector{1}  \, | \,   (\timeShiftOperator\identityMatrix{2} - \schemeMatrixBoundaryByPower{0}{0}) \tilde{\vectorial{\eigenvectorLetter}}_{0}(\timeShiftOperator) - \schemeMatrixBoundaryByPower{0}{1} \vectorial{\eigenvectorLetter}_{0} ] \laplaceTransformed{\boundarySourceTerm}_{1, -1}(\timeShiftOperator)$ and $\kreissLopatinskiiDet(\timeShiftOperator) \tilde{\coefficientZero}(\timeShiftOperator) = \determinant[(\timeShiftOperator\identityMatrix{2} - \schemeMatrixBoundaryByPower{0}{0})\vectorial{\eigenvectorLetter}_0 \, | \,     \momentMatrix\canonicalBasisVector{1}] \laplaceTransformed{\boundarySourceTerm}_{1, -1}(\timeShiftOperator)$.
	The second equation in \eqref{eq:kreissLopatinskiDependenceBoundaryDatumD1Q2Zero} comes from observing that $\determinant[(\timeShiftOperator\identityMatrix{2} - \schemeMatrixBoundaryByPower{0}{0})\vectorial{\eigenvectorLetter}_0 \, | \,     \momentMatrix\canonicalBasisVector{1}] = -\scalarFactorFromAdjugate(\timeShiftOperator)\equiv 0$ in the present setting.
	Explicit computations, where terms linked to specific boundary conditions on the right-hand side are collinear to $ \momentMatrix\canonicalBasisVector{1}$, thus vanish thanks to the multilinearity of the determinant, yield \eqref{eq:kreissLopatinskiDependenceBoundaryDatumD1Q2Zero}.
\end{proof}


\subparagraph{Continuity of the extension of $\stableSubspace(\timeShiftOperator)$ to $\unitCircle$}

We study the continuity (or lack thereof) of the extension of $\vectorial{\eigenvectorLetter}_{\textnormal{s}}(\timeShiftOperator)$ to $\unitCircle$.
\begin{itemize}
	\item $\courantNumber<0$. We compute $\transpose{(1, {\eigenvectorLetter}_{\pm, 2}(\timeShiftOperator))} \in\kernel(\matrixPolynomialBulk{\timeShiftOperator}(\fourierShift_{\pm}(\timeShiftOperator)))$, using both $\fourierShift_+(\timeShiftOperator)$ and $\fourierShift_-(\timeShiftOperator)$ (notice that only one produces $\vectorial{\eigenvectorLetter}_{\textnormal{s}}(\timeShiftOperator)$ at a given $\timeShiftOperator$).
	Here, ${\eigenvectorLetter}_{\pm, 2}(\timeShiftOperator) = -\transpose{\canonicalBasisVector{1}}\matrixPolynomialBulk{\timeShiftOperator}(\fourierShift_{\pm}(\timeShiftOperator))\canonicalBasisVector{1}/(\transpose{\canonicalBasisVector{1}}\matrixPolynomialBulk{\timeShiftOperator}(\fourierShift_{\pm}(\timeShiftOperator))\canonicalBasisVector{2})$ is a fraction. Studying the possibility of having a vanishing denominator, we obtain that the only possible values are $\timeShiftOperator = \pm 1$ when $\relaxationParameterLetter_2 = 2$, otherwise this never happens.
	In this case, the stable root $\stableRoot$ is parametrized by $\fourierShift_+$, and the numerator of ${\eigenvectorLetter}_{\textnormal{s}, 2}(\timeShiftOperator)$ in $\vectorial{\eigenvectorLetter}_{\textnormal{s}} = \transpose{(1, {\eigenvectorLetter}_{\textnormal{s}, 2})}$ equals $8\courantNumber^2$ when $\timeShiftOperator = \pm 1$, so the zero in the denominator is not compensated.
	We obtain
	\begin{equation}\label{eq:eigenvectorD1Q2CloseMinusOne}
		{\eigenvectorLetter}_{\textnormal{s}, 2}(\timeShiftOperator) =
		\begin{cases}
			\frac{\courantNumber\relaxationParameterLetter_2}{\relaxationParameterLetter_2 - 2} \pm \frac{{\left(\courantNumber^{2} - 1\right)} {\relaxationParameterLetter_2}^{2} + 4  {\relaxationParameterLetter_2} - 4}{\courantNumber {\relaxationParameterLetter_2}^{3} - 4  \courantNumber {\relaxationParameterLetter_2}^{2} + 4  \courantNumber {\relaxationParameterLetter_2}} (\timeShiftOperator \mp 1) + \bigO{(\timeShiftOperator\mp 1)^2}, \qquad &\text{if}\quad\relaxationParameterLetter_2 \in (0, 2),\\
			\mp \frac{2\courantNumber}{\timeShiftOperator\mp 1} \mp \frac{\timeShiftOperator\mp 1}{2\courantNumber} + \bigO{(\timeShiftOperator\mp 1)^2}, \qquad &\text{if}\quad\relaxationParameterLetter_2 = 2.
		\end{cases}
	\end{equation}
	When $\relaxationParameterLetter_2 = 2$, $\stableSubspace(\timeShiftOperator)$ cannot be continuously extended to $\unitCircle$ close to $\timeShiftOperator = \pm 1$, and the behavior of $\momentDiscrete_1$ and $\momentDiscrete_2$ shall therefore be quite different.
	The instabilities caused by $\timeShiftOperator = \pm 1$ are ``\strong{one order worse}'' for $\momentDiscrete_2$ than for $\momentDiscrete_1$.
	\begin{remark}
		In this case and in what follows, poles in $\vectorial{\eigenvectorLetter}_{\textnormal{s}}(\timeShiftOperator)$ are at most of order one. 
		Therefore, modes of kind \threeboxes{\notContinuousExtensionMark}{\eigenvalueMark}{\infiniteKL} and \threeboxes{\notContinuousExtensionMark}{\noEigenvalueMark}{\infiniteKL} are harmless as far as strong stability for all components of the system is concerned, \confer{} end of \Cref{sec:recurrenceSpace}.
		For \threeboxes{\notContinuousExtensionMark}{\noEigenvalueMark}{\infiniteKL} (\idEst{} the mode is not in \Cref{tab:modesD1Q2}), we do not even have to show the presence of a pole in $\kreissLopatinskiiScalarProduct(\timeShiftOperator)$, since this is the only possible case for \threeboxes{\notContinuousExtensionMark}{\noEigenvalueMark}{}.
	\end{remark}

	Notice that in the setting of \Cref{fig:D1Q2-s-3_2-C--1_2-revision}, we have $|\frac{\courantNumber\relaxationParameterLetter_2}{\relaxationParameterLetter_2 - 2}|=\tfrac{3}{2}$: it is therefore not surprising that in the plot relative to  \eqref{eq:BB} and \eqref{eq:ABB}, the values of $|\momentDiscrete_2|$ are $\tfrac{3}{2}$-times those of $|\momentDiscrete_1|$.


	\item $\courantNumber>0$. An eigenvector $\vectorial{\eigenvectorLetter}_{\textnormal{s}}(\timeShiftOperator)  = 
	\transpose{(1, {\eigenvectorLetter}_{\textnormal{s}, 2}(\timeShiftOperator))}$ features ${\eigenvectorLetter}_{\textnormal{s}, 2}(\timeShiftOperator)$ with vanishing denominator only for $\timeShiftOperator = \pm 1$ when $\relaxationParameterLetter_2 = 2$.
	In these vicinities, the stable root $\stableRoot$ is parametrized by $\fourierShift_+$, and the numerator of ${\eigenvectorLetter}_{\textnormal{s}, 2}(\timeShiftOperator)$ vanishes as well, compensating the pole stemming from the denominator.
	For instance
	\begin{equation}\label{eq:eigenvectorD1Q2CloseOneOne}
		{\eigenvectorLetter}_{\textnormal{s}, 2}(\timeShiftOperator) = \courantNumber \pm \frac{1-\courantNumber^2}{\courantNumber\relaxationParameterLetter_2} (\timeShiftOperator \mp 1)+\bigO{(\timeShiftOperator \mp 1)^2},
	\end{equation}
	regardless of the value of $\relaxationParameterLetter_2\in(0, 2]$, which shows that $\stableSubspace(\timeShiftOperator)$ continuously extends to $\unitCircle$ in this framework.
	In the setting of \Cref{fig:D1Q2-s-3_2-C-1_2-revision} and \ref{fig:D1Q2-s-2-C-1_2-revision}, we have that $|\courantNumber|=\tfrac{1}{2}$, so the values of $|\momentDiscrete_2|$ are half those of $|\momentDiscrete_1|$ in the plots relative to mildly-unstable boundary conditions.
	
	\begin{remark}
		Thanks to this continuous extension, for boundary conditions without critical eigenvalues listed in \Cref{tab:modesD1Q2} (all modes are \threeboxes{\continuousExtensionMark}{\noEigenvalueMark}{\nonZeroKL}), we can conclude that they are strongly stable without further discussions, as $\coefficientStableSolution(\timeShiftOperator)$ can be uniformly bounded from the datum as $\timeShiftOperator$ approaches $\unitCircle$, see \eqref{eq:kreissLopatinskiDependenceBoundaryDatumD1Q2}. 
	\end{remark}
	
\end{itemize}

Let us now conclude using Kreiss-Lopatinskii determinants bound to each specific boundary condition.

\begin{figure}[h]
	\begin{center}
		\includegraphics[width=1\textwidth]{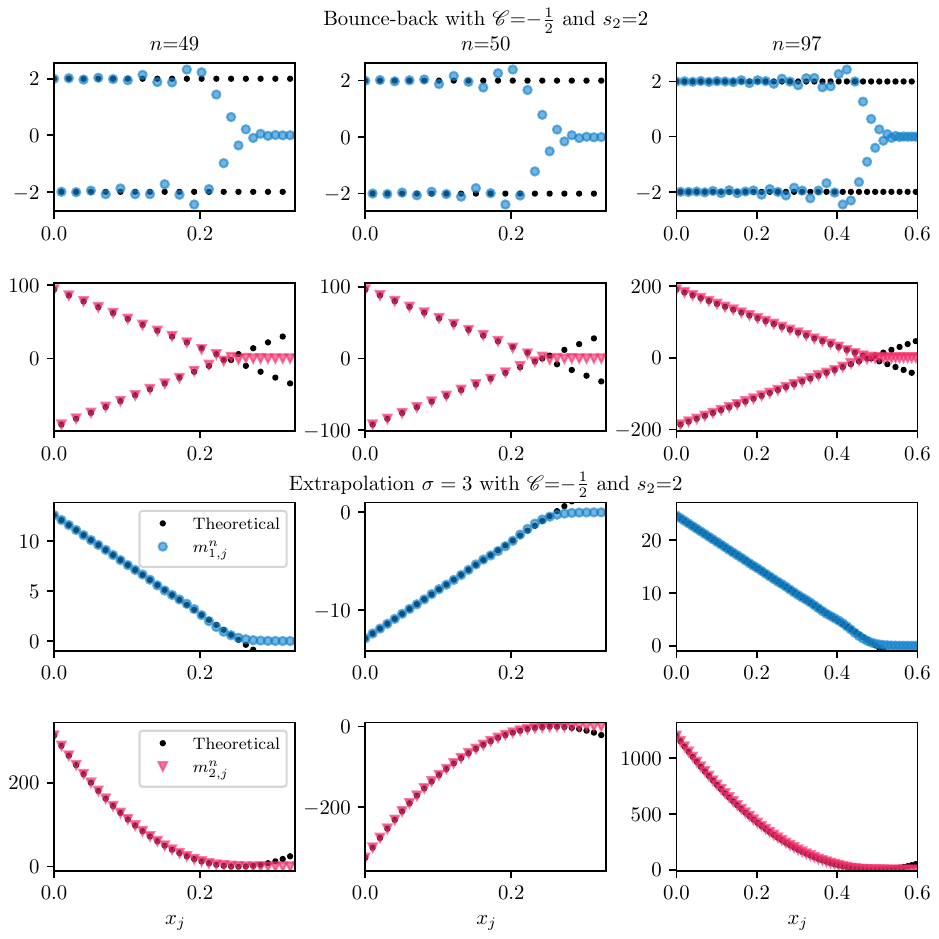}
	\end{center}\caption{\label{fig:D1Q2-theory-vs-pract}Solution at different times for the \lbmScheme{1}{2} scheme and two boundary conditions (colored) against the theoretical predictions (black) from \eqref{eq:theoreticalSolutionBB} and \eqref{eq:theoreticalSolutionSigma3}.}
\end{figure}

\begin{figure}[h]
	\begin{center}
		\includegraphics[width=1\textwidth]{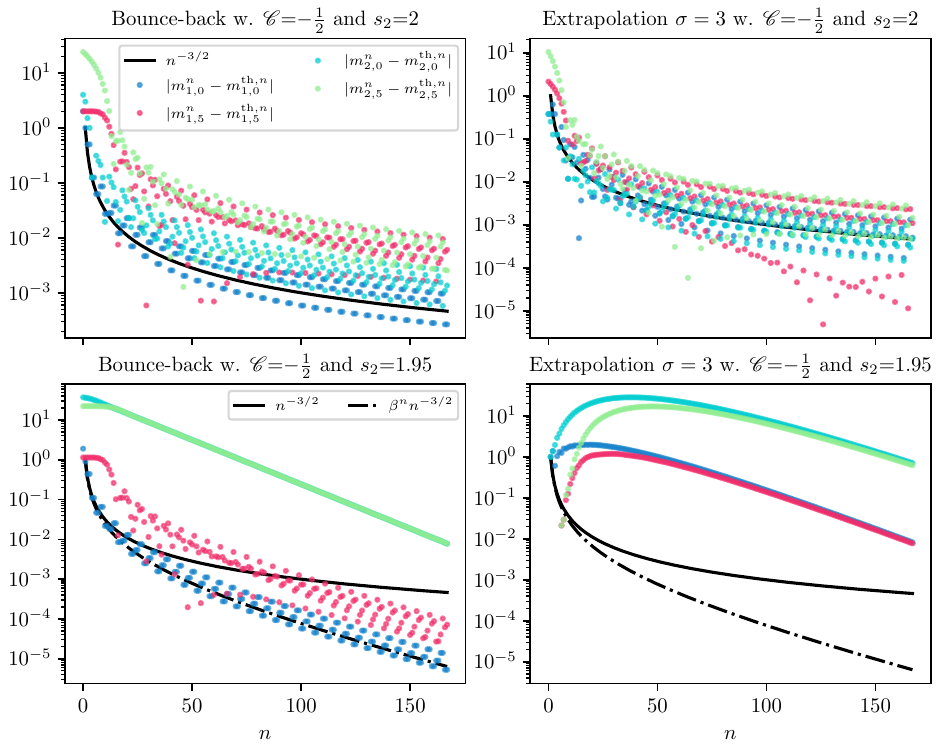}
	\end{center}\caption{\label{fig:D1Q2-theory-vs-pract-difference}For the \lbmScheme{1}{2} scheme endowed with two boundary conditions, difference at the first and sixth grid-point between numerical solution and theoretical prediction in the long-time limit given by the right-hand sides of \eqref{eq:theoreticalSolutionBB} and \eqref{eq:theoreticalSolutionSigma3}. On the second row, $\beta\approx 0.9747$.}
\end{figure}

\subparagraph{Kreiss-Lopatinskii determinants in the outflow case $\courantNumber<0$}

Consider $\courantNumber<0$.

	\begin{itemize}
		\item For \eqref{eq:BB}, the mode $(-1, -\productRootsDOneQThree)$ must be considered only for $\relaxationParameterLetter_2 = 2$.
		We have $\kreissLopatinskiiScalarProduct(\timeShiftOperator) = -2\courantNumber(\timeShiftOperator + 1)^{-1} + \bigO{1}$, and this mode is of kind \threeboxes{\notContinuousExtensionMark}{\eigenvalueMark}{\infiniteKL}, so harmless.
		For the mode $(1, \productRootsDOneQThree)$ is an eigenvalue for any value of $\relaxationParameterLetter_2$, we obtain that $\kreissLopatinskiiScalarProduct(\timeShiftOperator) = \frac{(\courantNumber+1)\relaxationParameterLetter_2 - 2}{2\courantNumber\relaxationParameterLetter_2}(\timeShiftOperator - 1) + \bigO{(\timeShiftOperator - 1)^2}$, hence \threeboxes{\continuousExtensionMark}{\eigenvalueMark}{\zeroKL} when $\relaxationParameterLetter_2\in(0, 2)$ and \threeboxes{\notContinuousExtensionMark}{\eigenvalueMark}{\zeroKL} when $\relaxationParameterLetter_2 = 2$.
		This shows that the scheme is neither SS nor SSOO, since $\coefficientStableSolution(\timeShiftOperator)$ has a first-order pole at $\timeShiftOperator = 1$ regardless of $\laplaceTransformed{\boundarySourceTerm}_{1, -1}(\timeShiftOperator)$.

        \begin{figure}[h]
            \begin{center}
                \includegraphics[width=1\textwidth]{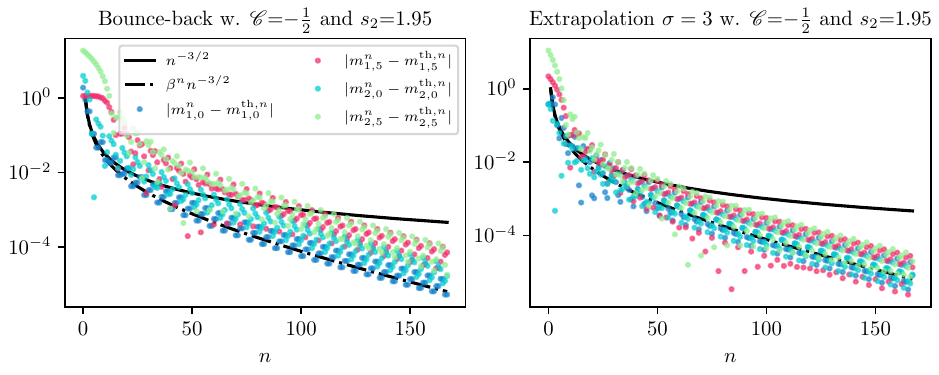}
            \end{center}\caption{\label{fig:D1Q2-theory-vs-pract-difference-more-terms}For the \lbmScheme{1}{2} scheme endowed with two boundary conditions, difference at the first and sixth grid-point between numerical solution and theoretical prediction in the long-time limit given by the right-hand sides of \eqref{eq:theoreticalSolutionBBWithSmallPole} and \eqref{eq:theoreticalSolutionSigma3MoreTerms}. We have $\beta\approx 0.9747$.}
        \end{figure}

		We now utilize previous information to quantitatively understand what is observed on \Cref{fig:D1Q2-s-3_2-C--1_2-revision} and \ref{fig:D1Q2-s-2-C--1_2-revision} for \eqref{eq:BB}.
		We take the inverse $\timeShiftOperator$-transform through \eqref{eq:inverseZTransform}, with a radius $R$ large enough (\idEst{}, which allows enclosing all the singularities of $\timeShiftOperator^{\indexTime-1}\laplaceTransformed{{\momentDiscrete}}_{\indexVelocity, \indexSpace}(\timeShiftOperator)$), $\indexVelocity\in\integerInterval{1}{\numberVelocities}$, and $\indexSpace\in\naturals$.
		We have 
		\begin{equation}\label{eq:inverseZPractical}
			{{\momentDiscrete}}_{\indexVelocity, \indexSpace}^{\indexTime} = \frac{1}{2\pi i}\oint_{\mathcal{C}(0, R)}\coefficientStableSolution(\timeShiftOperator) {\eigenvectorLetter}_{\textnormal{s}, \indexVelocity}(\timeShiftOperator) \timeShiftOperator^{\indexTime - 1}\stableRoot(\timeShiftOperator)^{\indexSpace}\differential{\timeShiftOperator}.
		\end{equation}
		We now have to propose a practical way of computing this contour integral, which very often passes through the \strong{residue theorem}.
		Still, except when $\relaxationParameterLetter_2 = \tfrac{2}{1-\courantNumber}$, the stable root $\stableRoot(\timeShiftOperator)$ (and thus $\coefficientStableSolution(\timeShiftOperator)$ and ${\eigenvectorLetter}_{\textnormal{s}, 2}(\timeShiftOperator)$) are not meromorphic due to the presence of \strong{branch points} generated by the square root in \eqref{eq:rootSecondOrder}.
		These branch points---the four roots of $\coefficientCharEquationInFourier_{0}(\timeShiftOperator)^2 - 4\coefficientCharEquationInFourier_{-1}(\timeShiftOperator)\coefficientCharEquationInFourier_{1}(\timeShiftOperator)$, which are either real or pairwise complex conjugate---are in $\unitDisk$ when $\relaxationParameterLetter_2 \in (0, 2)$ and on $\unitCircle$ for $\relaxationParameterLetter_2 =2$  for the considered setting with $\courantNumber\neq 0$, see \cite{bellotti2025perfectly}.
		We thus proceed by contour deformation on $\mathcal{C}(0, R)$, which is replaced by a dog-bone--like contour $\mathcal{D}$ that turns around suitable branch cuts.
		For each branch cut $\mathcal{B}$, let us denote the corresponding part of $\mathcal{D}$ above (respectively, below) $\mathcal{B}$ by $\mathcal{D}_{\mathcal{B}}^{\textnormal{ab}}$ (respectively, $\mathcal{D}_{\mathcal{B}}^{\textnormal{bl}}$).
		We obtain 
		\begin{multline*}
			\oint_{\mathcal{C}(0, R)}\coefficientStableSolution(\timeShiftOperator) {\eigenvectorLetter}_{\textnormal{s}, \indexVelocity}(\timeShiftOperator) \timeShiftOperator^{\indexTime - 1}\stableRoot(\timeShiftOperator)^{\indexSpace}\differential{\timeShiftOperator} = \oint_{\mathcal{D}}\coefficientStableSolution(\timeShiftOperator) {\eigenvectorLetter}_{\textnormal{s}, \indexVelocity}(\timeShiftOperator) \timeShiftOperator^{\indexTime - 1}\stableRoot(\timeShiftOperator)^{\indexSpace}\differential{\timeShiftOperator} \\
			- \sum_{\mathcal{B}\in\text{branch ct.}}\int_{\mathcal{D}_{\mathcal{B}}^{\textnormal{ab}}\cup \mathcal{D}_{\mathcal{B}}^{\textnormal{bl}}} \coefficientStableSolution(\timeShiftOperator) {\eigenvectorLetter}_{\textnormal{s}, \indexVelocity}(\timeShiftOperator) \timeShiftOperator^{\indexTime - 1}\stableRoot(\timeShiftOperator)^{\indexSpace}\differential{\timeShiftOperator}.
		\end{multline*}
		As the integrand is now \strong{assumed meromorphic} inside $\mathcal{D}$ (for instance because the $\laplaceTransformed{{\boundarySourceTerm}}_{\positiveVelocityIndex, -1}(\timeShiftOperator)$ at hand does not have branch points or essential singularities, thus is meromorphic), we have 
		\begin{align}
			\oint_{\mathcal{D}}\coefficientStableSolution(\timeShiftOperator) {\eigenvectorLetter}_{\textnormal{s}, \indexVelocity}(\timeShiftOperator) \timeShiftOperator^{\indexTime - 1}\stableRoot(\timeShiftOperator)^{\indexSpace}\differential{\timeShiftOperator} &=2\pi i  \sum_{\substack{\targetEigenvalue\in\complex\,\text{pole of}\\\coefficientStableSolution {\eigenvectorLetter}_{\textnormal{s}, \indexVelocity} \timeShiftOperator^{\indexTime - 1}\stableRoot^{\indexSpace}}}\residue{\coefficientStableSolution(\timeShiftOperator) {\eigenvectorLetter}_{\textnormal{s}, \indexVelocity}(\timeShiftOperator) \timeShiftOperator^{\indexTime - 1}\stableRoot(\timeShiftOperator)^{\indexSpace}}{\targetEigenvalue}\nonumber\\
			&\approx 2\pi i \sum_{\substack{\targetEigenvalue\in\closedNeighborhoodInfinity\,\text{pole of}\\\coefficientStableSolution {\eigenvectorLetter}_{\textnormal{s}, \indexVelocity} \timeShiftOperator^{\indexTime - 1}\stableRoot^{\indexSpace}}}\residue{\coefficientStableSolution(\timeShiftOperator) {\eigenvectorLetter}_{\textnormal{s}, \indexVelocity}(\timeShiftOperator) \timeShiftOperator^{\indexTime - 1}\stableRoot(\timeShiftOperator)^{\indexSpace}}{\targetEigenvalue}\label{eq:forgotAboutInnerPoles}
		\end{align}
		by the residue theorem and considering that contributions from poles in $\unitDisk$ decay geometrically (times polynomial terms in $\indexTime$ for poles of order larger than one).
		This last approximation is done under the assumption, considered in what follows, that $\indexTime$ is large enough\footnote{Notice that this ``large enough'' can also depend---for instance---on the chosen value of $\indexSpace$.}.
		For the term linked to branch points 
		\begin{equation}\label{eq:decreaseBranchPoint}
			\sum_{\mathcal{B}\in\text{branch ct.}}\int_{\mathcal{D}_{\mathcal{B}}^{\textnormal{ab}}\cup \mathcal{D}_{\mathcal{B}}^{\textnormal{bl}}} \coefficientStableSolution(\timeShiftOperator) {\eigenvectorLetter}_{\textnormal{s}, \indexVelocity}(\timeShiftOperator) \timeShiftOperator^{\indexTime - 1}\stableRoot(\timeShiftOperator)^{\indexSpace}\differential{\timeShiftOperator}= \bigO{\indexTime^{-3/2} \beta^{\indexTime}}\xrightarrow[]{\indexTime\to+\infty} 0,
		\end{equation}
		where $\beta\in[0, 1]$ is the largest modulus of the branch points, \confer{} \cite{bellotti2025perfectly}---which hugely relies on \cite{flajolet2009analytic}---or \cite{jury1961note}.
		As already pointed out, we have $\beta\in [0, 1)$ for $\relaxationParameterLetter_2\in(0, 2)$ and $\beta = 1$ if $\relaxationParameterLetter_2=2$.
		This means that this term---that we shall neglect---goes to zero (at least) geometrically in $\indexTime$ times a modulation by $\indexTime^{-3/2}$ when $\relaxationParameterLetter_2\in (0, 2)$ and goes to zero algebraically as $\indexTime^{-3/2}$ when $\relaxationParameterLetter_2 = 2$.
		We finally obtain, for large $\indexTime$, the approximation
		\begin{equation}\label{eq:residualApprox}
			{{\momentDiscrete}}_{\indexVelocity, \indexSpace}^{\indexTime} \approx \sum_{\substack{\targetEigenvalue\in\closedNeighborhoodInfinity\,\text{pole of}\\\coefficientStableSolution {\eigenvectorLetter}_{\textnormal{s}, \indexVelocity} \timeShiftOperator^{\indexTime - 1}\stableRoot^{\indexSpace}}}\residue{\coefficientStableSolution(\timeShiftOperator) {\eigenvectorLetter}_{\textnormal{s}, \indexVelocity}(\timeShiftOperator) \timeShiftOperator^{\indexTime - 1}\stableRoot(\timeShiftOperator)^{\indexSpace}}{\targetEigenvalue} = 
            \sum_{\substack{\targetEigenvalue\in\closedNeighborhoodInfinity\,\text{pole of}\\\boundarySourceTerm_{\positiveVelocityIndex, -1}/\kreissLopatinskiiScalarProduct {\eigenvectorLetter}_{\textnormal{s}, \indexVelocity} \timeShiftOperator^{\indexTime - 1}\stableRoot^{\indexSpace}}}\hspace{-1.5em}\residue{\frac{\laplaceTransformed{\boundarySourceTerm}_{\positiveVelocityIndex, -1}(\timeShiftOperator) {\eigenvectorLetter}_{\textnormal{s}, \indexVelocity}(\timeShiftOperator) \timeShiftOperator^{\indexTime - 1}\stableRoot(\timeShiftOperator)^{\indexSpace}}{\kreissLopatinskiiScalarProduct(\timeShiftOperator)}}{\targetEigenvalue}.
		\end{equation}
        We stress the fact that the previous computation is done by fixing the space-coordinate $\indexSpace$ and looking at the solution evolving at this point, in a sort of Eulerian fashion.
		This standpoint is not aimed at describing and following the propagation of sorts of ``discrete shocks'', in a Lagrangian-like point of view. This fact shall rather be investigated by group velocity.
        Equation \eqref{eq:residualApprox} gives a non-trivial generalization of the final value theorem for the $\timeShiftOperator$-transform \cite{jury1964theory} (in which the only considered residue is at $\targetEigenvalue=1$ when it is a first-order pole), encompassing various long-time behaviors such as oscillations (\strong{e.g.}, $\targetEigenvalue=-1$ being a first-order pole), polynomial growth (\strong{e.g.}, $\targetEigenvalue\in\unitCircle$ pole of order higher than one), geometric explosion (\strong{e.g.}, $\targetEigenvalue\in\neighborhoodInfinity$ a pole), \strong{etc}.
		The specific source term $\laplaceTransformed{{\boundarySourceTerm}}_{1, -1}(\timeShiftOperator) = 1$ used in \Cref{sec:numericalSimD1Q2} yields
		\begin{equation}\label{eq:theoreticalSolutionBB}
			\vectorial{\momentDiscrete}_{\indexSpace}^{\indexTime} \approx 
			\begin{cases}
				\transpose{(\frac{2   {\courantNumber} {\relaxationParameterLetter_2} \productRootsDOneQThree^{j}}{{\left({\courantNumber} + 1\right)} {\relaxationParameterLetter_2} - 2}, \frac{2  {\courantNumber}^{2} {\relaxationParameterLetter_2}^{2} \productRootsDOneQThree^{j}}{{\left({\courantNumber} + 1\right)} {\relaxationParameterLetter_2}^{2} - 2  {\left({\courantNumber} + 2\right)} {\relaxationParameterLetter_2} + 4})}, \qquad &\text{if}\quad \relaxationParameterLetter_2 \in (0, 2), \\
				\transpose{(2  \left(-1\right)^{\indexSpace}, -2  {\left(2  {\courantNumber} n - 2  {\courantNumber} + 2  j + 1\right)} \left(-1\right)^{\indexSpace})}, \qquad &\text{if}\quad \relaxationParameterLetter_2 = 2,
			\end{cases}
		\end{equation}
		where only the pole at $\targetEigenvalue = 1$ has been necessary in \eqref{eq:residualApprox}.
		This result is in perfect agreement with \Cref{fig:D1Q2-s-3_2-C--1_2-revision} and \ref{fig:D1Q2-s-2-C--1_2-revision}---see for instance \Cref{fig:D1Q2-theory-vs-pract} and \ref{fig:D1Q2-theory-vs-pract-difference}---and confirms that the instability is one order worse on the second moment when $\relaxationParameterLetter_2 = 2$.
	    Through \Cref{fig:D1Q2-theory-vs-pract-difference}, we see that the leading neglected terms come from branch points and behave as stipulated by \eqref{eq:decreaseBranchPoint}.
        There is only one snag to this concerning $\momentDiscrete_2$ when $\relaxationParameterLetter_2 = 1.95$, since in this regime and for the relatively short times that we considered in this empirical demonstration, we cannot claim that the neglected terms from the branch points dominate those coming from the neglected pole (\confer{} \eqref{eq:forgotAboutInnerPoles}) at $\targetEigenvalue = \relaxationParameterLetter_2 - 1\in\unitDisk$, which comes from the lack of continuous extension of $\eigenvectorLetter_{\textnormal{s}, 2}(\timeShiftOperator)$ at $\timeShiftOperator = \relaxationParameterLetter_2 - 1$. Considering also this residue, we gain 
        \begin{equation}\label{eq:theoreticalSolutionBBWithSmallPole}
			{\momentDiscrete}_{2, \indexSpace}^{\indexTime} \approx \frac{2  {\courantNumber}^{2} {\relaxationParameterLetter_2}^{2} \productRootsDOneQThree^{j}}{{\left({\courantNumber} + 1\right)} {\relaxationParameterLetter_2}^{2} - 2  {\left({\courantNumber} + 2\right)} {\relaxationParameterLetter_2} + 4} \underbrace{- \frac{2  \left(-1\right)^{\indexSpace} {\courantNumber} {\left({\relaxationParameterLetter_2} - 1\right)}^{\indexTime} {\relaxationParameterLetter_2}}{{\relaxationParameterLetter_2}^{2} - 3  {\relaxationParameterLetter_2} + 2}}_{\text{additional term \strong{vs.} \eqref{eq:theoreticalSolutionBB}}}, \qquad \text{if}\quad \relaxationParameterLetter_2 \in (0, 2),
        \end{equation}
        with which we can now fully claim that we have only neglected, and thereby isolate, the dynamics stemming from the branch points, see \Cref{fig:D1Q2-theory-vs-pract-difference-more-terms} to be compared to the lower row of \Cref{fig:D1Q2-theory-vs-pract-difference}.
        Moreover, when $\relaxationParameterLetter\in(0, 2)$, the instability remains localized on the boundary as $\productRootsDOneQThree\in\unitDisk$.\footnote{Even more extremely, if $\relaxationParameterLetter_2=\tfrac{2}{1-\courantNumber}$, we have $\productRootsDOneQThree=0$, hence $\productRootsDOneQThree^{\indexSpace} = \delta_{0\indexSpace}$.}
		From the form of $\momentDiscrete_{2, \indexSpace}^{\indexTime}$, in particular the coefficients of the terms $\indexTime$ and $\indexSpace$, we see that the instability propagates inside domain at velocity $-\advectionVelocity$ ($-\courantNumber$ is dimensionless units) when $\relaxationParameterLetter_2 = 2$.
		This can be also obtained by the theory of group velocity introduced by \cite{trefethen1984instability} (see also \cite{coulombel2015fully}): since most of the numerical solution stems from a pole at $\targetEigenvalue = 1$, such that $\targetEigenvalue, \stableRoot(\targetEigenvalue)\in\unitCircle$ when $\relaxationParameterLetter_2=2$, we conclude that in this setting the instability is transported at group velocity $\groupVelocity(\targetEigenvalue)$ given by \cite[Equation (3.18)]{trefethen1984instability}
		\begin{equation}\label{eq:groupVelocity}
		\groupVelocity(\targetEigenvalue) = -\latticeVelocity \frac{\stableRoot(\targetEigenvalue)}{\targetEigenvalue}\Bigl ( \frac{\differential{\stableRoot}}{\differential{\timeShiftOperator}}(\targetEigenvalue)\Bigr )^{-1}.
		\end{equation}
		For instance, we obtain $\groupVelocity(1) = -\latticeVelocity\courantNumber = -\advectionVelocity>0$.

		\item Consider \eqref{eq:ABB}. Consider the mode $(-1, -\productRootsDOneQThree)$.
		We have $\kreissLopatinskiiScalarProduct(\timeShiftOperator) = \frac{(\courantNumber+1)\relaxationParameterLetter_2 + 2}{2\courantNumber\relaxationParameterLetter_2}(\timeShiftOperator + 1) + \bigO{(\timeShiftOperator+1)^2}$ (and $\kreissLopatinskiiDet(\timeShiftOperator) = \bigO{\timeShiftOperator+1}$ when \eqref{eq:kreissLopatinskiDependenceBoundaryDatumD1Q2Zero} applies).
		Hence \threeboxes{\continuousExtensionMark}{\eigenvalueMark}{\zeroKL} for $\relaxationParameterLetter_2\in (0, 2)$ and \threeboxes{\notContinuousExtensionMark}{\eigenvalueMark}{\zeroKL} with $\relaxationParameterLetter_2 = 2$, which shows that the boundary condition is neither SS, nor SSOO.
		As, in the case $\relaxationParameterLetter_2 = 2$, $\timeShiftOperator = 1$ is not an eigenvalue, we fall under the harmless case \threeboxes{\notContinuousExtensionMark}{\noEigenvalueMark}{\infiniteKL}, and this argument shall not be repeated in the pages to come.

		Concerning \Cref{fig:D1Q2-s-3_2-C--1_2-revision} and \ref{fig:D1Q2-s-2-C--1_2-revision} for \eqref{eq:ABB}, using \eqref{eq:residualApprox}, we obtain 
		\begin{equation*}
	\vectorial{\momentDiscrete}_{\indexSpace}^{\indexTime} \approx 
	\begin{cases}
		\transpose{(-\frac{2  \left(-1\right)^{n} {\courantNumber} {\relaxationParameterLetter_2} (-\productRootsDOneQThree)^{j}}{{\left({\courantNumber} + 1\right)} {\relaxationParameterLetter_2} - 2}, -\frac{2  \left(-1\right)^{n} {\courantNumber}^{2} {\relaxationParameterLetter_2}^{2} (-\productRootsDOneQThree)^{j}}{{\left({\courantNumber} + 1\right)} {\relaxationParameterLetter_2}^{2} - 2  {\left({\courantNumber} + 2\right)} {\relaxationParameterLetter_2} + 4})}, \qquad &\text{if}\quad \relaxationParameterLetter_2 \in (0, 2), \\
		\transpose{(-2  \left(-1\right)^{n}, 2  {\left(2  {\courantNumber} n - 2  {\courantNumber} + 2  j + 1\right)} \left(-1\right)^{n})}, \qquad &\text{if}\quad \relaxationParameterLetter_2 = 2.
	\end{cases}
\end{equation*}

	\item 	We now consider both \eqref{eq:TwoABB} and \eqref{eq:extrapolation} with $\orderExtrapolation = 1$ when $\relaxationParameterLetter_2 = 2$.
	For both boundary conditions, $\kreissLopatinskiiScalarProduct(-1) = -1 \neq 0$ (hence \threeboxes{\notContinuousExtensionMark}{\eigenvalueMark}{\nonZeroKL}).
	This ensures that this mode is stable as far as $\momentDiscrete_1$ is concerned, but is not enough to compensate the pole on $\eigenvectorLetter_{\textnormal{s}, 2}(\timeShiftOperator)$ to gain a control on $\momentDiscrete_2$ from the boundary data.
	These boundary conditions are not SS when $\relaxationParameterLetter_2 = 2$.
	The mode $\timeShiftOperator = 1$ must be considered for \eqref{eq:TwoABB}: $\kreissLopatinskiiScalarProduct(1) = -1 \neq 0$, thus \threeboxes{\notContinuousExtensionMark}{\eigenvalueMark}{\nonZeroKL}.
	
	This study shows, for \eqref{eq:TwoABB} and \eqref{eq:extrapolation} with $\orderExtrapolation = 1$ with $\relaxationParameterLetter_2 = 2$, that there exists a constant $K>0$ such that 
\begin{equation*}
	|\coefficientStableSolution(\timeShiftOperator)|\leq K |\laplaceTransformed{{\boundarySourceTerm}}_{1, -1}(\timeShiftOperator)|, \qquad \timeShiftOperator\in\neighborhoodInfinity.
\end{equation*}
This control is uniform in $\timeShiftOperator\in\neighborhoodInfinity$ (particularly when $\timeShiftOperator$ approaches $\unitCircle$).
One must not forget that this allows, by looking at the eigenvector, \confer{} \eqref{eq:eigenvectorD1Q2CloseMinusOne}, only a control on $\momentDiscrete_1$ from the boundary data.
As a result 
\begin{equation}
	\sum_{\indexSpace\in\naturals}|\laplaceTransformed{\momentDiscrete}_{1, \indexSpace}(\timeShiftOperator)|^2 \leq K^2 |\laplaceTransformed{{\boundarySourceTerm}}_{1, -1}(\timeShiftOperator)|^2 \sum_{\indexSpace\in\naturals} |\stableRoot(\timeShiftOperator)|^{2\indexSpace} = K^2 |\laplaceTransformed{{\boundarySourceTerm}}_{1, -1}(\timeShiftOperator)|^2 \frac{1}{1-|\stableRoot(\timeShiftOperator)|^{2}}.\label{eq:tmp7}
\end{equation}
Under \eqref{eq:sCondD1Q2-1}--\eqref{eq:sCondD1Q2-2} (\eqref{eq:D1Q2VonNeumann} is not enough) we can use \cite[Lemma 11.3.2]{strikwerda2004finite} to obtain that there exists $\tilde{K}>0$ such that
\begin{equation*}
	\frac{1}{1-|\stableRoot(\timeShiftOperator)|^{2}}\leq \frac{1}{1-|\stableRoot(\timeShiftOperator)|}\leq \frac{\tilde{K}}{|\timeShiftOperator|-1}.
\end{equation*}
This yields 
\begin{equation*}
	\frac{|\timeShiftOperator|-1}{|\timeShiftOperator|} \sum_{\indexSpace\in\naturals}|\laplaceTransformed{\momentDiscrete}_{1, \indexSpace}(\timeShiftOperator)|^2 \leq (|\timeShiftOperator|-1) \sum_{\indexSpace\in\naturals}|\laplaceTransformed{\momentDiscrete}_{1, \indexSpace}(\timeShiftOperator)|^2\leq K^2\tilde{K} |\laplaceTransformed{{\boundarySourceTerm}}_{1, -1}(\timeShiftOperator)|^2.
\end{equation*}
Following the proof of \cite[Proposition 3.2]{lebarbenchon:tel-04214887}, calling $C = \latticeVelocity K^2\tilde{K}$, using the Parseval identity and the inequality by \cite[Lemma 9]{coulombel00616497}, the previous outermost inequality is equivalent to: for all $\alpha >0$
\begin{equation*}
	\frac{\alpha}{1+\alpha\timeStep}  \sum_{\indexSpace\in\naturals}\sum_{\indexTime\in\naturals} \spaceStep\timeStep \, e^{-2\alpha \indexTime\timeStep} |\momentDiscrete_{1, \indexSpace}^{\indexTime}|^2 \leq  C \sum_{\indexTime\in\naturals} \timeStep \, e^{-2\alpha \indexTime\timeStep}|{{\boundarySourceTerm}}_{1, -1}^{\indexTime}|^2.
\end{equation*}

Again using \eqref{eq:residualApprox}, we obtain that in the context of \Cref{fig:D1Q2-s-2-C--1_2-revision} with  \eqref{eq:TwoABB}: $\momentDiscrete_{1, \indexSpace}^{\indexTime}\approx 0$ and $\momentDiscrete_{2, \indexSpace}^{\indexTime}\approx 2\courantNumber ((-1)^{\indexTime}-(-1)^{\indexSpace})$ in the long-time limit.
For \eqref{eq:extrapolation} for $\orderExtrapolation = 1$: $\momentDiscrete_{1, \indexSpace}^{\indexTime}\approx 0$ and $\momentDiscrete_{2, \indexSpace}^{\indexTime}\approx 2\courantNumber (-1)^{\indexTime}$.

	\item For \eqref{eq:extrapolation} with $\orderExtrapolation = 2, 3, 4$, we have to consider the case $\relaxationParameterLetter_2 = 2$.
	In particular $\kreissLopatinskiiScalarProduct(-1) = -\courantNumber^{-\orderExtrapolation}(\timeShiftOperator + 1)^{\orderExtrapolation-1}+\bigO{(\timeShiftOperator+1)^{\orderExtrapolation}}$: \threeboxes{\notContinuousExtensionMark}{\eigenvalueMark}{\zeroKL} and no hope of having SS or even SSOO.


In the context of \Cref{fig:D1Q2-s-2-C--1_2-revision}, this instability can be described by $\momentDiscrete_{1, \indexSpace}^{\indexTime} \approx \courantNumber(-1)^{\indexTime}$ and $\momentDiscrete_{2, \indexSpace}^{\indexTime} \approx -\courantNumber(2\courantNumber\indexTime - 2\courantNumber + 2\indexSpace - 1)(-1)^{\indexTime}$ for $\orderExtrapolation = 2$ and by 
\begin{multline}\label{eq:theoreticalSolutionSigma3}
	\momentDiscrete_{1, \indexSpace}^{\indexTime} \approx  -\tfrac{1}{2}  {\left(2  \courantNumber \indexTime - \courantNumber + 2  \indexSpace - 2\right)} \left(-1\right)^{\indexTime} \courantNumber, \\
	\momentDiscrete_{2, \indexSpace}^{\indexTime} \approx \tfrac{1}{2}  {\left(2  \courantNumber^{2} \indexTime^{2} - 4  \courantNumber^{2} \indexTime + 4  \courantNumber \indexSpace \indexTime - 4  \courantNumber \indexSpace + 2  \indexSpace^{2} - 4  \courantNumber \indexTime + 5  \courantNumber - 4  \indexSpace + 4\right)} \left(-1\right)^{\indexTime} \courantNumber,
\end{multline}
for $\orderExtrapolation = 3$.
For the sake of illustration, we compare \eqref{eq:theoreticalSolutionSigma3} against the actual solution in \Cref{fig:D1Q2-theory-vs-pract} and \ref{fig:D1Q2-theory-vs-pract-difference}, showing excellent agreement and validating the expected trend of the neglected terms pertaining to branch points, see \eqref{eq:decreaseBranchPoint}, except for $\relaxationParameterLetter_2 = 1.95$.
In this case, the approximation of both $\momentDiscrete_{1, \indexSpace}^{\indexTime} \approx 0$ and $\momentDiscrete_{2, \indexSpace}^{\indexTime} \approx 0$ (the boundary condition is SS) is quite crude for relatively short times, as the observable reminder is not only made up of neglected branch points. Indeed, adding the residue of the pole $\targetEigenvalue = 1-\relaxationParameterLetter_2$, we obtain
\begin{equation}\label{eq:theoreticalSolutionSigma3MoreTerms}
    {\momentDiscrete}_{1, \indexSpace}^{\indexTime} \approx -\frac{{\left(2  {\courantNumber}^{2} \indexTime {\relaxationParameterLetter_2} - 4  {\courantNumber}^{2} {\relaxationParameterLetter_2} + 2  {\courantNumber} \indexSpace {\relaxationParameterLetter_2} + 6  {\courantNumber}^{2} - {\courantNumber} {\relaxationParameterLetter_2} - 2  {\courantNumber} + 2  {\relaxationParameterLetter_2} - 4\right)} {\left(1-{\relaxationParameterLetter_2}\right)}^{\indexTime}}{{\relaxationParameterLetter_2}^{2}}, \qquad \text{if}\quad \relaxationParameterLetter_2 \in (0, 2),
\end{equation}
and an analogous but more involved expression for ${\momentDiscrete}_{2, \indexSpace}^{\indexTime}$, where a second-order polynomial in $\indexTime$ and $\indexSpace$ replaces the first-order one of \eqref{eq:theoreticalSolutionSigma3MoreTerms}, owing to the first-order pole of $\eigenvectorLetter_{\textnormal{s}, 2}(\timeShiftOperator)$ at $\timeShiftOperator=1-\relaxationParameterLetter_2$.
\Cref{fig:D1Q2-theory-vs-pract-difference-more-terms} shows the improvement from \eqref{eq:theoreticalSolutionSigma3MoreTerms} compared to the bottom row of \Cref{fig:D1Q2-theory-vs-pract-difference}.
Back to $\relaxationParameterLetter_2 = 2$, an even better way of understanding the propagation of the instability is to compute the (approximate) value of $\indexSpace$ of the vertex of the parabola for $\momentDiscrete_{2, \indexSpace}^{\indexTime}$, which is located at $\indexSpace\approx \courantNumber+1-\courantNumber\indexTime\approx -\courantNumber\indexTime$.
We do not explicitly give the long expressions that we obtained for $\orderExtrapolation = 4$ using symbolic computations: one can be easily persuaded that they give a second-order polynomial in $\indexTime$ and $\indexSpace$ for $\momentDiscrete_{1, \indexSpace}^{\indexTime}$ and a third-order polynomial in $\indexTime$ and $\indexSpace$ for $\momentDiscrete_{2, \indexSpace}^{\indexTime}$.

	\item We finish on \eqref{eq:kineticDirichlet}, which is SS through a control like \eqref{eq:tmp7} both for $\momentDiscrete_1$ and $\momentLetter_2$, as it does not admit any eigenvalue.

	\end{itemize}

\subparagraph{Kreiss-Lopatinskii determinants in the inflow case $\courantNumber>0$}

The only boundary condition to check is \eqref{eq:extrapolation} for $\orderExtrapolation = 1, 2, 3, 4$, as it features shared eigenvalues.
For all these boundary conditions, $\kreissLopatinskiiScalarProduct(\timeShiftOperator) =\frac{\courantNumber+1}{2\courantNumber^{\orderExtrapolation}}(\timeShiftOperator-1)^{\orderExtrapolation} + \bigO{(\timeShiftOperator-1)^{\orderExtrapolation + 1}}$  (\idEst{} \threeboxes{\continuousExtensionMark}{\eigenvalueMark}{\zeroKL}), which proves that they are not strongly stable, with instabilities of the same order on $\momentDiscrete_1$ and $\momentDiscrete_2$.
The description of the instabilities on \Cref{fig:D1Q2-s-3_2-C-1_2-revision} and \ref{fig:D1Q2-s-2-C-1_2-revision} from \eqref{eq:residualApprox} reads (we give only $\momentDiscrete_1$ since $\momentDiscrete_2$ is roughly $\courantNumber\times\momentDiscrete_1$) 
\begin{align}
	&\text{For }\eqref{eq:extrapolation}\text{ with }\orderExtrapolation = 1, \qquad \momentDiscrete_{1, \indexSpace}^{\indexTime} \approx \frac{2\courantNumber}{  {\left(\courantNumber + 1\right)}}.\label{eq:estimationFirstMomentCposSigma1}\\
	&\text{For }\eqref{eq:extrapolation}\text{ with }\orderExtrapolation = 2, \qquad \momentDiscrete_{1, \indexSpace}^{\indexTime} \approx\frac{2  {\left(\courantNumber^{2} \indexTime {\relaxationParameterLetter_2} - \courantNumber \indexSpace {\relaxationParameterLetter_2} - \courantNumber^{2} + \courantNumber {\relaxationParameterLetter_2} - \courantNumber - {\relaxationParameterLetter_2} + 2\right)}}{{\left(\courantNumber + 1\right)} {\relaxationParameterLetter_2}}.\nonumber
\end{align}
\begin{multline*}
	\text{For }\eqref{eq:extrapolation}\text{ with }\orderExtrapolation = 3, \qquad
	\momentDiscrete_{1, \indexSpace}^{\indexTime} \approx \frac{1}{{\left(\courantNumber + 1\right)} {\relaxationParameterLetter_2}^{2}} \Bigl ( \courantNumber^{3} \indexTime^{2} {\relaxationParameterLetter_2}^{2} + \courantNumber^{3} \indexTime {\relaxationParameterLetter_2}^{2} - 2  \courantNumber^{2} \indexSpace \indexTime {\relaxationParameterLetter_2}^{2} \\
	- 4  \courantNumber^{3} \indexTime {\relaxationParameterLetter_2} + \courantNumber \indexSpace^{2} {\relaxationParameterLetter_2}^{2} + 3  \courantNumber^{2} \indexTime {\relaxationParameterLetter_2}^{2} - 4  \courantNumber^{3} {\relaxationParameterLetter_2} + 2  \courantNumber^{2} \indexSpace {\relaxationParameterLetter_2} - 2  \courantNumber^{2} \indexTime {\relaxationParameterLetter_2} \\
	- 3  \courantNumber \indexSpace {\relaxationParameterLetter_2}^{2} - 3  \courantNumber \indexTime {\relaxationParameterLetter_2}^{2} + 6  \courantNumber^{3} - 5  \courantNumber^{2} {\relaxationParameterLetter_2} 
	+ 2  \courantNumber \indexSpace {\relaxationParameterLetter_2} + 6  \courantNumber \indexTime {\relaxationParameterLetter_2} + 3  \courantNumber {\relaxationParameterLetter_2}^{2} + 2  \indexSpace {\relaxationParameterLetter_2}^{2} \\
	+ 4  \courantNumber^{2} + \courantNumber {\relaxationParameterLetter_2} - 4  \indexSpace {\relaxationParameterLetter_2} - 3  {\relaxationParameterLetter_2}^{2} - 6  \courantNumber + 8  {\relaxationParameterLetter_2} - 4 \Bigr ).
\end{multline*}
Expressions for $\orderExtrapolation = 4$ are long but analogous.
It is now interesting to compare again \Cref{fig:D1Q2-s-3_2-C-1_2-revision} (\idEst{}, $\relaxationParameterLetter_2=\tfrac{3}{2}$) and \ref{fig:D1Q2-s-2-C-1_2-revision} (\idEst{}, $\relaxationParameterLetter_2=2$) concerning \eqref{eq:extrapolation} with $\orderExtrapolation = 1$ (we could do the same for \eqref{eq:extrapolatedEquilibrium} with $\orderExtrapolation = 1$), keeping \eqref{eq:estimationFirstMomentCposSigma1} in mind.
Indeed, the estimate from \eqref{eq:estimationFirstMomentCposSigma1} does not depend on the value of $\relaxationParameterLetter_2$, and the numerical results on \Cref{fig:D1Q2-s-3_2-C-1_2-revision} and \ref{fig:D1Q2-s-2-C-1_2-revision} differ only by the presence of wiggles for the latter, which are indeed the signature of branch points on $\unitCircle$ that we have neglected in the actual computation of the right-hand side of \eqref{eq:inverseZPractical}, \confer{} \eqref{eq:decreaseBranchPoint}. At fixed $\indexSpace$, the neglected terms when $\relaxationParameterLetter_2=2$ decrease like $\indexTime^{-3/2}$, whereas they trend like $\beta^{\indexTime}\indexTime^{-3/2}$ with $\beta\approx 0.7071$ when $\relaxationParameterLetter_2 = \tfrac{3}{2}$, explaining the (rapid) absence of oscillations in this latter case.
Analogous wiggles---also linked to branch points on $\unitCircle$---are those explaining the shape of the numerical solution on the first row of \Cref{fig:D1Q2-s-2-C-1_2-revision}, where boundary conditions are indeed strongly stable.

So far, we have not considered \eqref{eq:extrapolatedEquilibrium}, \eqref{eq:kinrod}, and \eqref{eq:godunovRyabenkii}, since correct symbolic computations of the solutions to \eqref{eq:toSolveBoundaryBulkD1Q2} are hard to carry without specifying the values of $\relaxationParameterLetter_2$ and $\courantNumber$. 
We consider them in the \Cref{app:moreBCD1Q2} for the values that we have utilized in \Cref{sec:numericalSimD1Q2}.

\begin{remark}[On the meaning of strong stability]\label{rem:remarkStrongStab}
	We now consider empirical observations to understand the meaning of \eqref{eq:strongStability}--\eqref{eq:strongStabilityObserved}.
	\begin{itemize}
		\item Let $\courantNumber<0$, $\relaxationParameterLetter_2 = \tfrac{2}{1-\courantNumber}$ and employ the bounce-back condition \eqref{eq:BB}, which is neither SS nor SSOO.
		We know that the solution is localized on the first cell of the domain only.
		As strong stability must hold for \strong{any} boundary datum, we take a specific one with a pole at $\timeShiftOperator = 1$, which thus excites the instability. 
		We neglect multiplicative non-zero constants and consider $\boundarySourceTerm_{1, -1}^{\indexTime} \sim 1$ for all $\indexTime\in\naturals$. 
		We hence obtain $\momentDiscrete_{1, \indexSpace}^{\indexTime} \sim \indexTime \delta_{0\indexSpace}$.
		Inspired by \eqref{eq:strongStabilityObserved}, we construct, for $\alpha>0$ and $\spaceStep>0$
		\begin{align*}
			\textnormal{output} &\definitionEquality \frac{\alpha}{1+\alpha\timeStep}  \sum_{\indexTime\in\naturals} \sum_{\indexSpace\in\naturals}\timeStep\spaceStep \, e^{-2\alpha \indexTime\timeStep} |{\momentDiscrete}_{1, \indexSpace}^{\indexTime}|^2 \sim \latticeVelocity\frac{\alpha\timeStep^2 e^{2\alpha \timeStep} (e^{2\alpha\timeStep}+1)}{(1+\alpha\timeStep)(e^{2\alpha\timeStep}-1)^3}, \\
			\textnormal{input} &\definitionEquality \sum_{\indexTime\in\naturals} \timeStep \, e^{-2\alpha \indexTime\timeStep}|{{\boundarySourceTerm}}_{1, -1}^{\indexTime}|^2\sim \frac{\timeStep e^{2\alpha\timeStep}}{e^{2\alpha\timeStep}-1}.
		\end{align*}
		The fact of being able to find a (universal) constant in \eqref{eq:strongStabilityObserved} boils down to be able to find a uniform control on the ratio between output and input:
		\begin{equation*}
			\frac{\textnormal{output}}{\textnormal{input}} \sim \latticeVelocity\frac{\alpha\timeStep (e^{2\alpha\timeStep} + 1)}{(1+\alpha\timeStep)(e^{2\alpha\timeStep}-1)^2} \xrightarrow[]{\alpha\timeStep\to 0^+} +\infty, 
		\end{equation*}
		hence we cannot find such a universal constant.
		\item Let $\courantNumber<0$, $\relaxationParameterLetter_2 = 2$ and use the extrapolation condition \eqref{eq:extrapolation} with $\orderExtrapolation = 1$, which is SSOO.
		Let us excite the critical eigenvalue: $\boundarySourceTerm_{1, -1}^{\indexTime} \sim (-1)^{\indexTime}$, from which we obtain a travelling solution $\momentDiscrete_{1, \indexSpace}^{\indexTime} \sim (-1)^{\indexTime} \indicatorFunction{\indexSpace-\courantNumber\indexTime\leq 0}$.
		This provides
		\begin{equation*}
			\textnormal{output} \sim \latticeVelocity\courantNumber\frac{\alpha\timeStep^2 e^{2\alpha \timeStep}}{(1+\alpha\timeStep)(e^{2\alpha\timeStep}-1)^2}, \qquad
			\textnormal{input} \sim \frac{\timeStep e^{2\alpha \timeStep}}{e^{2\alpha\timeStep}-1},
		\end{equation*}
		thus
		\begin{equation*}
			\frac{\textnormal{output}}{\textnormal{input}} \sim \latticeVelocity\courantNumber \frac{\alpha\timeStep }{(1+\alpha\timeStep)(e^{2\alpha\timeStep}-1)},
		\end{equation*}
		where the right-hand side is a uniformly bounded function of $\alpha\timeStep\in(0, +\infty)$, in particular when $\alpha\timeStep\to 0^+$.
	\end{itemize}
\end{remark}

\subsection{\lbmScheme{1}{3} scheme generalizing the Lax-Wendroff method}

We now consider the scheme introduced in \cite[Chapter 7]{bellotti2023numerical}:
\begin{equation}\label{eq:descriptionD1Q3LW}
	\numberVelocities = 3, \qquad \dimensionlessDiscreteVelocityLetter_1 = 0, \quad \dimensionlessDiscreteVelocityLetter_2 = -\dimensionlessDiscreteVelocityLetter_3 = 1, \qquad 
	\momentMatrix = 
	\begin{pmatrix}
		1 & 1 & 1\\
		0 & 1 & -1\\
		0 & 1 & 1
	\end{pmatrix}, \qquad 
	\equilibriumVectorLetter_2 = \courantNumber, \quad \equilibriumVectorLetter_3 = \courantNumber^2.
\end{equation}
This scheme is a generalization of the Lax-Wendroff scheme, in the sense that whenever $\relaxationParameterLetter_2 = \relaxationParameterLetter_3 = 1$, we recover this very scheme on the conserved moment $\momentDiscrete_1$.
This scheme is second-order accurate.
In what follows, we always consider $\courantNumber\in (-1, 0)\cup (0, 1)$ and $\relaxationParameterLetter_2, \relaxationParameterLetter_3 \in (0, 2)$, in order to avoid dealing with special cases on the boundary of this domain, which are neither particularly enlightening, nor useful in practice. 

\begin{figure}
	\begin{center}
		\includegraphics{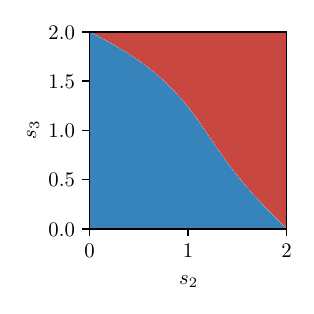}
	\end{center}
	\caption{\label{fig:D1Q3LWStabilityArea}Stability area (in blue) for the \lbmScheme{1}{3} \strong{à la} Lax-Wendroff according to \eqref{eq:stabilityD1Q3LW}.}
\end{figure}

\begin{proposition}[Stability of the \lbmScheme{1}{3} \emph{à la} Lax-Wendroff scheme for the Cauchy problem]\label{prop:vonNeumannStabilityD1Q3LW}
	Assume that $\courantNumber\in (-1, 0)\cup (0, 1)$ and $\relaxationParameterLetter_2, \relaxationParameterLetter_3 \in (0, 2)$.
	The \lbmScheme{1}{3} \emph{à la} Lax-Wendroff is stable according to \Cref{def:vonNeumannStability}, \ref{def:stabLBM}, and \ref{def:stabFD} if and only if 
	\begin{equation}\label{eq:stabilityD1Q3LW}
		\relaxationParameterLetter_3\leq \relaxationParameterLetter_3^{\star}(\relaxationParameterLetter_2), \qquad \text{with}\qquad 
		\relaxationParameterLetter_3^{\star}(\relaxationParameterLetter)\definitionEquality 
		\begin{cases}
			1+\frac{\sqrt{(\relaxationParameterLetter^2 - 8\relaxationParameterLetter + 8)^2 + 4\relaxationParameterLetter^2(2-\relaxationParameterLetter)^2}-2\relaxationParameterLetter (2-\relaxationParameterLetter)}{\relaxationParameterLetter^2 - 8\relaxationParameterLetter + 8}, \qquad &\text{if}\quad \relaxationParameterLetter \neq 3-2\sqrt{2}, \\
			1, \qquad &\text{if}\quad \relaxationParameterLetter = 3-2\sqrt{2}.
		\end{cases}
	\end{equation}
\end{proposition}
\Cref{prop:vonNeumannStabilityD1Q3LW} is proved in \cite{bcdg25} and the stability area given by \eqref{eq:stabilityD1Q3LW} depicted in \Cref{fig:D1Q3LWStabilityArea}. 
We consider the following boundary conditions
\begin{align}
	\distributionFunctionDiscrete_{2, -1}^{\indexTime\collided} = \distributionFunctionDiscrete_{3, 0}^{\indexTime\collided} + \boundarySourceTerm_{2, -1}^{\indexTime}&\qquad \text{(bounce-back)}, \label{eq:BBD1Q3LW}\\
	\distributionFunctionDiscrete_{2, -1}^{\indexTime\collided} = -\distributionFunctionDiscrete_{3, 0}^{\indexTime\collided} + \boundarySourceTerm_{2, -1}^{\indexTime}&\qquad \text{(anti-bounce-back)}, \label{eq:ABBD1Q3LW}\\
	\distributionFunctionDiscrete_{2, -1}^{\indexTime\collided} = -\distributionFunctionDiscrete_{1, 0}^{\indexTime\collided} -\distributionFunctionDiscrete_{3, 1}^{\indexTime\collided} + \boundarySourceTerm_{2, -1}^{\indexTime}&\qquad \text{(two-steps anti-bounce-back)}, \label{eq:TwoABBD1Q3LW}\\
	\distributionFunctionDiscrete_{2, -1}^{\indexTime\collided} = \sum_{\indexSpace=0}^{\orderExtrapolation-1}(-1)^{\indexSpace}\binom{\orderExtrapolation}{\indexSpace + 1} \distributionFunctionDiscrete_{2, \indexSpace}^{\indexTime\collided} + \boundarySourceTerm_{2, -1}^{\indexTime}&\qquad \text{(extrapolation of order }\orderExtrapolation\geq 1\text{)}, \label{eq:extrapolationD1Q3LW} \\
	\distributionFunctionDiscrete_{2, -1}^{\indexTime\collided} = \boundarySourceTerm_{2, -1}^{\indexTime}&\qquad \text{(kinetic Dirichlet)}. \label{eq:kineticDirichletD1Q3LW}
\end{align}
Compared to the conditions introduced in \Cref{sec:D1Q2}, only \eqref{eq:TwoABBD1Q3LW} deserves further discussions.
In the homogeneous case, $\momentDiscrete_{1, 0}^{\indexTime+1} = \distributionFunctionDiscrete_{1, 0}^{\indexTime + 1} + \distributionFunctionDiscrete_{2, 0}^{\indexTime + 1} + \distributionFunctionDiscrete_{3, 0}^{\indexTime + 1} = \distributionFunctionDiscrete_{1, 0}^{\indexTime\collided} + \distributionFunctionDiscrete_{2, -1}^{\indexTime\collided} + \distributionFunctionDiscrete_{3, 1}^{\indexTime \collided} = 0$.
From this identity, we obtain \eqref{eq:TwoABBD1Q3LW} up to adding a source term.

\begin{theorem}[Strong stability--instability of the boundary-value \lbmScheme{1}{3} scheme \emph{à la} Lax-Wendroff]\label{thm:stabD1Q3BoundaryConditions}
	Under the stability conditions given in \Cref{prop:vonNeumannStabilityD1Q3LW}, the stability of the boundary-value \lbmScheme{1}{3} \emph{à la} Lax-Wendroff is as follows.
	\begin{center}
		\begin{tabular}{|c||c|c|}
			\cline{2-3}
			\multicolumn{1}{c||}{} & $\courantNumber<0$ (outflow) & $\courantNumber>0$ (inflow)\\
			\cline{2-3}
			\multicolumn{1}{c||}{} & $\relaxationParameterLetter_2, \relaxationParameterLetter_3\in(0, 2)$ & $\relaxationParameterLetter_2, \relaxationParameterLetter_3\in(0, 2)$ \\
			\hline
			\hline
			Bounce-back \eqref{eq:BBD1Q3LW} & MU-L & SS\\
			Anti-bounce-back \eqref{eq:ABBD1Q3LW} &SS & SS\\
			Two-steps anti-bounce-back \eqref{eq:TwoABBD1Q3LW} & SS & SS \\
			Extrapolation $1\leq \orderExtrapolation \leq 4$ \eqref{eq:extrapolationD1Q3LW} & SS & MU-E\\
			Kinetic Dirichlet \eqref{eq:kineticDirichletD1Q3LW} & SS & SS\\
			\hline
		\end{tabular}
	\end{center}
\end{theorem}
These results of stability--instability are similar to those from \Cref{thm:stabilityBoundaryD1Q2} with $\relaxationParameterLetter_2 \in (0, 2)$, which concern a diffusive scheme---whereas the scheme is now a dispersive one.

\subsubsection{Proof of \Cref{thm:stabD1Q3BoundaryConditions}}

\paragraph{Description of the roots of the characteristic equation}

The characteristic equation \eqref{eq:characteristicEquation} has $\stencilLeftCharacteristic = \stencilRightCharacteristic = 1$ with 
\begin{align*}
	\coefficientCharEquationInFourier_{-1}(\timeShiftOperator) &= -\tfrac{1}{2} \, {\left({\left(\courantNumber - 1\right)} {\relaxationParameterLetter_2} + {\left(\courantNumber^{2} - 1\right)} {\relaxationParameterLetter_3} + 2\right)} {\timeShiftOperator}^{2} + \tfrac{1}{2} \, {\left({\left(\courantNumber - 1\right)} {\relaxationParameterLetter_2} - {\left(\courantNumber^{2} - {\left(\courantNumber^{2} - \courantNumber\right)} {\relaxationParameterLetter_2} + 1\right)} {\relaxationParameterLetter_3} + 2\right)} {\timeShiftOperator}, \\ 
	\coefficientCharEquationInFourier_{0}(\timeShiftOperator) &= {\left(\courantNumber^{2} {\relaxationParameterLetter_3} - 1\right)} {\timeShiftOperator}^{2} + {\timeShiftOperator}^{3} - {\left({\relaxationParameterLetter_2} - 1\right)} {\relaxationParameterLetter_3} + {\left({\left(\courantNumber^{2} - {\left(\courantNumber^{2} - 1\right)} {\relaxationParameterLetter_2} - 1\right)} {\relaxationParameterLetter_3} - {\relaxationParameterLetter_2} + 1\right)} {\timeShiftOperator} + {\relaxationParameterLetter_2} - 1\\
	\coefficientCharEquationInFourier_{1}(\timeShiftOperator) &= \tfrac{1}{2} \, {\left({\left(\courantNumber + 1\right)} {\relaxationParameterLetter_2} - {\left(\courantNumber^{2} - 1\right)} {\relaxationParameterLetter_3} - 2\right)} {\timeShiftOperator}^{2} - \tfrac{1}{2} \, {\left({\left(\courantNumber + 1\right)} {\relaxationParameterLetter_2} + {\left(\courantNumber^{2} - {\left(\courantNumber^{2} + \courantNumber\right)} {\relaxationParameterLetter_2} + 1\right)} {\relaxationParameterLetter_3} - 2\right)} {\timeShiftOperator}.
\end{align*}
The coefficients $\coefficientCharEquationInFourier_{-1}(\timeShiftOperator)$ and $\coefficientCharEquationInFourier_{1}(\timeShiftOperator)$ do not identically vanish for certain choices of the parameters (unlike the \lbmScheme{1}{2} scheme), but only for specific $\timeShiftOperator$.
The product $\productRootsDOneQThree$ of the roots, unlike for the \lbmScheme{1}{2} scheme, is not independent of $\timeShiftOperator$.

\begin{proposition}
	Let the conditions by \Cref{prop:vonNeumannStabilityD1Q3LW} be fulfilled.
	Then for $\timeShiftOperator\in\neighborhoodInfinity$, the characteristic equation \eqref{eq:characteristicEquation} for the \lbmScheme{1}{3} \strong{à la} Lax-Wendroff scheme has two roots (except at points where $\coefficientCharEquationInFourier_{-1}(\timeShiftOperator)$ or $\coefficientCharEquationInFourier_{1}(\timeShiftOperator)$ vanish), one stable root $\stableRoot(\timeShiftOperator)\in\unitDisk$, and one unstable root $\unstableRoot(\timeShiftOperator)\in\neighborhoodInfinity$.
	They can be extended to the unit circle, \idEst{} to $\timeShiftOperator\in\unitCircle$, and this extension is still denoted $\stableRoot(\timeShiftOperator)$ and $\unstableRoot(\timeShiftOperator)$, with
	\begin{equation*}
		\stableRoot(1) = 
		\begin{cases}
			\productRootsDOneQThree|_{\timeShiftOperator = 1}  = \frac{2\courantNumber-(\courantNumber-1)\relaxationParameterLetter_2}{2\courantNumber-(\courantNumber+1)\relaxationParameterLetter_2}, \qquad &\text{if}\quad \courantNumber<0, \\
			1 \qquad &\text{if}\quad \courantNumber>0.
		\end{cases}\qquad 
		\unstableRoot(1) = 
		\begin{cases}
			1, \qquad &\text{if}\quad \courantNumber<0, \\
			\productRootsDOneQThree|_{\timeShiftOperator = 1}  = \frac{2\courantNumber-(\courantNumber-1)\relaxationParameterLetter_2}{2\courantNumber-(\courantNumber+1)\relaxationParameterLetter_2}\qquad &\text{if}\quad \courantNumber>0.
		\end{cases}
	\end{equation*}
\end{proposition}
\begin{proof}
	The first part of the claim is the Hersh lemma.
	The second part is conducted by usual perturbation techniques: let $\timeShiftOperator = 1+\delta\timeShiftOperator$ and $\fourierShift = 1+\delta\fourierShift$. Into \eqref{eq:characteristicEquation}, this gives, at leading order $\delta\fourierShift = -1/\courantNumber\delta\timeShiftOperator$, hence the claim.
\end{proof}

\paragraph{Shared eigenvalues}

From the experience we have gained in \Cref{sec:D1Q2}, we avoid computing roots with $\timeShiftOperator\in \puncturedPlane$ but $\fourierShift=0$, and directly consider 
\begin{equation}\label{eq:eigenvalueProblemsExcludingKappaZero}
	\determinant(\timeShiftOperator\identityMatrix{\numberVelocities} - \schemeMatrixBulkFourier(\fourierShift)) = 0\qquad \text{and} \qquad \determinant\Bigl ( \timeShiftOperator\identityMatrix{\numberVelocities} - \sum_{\indexFreeOne = 0}^{\stencilBoundaryCondition}\schemeMatrixBoundaryByPower{0}{\indexFreeOne}\fourierShift^{\indexFreeOne}\Bigr ) = 0
\end{equation}
instead of \eqref{eq:toSolveBoundaryBulkD1Q2}.
The candidates are as follows.
\begin{align*}
	\text{For }\eqref{eq:BBD1Q3LW}\qquad (\timeShiftOperator, \fourierShift) =  (\relaxationParameterLetter_2-1,-1), \quad (\timeShiftOperator, \fourierShift) &= (1-\relaxationParameterLetter_3, 1),\quad (\timeShiftOperator, \fourierShift) = (1, \tfrac{2\courantNumber-(\courantNumber-1)\relaxationParameterLetter_2}{2\courantNumber-(\courantNumber+1)\relaxationParameterLetter_2}),\\
	(\timeShiftOperator, \fourierShift) &= (\tfrac{\relaxationParameterLetter_2 + (\courantNumber\relaxationParameterLetter_2-\courantNumber-1)\relaxationParameterLetter_3}{\relaxationParameterLetter_2 - (\courantNumber+1)\relaxationParameterLetter_3}, \tfrac{\relaxationParameterLetter_2 + (\courantNumber\relaxationParameterLetter_2-\courantNumber-1)\relaxationParameterLetter_3}{(\relaxationParameterLetter_2 - (\courantNumber+1) \relaxationParameterLetter_3) (1-\relaxationParameterLetter_2)} ).
\end{align*}
The first two couples are eigenvalues only if $\relaxationParameterLetter_2 = 2$ or $\relaxationParameterLetter_3 = 2$, which we excluded.
The third couple is an eigenvalue, regardless of the relaxation parameters, when $\courantNumber<0$, with $\fourierShift\in\unitDisk$.
For the fourth couple, we have that $|\fourierShift| = {|\timeShiftOperator|}/{|1-\relaxationParameterLetter_2|}$.
For $\timeShiftOperator\in\neighborhoodInfinity$, we have that $\fourierShift\in\neighborhoodInfinity$ as well for $\relaxationParameterLetter_2\in (0, 2)$, hence $\fourierShift = \unstableRoot$: no Godunov-Ryabenkii instability can happen.
Finally, let $\timeShiftOperator\in\unitCircle$, by a continuity argument, we deduce that $\fourierShift = \unstableRoot$, and this couple is not an eigenvalue.

Let us now discuss \eqref{eq:ABBD1Q3LW}: we first solve the second equation in \eqref{eq:eigenvalueProblemsExcludingKappaZero}, which is first-order in $\fourierShift$---in $\fourierShift$. We plug this result into the first equation of \eqref{eq:eigenvalueProblemsExcludingKappaZero} to yield the corresponding $\timeShiftOperator$.
We have
\begin{align*}
	(\timeShiftOperator, \fourierShift) = (1-\relaxationParameterLetter_2, 1), \qquad (\timeShiftOperator, \fourierShift) &=(\tfrac{\relaxationParameterLetter_2 + (\courantNumber\relaxationParameterLetter_2-\courantNumber-1)\relaxationParameterLetter_3}{\relaxationParameterLetter_2 - (\courantNumber+1)\relaxationParameterLetter_3}, \tfrac{\relaxationParameterLetter_2 + (\courantNumber\relaxationParameterLetter_2-\courantNumber-1)\relaxationParameterLetter_3}{(\relaxationParameterLetter_2 - (\courantNumber+1) \relaxationParameterLetter_3) (1-\relaxationParameterLetter_2)} ), \\
	\timeShiftOperator &= -\tfrac{1}{2}(2\courantNumber^2-1)\relaxationParameterLetter_3 \pm \tfrac{1}{2}\sqrt{(4\courantNumber^4 - 4\courantNumber^2+1)\relaxationParameterLetter_3^2 - 4\relaxationParameterLetter_3+4},
\end{align*}
where we do not even list the corresponding $\fourierShift$ for the last couple of complex conjugate $\timeShiftOperator$'s, since we want to study whether these $\timeShiftOperator$'s can be critical, \idEst{} in $\closedNeighborhoodInfinity$.
The monic polynomial whose roots are these $\timeShiftOperator$'s is 
\begin{equation*}
	\phi_2(\timeShiftOperator) = \timeShiftOperator^2 + (2\courantNumber^2 - 1)\relaxationParameterLetter_3 \timeShiftOperator + \relaxationParameterLetter_3 - 1.
\end{equation*}
We first use \cite[Theorem 4.3.1]{strikwerda2004finite} to study the position of these roots with respect to $\unitCircle$.
We have $\phi_2^{\star}(\timeShiftOperator) =  \timeShiftOperator^2 \overline{\phi_2(1/\overline{\timeShiftOperator})}= (\relaxationParameterLetter_3 - 1)\timeShiftOperator^2 + (2\courantNumber^2 - 1)\relaxationParameterLetter_3 \timeShiftOperator + 1$, and request the condition $|\phi_2(0)|<|\phi_2^{\star}(0)|$, equivalent to $\relaxationParameterLetter_3\in(0, 2)$, which we assumed from the very beginning.
Computing $\phi_1(\timeShiftOperator) =\timeShiftOperator^{-1} (\phi_2^{\star}(0)\phi_2(\timeShiftOperator)-\phi_2(0)\phi_2^{\star}(\timeShiftOperator))$, this polynomial has root $1-2\courantNumber^2$. We thus request that  $|1-2\courantNumber^2|<1$, which is equivalent to $|\courantNumber|<1$, which we also assumed from the very beginning.
We conclude using \cite[Theorem 4.3.1]{strikwerda2004finite} that the two $\timeShiftOperator$'s under discussion belong to $\unitDisk$, hence are not eigenvalues.

\begin{align*}
	\text{For }\eqref{eq:TwoABBD1Q3LW}\qquad (\timeShiftOperator, \fourierShift) =  (\relaxationParameterLetter_2-1,-1), \quad (\timeShiftOperator, \fourierShift) =  (1-\relaxationParameterLetter_2,1), \quad (\timeShiftOperator, \fourierShift) &= (1-\relaxationParameterLetter_3, 1),\\
	(\timeShiftOperator, \fourierShift) &= (\tfrac{\relaxationParameterLetter_2 + (\courantNumber\relaxationParameterLetter_2-\courantNumber-1)\relaxationParameterLetter_3}{\relaxationParameterLetter_2 - (\courantNumber+1)\relaxationParameterLetter_3}, \tfrac{\relaxationParameterLetter_2 + (\courantNumber\relaxationParameterLetter_2-\courantNumber-1)\relaxationParameterLetter_3}{(\relaxationParameterLetter_2 - (\courantNumber+1) \relaxationParameterLetter_3) (1-\relaxationParameterLetter_2)} ),
\end{align*}
where none are eigenvalues under the current assumptions.
\begin{align*}
	\text{For }\eqref{eq:extrapolationD1Q3LW}, \quad \orderExtrapolation = 1, 2, 3, 4\qquad(\timeShiftOperator, \fourierShift) &= (1, 1), \quad (\timeShiftOperator, \fourierShift) =  (1-\relaxationParameterLetter_2,1), \quad (\timeShiftOperator, \fourierShift) = (1-\relaxationParameterLetter_3, 1),\\
	(\timeShiftOperator, \fourierShift) &= (\tfrac{\relaxationParameterLetter_2 + (\courantNumber\relaxationParameterLetter_2-\courantNumber-1)\relaxationParameterLetter_3}{\relaxationParameterLetter_2 - (\courantNumber+1)\relaxationParameterLetter_3}, \tfrac{\relaxationParameterLetter_2 + (\courantNumber\relaxationParameterLetter_2-\courantNumber-1)\relaxationParameterLetter_3}{(\relaxationParameterLetter_2 - (\courantNumber+1) \relaxationParameterLetter_3) (1-\relaxationParameterLetter_2)} ), \\
	(\timeShiftOperator, \fourierShift) &= (\tfrac{\relaxationParameterLetter_2 - (\courantNumber\relaxationParameterLetter_2-\courantNumber+1)\relaxationParameterLetter_3}{\relaxationParameterLetter_2 + (\courantNumber-1)\relaxationParameterLetter_3}, -\tfrac{{\left(\courantNumber^{2} {\relaxationParameterLetter_2} - \courantNumber^{2} + \courantNumber\right)} {\relaxationParameterLetter_3}^{2} - {\relaxationParameterLetter_2}^{2} + {\left(\courantNumber {\relaxationParameterLetter_2}^{2} - {\left(2 \, \courantNumber - 1\right)} {\relaxationParameterLetter_2}\right)} {\relaxationParameterLetter_3}}{{\left(\courantNumber^{2} - {\left(\courantNumber^{2} - 1\right)} {\relaxationParameterLetter_2} - \courantNumber\right)} {\relaxationParameterLetter_3}^{2} + {\relaxationParameterLetter_2}^{2} - {\left({\left(\courantNumber + 1\right)} {\relaxationParameterLetter_2}^{2} - {\left(2 \, \courantNumber - 1\right)} {\relaxationParameterLetter_2}\right)} {\relaxationParameterLetter_3}} ).
\end{align*}
We already know that the first couple is an eigenvalue when $\courantNumber>0$.
We are left to analyze the last couple, for which we have 
\begin{equation*}
	\frac{|\fourierShift|}{|\timeShiftOperator|} = \frac{|\relaxationParameterLetter_2 + \courantNumber\relaxationParameterLetter_3|}{|\relaxationParameterLetter_2 +\courantNumber\relaxationParameterLetter_3- (\courantNumber+1)\relaxationParameterLetter_2 \relaxationParameterLetter_3|}.
\end{equation*}
When $\courantNumber\in (0, 1)$, we are going to see that $|\fourierShift|/|\timeShiftOperator|>1$, which allows concluding that $\fourierShift = \unstableRoot$, hence these are not eigenvalues.
To this end, the target inequality is 
\begin{equation}\label{eq:tmp11}
	\left |1 - \frac{(\courantNumber+1)\relaxationParameterLetter_2 \relaxationParameterLetter_3}{\relaxationParameterLetter_2 + \courantNumber\relaxationParameterLetter_3}\right |<1 \qquad \Longleftrightarrow \qquad 0 < \psi(\relaxationParameterLetter_2, \relaxationParameterLetter_3, \courantNumber)\definitionEquality\frac{(\courantNumber+1)\relaxationParameterLetter_2 \relaxationParameterLetter_3}{\relaxationParameterLetter_2 + \courantNumber\relaxationParameterLetter_3} < 2.
\end{equation}
The numerator in $\psi$ is always positive.
The denominator may change sign when $\courantNumber<0$. 
For $\courantNumber\in (0, 1)$, the lower bound in \eqref{eq:tmp11} is trivially fulfilled. 
For the upper one, we have 
\begin{equation*}
	\partial_{\relaxationParameterLetter_2}\psi(\relaxationParameterLetter_2, \relaxationParameterLetter_3, \courantNumber) = \frac{\courantNumber(\courantNumber+1)\relaxationParameterLetter_2 \relaxationParameterLetter_3^2}{(\relaxationParameterLetter_2 + \courantNumber\relaxationParameterLetter_3)^2}>0, \qquad \text{thus}\qquad \psi(\relaxationParameterLetter_2, \relaxationParameterLetter_3, \courantNumber)< \psi(2, \relaxationParameterLetter_3, \courantNumber) = \frac{2(\courantNumber+1) \relaxationParameterLetter_3}{2 + \courantNumber\relaxationParameterLetter_3}, 
\end{equation*}
for the considered parameters.
Then, we have
\begin{equation*}
	\partial_{\relaxationParameterLetter_3}\psi(2, \relaxationParameterLetter_3, \courantNumber)  = \frac{4(\courantNumber+1) }{(2 + \courantNumber\relaxationParameterLetter_3)^2}>0, \qquad \text{thus}\qquad \psi(\relaxationParameterLetter_2, \relaxationParameterLetter_3, \courantNumber)< \psi(2, \relaxationParameterLetter_3, \courantNumber) = \frac{2(\courantNumber+1) \relaxationParameterLetter_3}{2 + \courantNumber\relaxationParameterLetter_3}<\psi(2, 2, \courantNumber)  = 2,
\end{equation*}
and allows concluding for $\courantNumber\in(0, 1)$.
Whenever $\courantNumber\in (-1, 0)$, it is false that $|\fourierShift|/|\timeShiftOperator|-1$ has a well-determined sign.
In this case the strategy of the proof is different: we show that when $\timeShiftOperator\in\neighborhoodInfinity$, then also $|\fourierShift|/|\timeShiftOperator|>1$, which entails $\fourierShift\in\neighborhoodInfinity$.
This is done by noting (and eventually proving) that in the $(\relaxationParameterLetter_2, \relaxationParameterLetter_3)$-plane restricted to $(0, 2)^2$, the boundary of $\{(\relaxationParameterLetter_2, \relaxationParameterLetter_3)\,:\,\timeShiftOperator\in\neighborhoodInfinity\}$ is strictly convex between the points $(0, 0)$ and $(2, 2)$, whereas the boundary of $\{(\relaxationParameterLetter_2, \relaxationParameterLetter_3)\,:\,|\fourierShift|/|\timeShiftOperator|>1\}$  is strictly concave between the points $(0, 0)$ and $(2, 2)$. This would prove that $\{(\relaxationParameterLetter_2, \relaxationParameterLetter_3)\,:\,\timeShiftOperator\in\neighborhoodInfinity\}\subset \{(\relaxationParameterLetter_2, \relaxationParameterLetter_3)\,:\,|\fourierShift|/|\timeShiftOperator|>1\}$.
Let us find a parametrization of these boundaries, starting from that of $\{(\relaxationParameterLetter_2, \relaxationParameterLetter_3)\,:\,\timeShiftOperator\in\neighborhoodInfinity\}$:
\begin{equation*}
	\timeShiftOperator\in\unitCircle \qquad \Longleftrightarrow\qquad \left | \frac{\relaxationParameterLetter_2 - (\courantNumber\relaxationParameterLetter_2-\courantNumber+1)\relaxationParameterLetter_3}{\relaxationParameterLetter_2 + (\courantNumber-1)\relaxationParameterLetter_3} \right | = 1\qquad \Longleftrightarrow\qquad  \frac{\relaxationParameterLetter_2 - (\courantNumber\relaxationParameterLetter_2-\courantNumber+1)\relaxationParameterLetter_3}{\relaxationParameterLetter_2 + (\courantNumber-1)\relaxationParameterLetter_3}  = \pm 1,
\end{equation*}
since $\timeShiftOperator\in\reals$.
There is a solution for the minus sign, which yields
\begin{equation*}
	\relaxationParameterLetter_3 = \relaxationParameterLetter_3^{\cup}(\relaxationParameterLetter_2) \definitionEquality \frac{2\relaxationParameterLetter_2}{\courantNumber\relaxationParameterLetter_2 - 2\courantNumber + 2}, \qquad \text{hence also}\qquad (\relaxationParameterLetter_3^{\cup})''(\relaxationParameterLetter_2) = \frac{8\courantNumber(\courantNumber-1)}{(\courantNumber\relaxationParameterLetter_2 - 2\courantNumber + 2)^3}.
\end{equation*}
Considering that $\courantNumber(\courantNumber-1)>0$, we look for a upper bound for $\courantNumber\relaxationParameterLetter_2 - 2\courantNumber + 2$.
This function is decreasing in $\relaxationParameterLetter_2$, hence $\courantNumber\relaxationParameterLetter_2 - 2\courantNumber + 2\leq 2(1-\courantNumber)$, giving $(\relaxationParameterLetter_3^{\cup})''(\relaxationParameterLetter_2)\geq {-4\courantNumber}/{(1-\courantNumber)^2}>0$, and shows that the function $\relaxationParameterLetter_3^{\cup}$ is strictly convex. 
Moreover, $\relaxationParameterLetter_3^{\cup}(0)=0$ and $\relaxationParameterLetter_3^{\cup}(2)=2$.
For the boundary of $\{(\relaxationParameterLetter_2, \relaxationParameterLetter_3)\,:\,|\fourierShift|/|\timeShiftOperator|>1\}$, we have 
\begin{equation*}
	|\fourierShift|/|\timeShiftOperator|=1 \qquad \Longleftrightarrow\qquad \left |\frac{\relaxationParameterLetter_2 + \courantNumber\relaxationParameterLetter_3}{\relaxationParameterLetter_2 +\courantNumber\relaxationParameterLetter_3- (\courantNumber+1)\relaxationParameterLetter_2 \relaxationParameterLetter_3} \right | = 1\qquad \Longleftrightarrow\qquad  \frac{\relaxationParameterLetter_2 + \courantNumber\relaxationParameterLetter_3}{\relaxationParameterLetter_2 +\courantNumber\relaxationParameterLetter_3- (\courantNumber+1)\relaxationParameterLetter_2 \relaxationParameterLetter_3} = \pm 1,
\end{equation*}
which solution is 
\begin{equation*}
	\relaxationParameterLetter_3 = \relaxationParameterLetter_3^{\cap}(\relaxationParameterLetter_2) \definitionEquality \frac{2\relaxationParameterLetter_2}{(\courantNumber+1)\relaxationParameterLetter_2 - 2\courantNumber}, \qquad \text{hence also}\qquad (\relaxationParameterLetter_3^{\cap})''(\relaxationParameterLetter_2) = \frac{8\courantNumber(\courantNumber+1)}{((\courantNumber+1)\relaxationParameterLetter_2 - 2\courantNumber)^3}.
\end{equation*}
We have that $\courantNumber(\courantNumber+1)<0$.
Therefore, we look for an upper bound of $(\courantNumber+1)\relaxationParameterLetter_2 - 2\courantNumber$.
Since $(\courantNumber+1)\relaxationParameterLetter_2 - 2\courantNumber$ is increasing in $\relaxationParameterLetter_2$, we have $(\courantNumber+1)\relaxationParameterLetter_2 - 2\courantNumber\leq 2$.
This proves that $(\relaxationParameterLetter_3^{\cap})''(\relaxationParameterLetter_2) \leq 4\courantNumber(\courantNumber+1) < 0$, thus shows that $\relaxationParameterLetter_3^{\cap}$ is strictly concave.
Finally, $\relaxationParameterLetter_3^{\cap}(0) = 0$ and $\relaxationParameterLetter_3^{\cap}(2) = 2$.

Overall, this gives that when $\timeShiftOperator\in\neighborhoodInfinity$, then $\fourierShift\in\neighborhoodInfinity$, hence corresponds to $\unstableRoot$. By continuity of the roots, we deduce that these roots are not eigenvalues even when on $\unitCircle$.

\begin{align*}
	\text{For }\eqref{eq:kineticDirichletD1Q3LW} \qquad
	(\timeShiftOperator, \fourierShift) &= (\tfrac{\relaxationParameterLetter_2 + (\courantNumber\relaxationParameterLetter_2-\courantNumber-1)\relaxationParameterLetter_3}{\relaxationParameterLetter_2 - (\courantNumber+1)\relaxationParameterLetter_3}, \tfrac{\relaxationParameterLetter_2 + (\courantNumber\relaxationParameterLetter_2-\courantNumber-1)\relaxationParameterLetter_3}{(\relaxationParameterLetter_2 - (\courantNumber+1) \relaxationParameterLetter_3) (1-\relaxationParameterLetter_2)} ), \\
	(\timeShiftOperator, \fourierShift) &= (\tfrac{\relaxationParameterLetter_2 - (\courantNumber\relaxationParameterLetter_2-\courantNumber+1)\relaxationParameterLetter_3}{\relaxationParameterLetter_2 + (\courantNumber-1)\relaxationParameterLetter_3}, -\tfrac{{\left(\courantNumber^{2} {\relaxationParameterLetter_2} - \courantNumber^{2} + \courantNumber\right)} {\relaxationParameterLetter_3}^{2} - {\relaxationParameterLetter_2}^{2} + {\left(\courantNumber {\relaxationParameterLetter_2}^{2} - {\left(2 \, \courantNumber - 1\right)} {\relaxationParameterLetter_2}\right)} {\relaxationParameterLetter_3}}{{\left(\courantNumber^{2} - {\left(\courantNumber^{2} - 1\right)} {\relaxationParameterLetter_2} - \courantNumber\right)} {\relaxationParameterLetter_3}^{2} + {\relaxationParameterLetter_2}^{2} - {\left({\left(\courantNumber + 1\right)} {\relaxationParameterLetter_2}^{2} - {\left(2 \, \courantNumber - 1\right)} {\relaxationParameterLetter_2}\right)} {\relaxationParameterLetter_3}} ),
\end{align*}
thus we again conclude that there is no harmful eigenvalue in this case.
The overall picture is as resumed in \Cref{tab:modesD1Q3LW}.

\begin{table}\caption{\label{tab:modesD1Q3LW}Unstable candidate eigenvalues (shared between bulk and boundary scheme) for the \lbmScheme{1}{3} scheme \strong{à la} Lax-Wendroff.}
	\begin{center}
		\begin{tabular}{|c||c|c|}
			\cline{2-3}
			\multicolumn{1}{c||}{} & $\courantNumber<0$ (outflow) & $\courantNumber>0$ (inflow)\\
			\cline{2-3}
			\multicolumn{1}{c||}{} & $\relaxationParameterLetter_2\in(0, 2)$  & $\relaxationParameterLetter_2\in(0, 2)$ \\
			\hline
			\hline
			Bounce-back \eqref{eq:BBD1Q3LW} & $(1, \tfrac{2\courantNumber-(\courantNumber-1)\relaxationParameterLetter_2}{2\courantNumber-(\courantNumber+1)\relaxationParameterLetter_2})$  & none \\
			Anti-bounce-back \eqref{eq:ABBD1Q3LW} & none &none \\
			Two-steps anti-bounce-back \eqref{eq:TwoABBD1Q3LW} & none &  none \\
			Extrapolation $1\leq \orderExtrapolation \leq 4$ \eqref{eq:extrapolationD1Q3LW} & none &  $(1, 1)$ \\
			Kinetic Dirichlet \eqref{eq:kineticDirichletD1Q3LW} & none &  none \\
			\hline
		\end{tabular}
	\end{center}
\end{table}

\paragraph{Eigenvectors: conclusion of the proof}
We can readily apply \Cref{prop:noDiracBoundary}.
For the considered parameters, $\vectorial{\eigenvectorLetter}_{\textnormal{s}}(\timeShiftOperator)$ can be continuously extended to $\unitCircle$.
Let $(\timeShiftOperator, \fourierShift) = (1, \productRootsDOneQThree|_{\timeShiftOperator = 1})$, which must be studied for \eqref{eq:BBD1Q3LW} with $\courantNumber<0$.
	We see that $\vectorial{\eigenvectorLetter}_{\textnormal{s}}(1)$ can be computed:
	\begin{equation*}
		\vectorial{\eigenvectorLetter}_{\textnormal{s}}(1) = \transpose{\Bigl (1, \frac{\courantNumber\relaxationParameterLetter_2}{\relaxationParameterLetter_2 - 2}, \courantNumber^2 \Bigr )}.\qquad \text{Moreover, }\qquad 
		\kreissLopatinskiiDet(1) = 0,
	\end{equation*}
	thus \threeboxes{\continuousExtensionMark}{\eigenvalueMark}{\zeroKL}: the boundary condition is not strongly stable.

Let $(\timeShiftOperator, \fourierShift) = (1, 1)$, which must be studied for \eqref{eq:extrapolationD1Q3LW} with $\courantNumber>0$, computations are analogous to those of \Cref{sec:D1Q2}, and the boundary condition is not strongly stable.
This proves \Cref{thm:stabD1Q3BoundaryConditions}.

\subsection{Fourth-order \lbmScheme{1}{3} scheme}\label{sec:D1Q3FourthOrder}

We consider the scheme introduced in \cite{bellotti:hal-04358349}.
The setting is the one of \eqref{eq:descriptionD1Q3LW} except for the following choice of equilibrium and the fact of fixing the relaxation parameters to two:
\begin{equation*}
	\equilibriumVectorLetter_3 = \tfrac{1}{3}(1+2\courantNumber^2)\qquad \text{and}\qquad \relaxationParameterLetter_2=\relaxationParameterLetter_3 = 2.
\end{equation*}
This scheme is fourth-order accurate.

\begin{proposition}[Stability of the fourth-order \lbmScheme{1}{3} scheme for the Cauchy problem]\label{prop:stabD1Q3FourthOrder}
	The fourth-order \lbmScheme{1}{3} scheme is stable according to \emph{von Neumann} if and only if $|\courantNumber|\leq \tfrac{1}{2}$.
	Moreover, stability according to \Cref{def:stabLBM} holds if and only if $|\courantNumber|< \tfrac{1}{2}$.
	Finally, the scheme is never stable according to \Cref{def:stabFD}.
\end{proposition}
\begin{proof}
	We discuss different values of the Courant number $\courantNumber$.
	\begin{itemize}
		\item Let $|\courantNumber|< \tfrac{1}{2}$. We have shown in \cite{bellotti:hal-04358349}, using \cite[Theorem 4.3.2 and 4.3.8]{strikwerda2004finite} that the roots $\timeShiftOperator$ of \eqref{eq:characteristicEquation} belong to $\unitCircle$ for all $\fourierShift\in\unitCircle$ (the scheme is stable according to \emph{von Neumann}) and that they are simple, except for $\fourierShift = 1$.
		In this case, there is a semi-simple root $\timeShiftOperator = -1$ of algebraic multiplicity equal to two. Stability in the lattice Boltzmann case needs further investigation, see below.
		On the other hand, this entails that $\timeShiftOperator = -1$ for $\fourierShift = 1$ shall have algebraic multiplicity equal to two and geometric multiplicity equal to one---when seen as eigenvalue of the companion matrix associated with the characteristic polynomial. Therefore, the corresponding Finite Difference scheme is not stable.
		\item Let $|\courantNumber|=\tfrac{1}{2}$. Using \cite[Theorem 4.3.1 and 4.3.2]{strikwerda2004finite}, we obtain that the roots $\timeShiftOperator$ of \eqref{eq:characteristicEquation} belong to $\unitCircle$ for all $\fourierShift\in\unitCircle$ (the scheme is stable according to \emph{von Neumann}) and that they are simple, except for $\fourierShift = e^{i \waveNumber}$ with $\waveNumber \in [-\pi, \pi]$ solution of $|3e^{3i\, \sign(\courantNumber)\waveNumber} + 1| = 4$, hence $\waveNumber = 0, \pm \tfrac{2}{3}\pi$.
		Like in the previous item, for $\waveNumber = 0$, the double root $\timeShiftOperator = -1$ is semi-simple.
		This is not the case for $\waveNumber = \pm \tfrac{2}{3}\pi$ with the double root $\timeShiftOperator = \tfrac{1}{2}\mp\sign(\courantNumber)\tfrac{\sqrt{3}}{2}$, which has geometric multiplicity equal to one, \idEst{} $\dimension{\kernel((\tfrac{1}{2}\mp\sign(\courantNumber)\tfrac{\sqrt{3}}{2})\identityMatrix{3} - \schemeMatrixBulkFourier(e^{\pm i {2}/{3}\pi}))} = 1$.
		This violates \cite[Lemma 3]{coulombel00616497} showing that the lattice Boltzmann scheme is not stable, since its amplification matrix cannot be power bounded for this mode.
		\item Let $|\courantNumber|>\tfrac{1}{2}$. Using \cite[Theorem 4.3.7]{strikwerda2004finite}, one can see that some roots $\timeShiftOperator$ of \eqref{eq:characteristicEquation} belong to $\neighborhoodInfinity$ for some $\fourierShift\in\unitCircle$, so the scheme is unstable.
	\end{itemize}

	Let us now focus on $|\courantNumber|< \tfrac{1}{2}$.
	The rest of the proof is similar to that of \Cref{prop:stableFDDoncStableLBM}.
	We recall that for $\fourierShift\in \unitCircle\smallsetminus\{1\}$, we have that $\spectrum(\schemeMatrixBulkFourier(\fourierShift))\subset \unitCircle$ and made up of distinct eigenvalues.
	For $\fourierShift = 1$, we have $\spectrum(\schemeMatrixBulkFourier(\fourierShift)) = \{1, -1\}$ with the second eigenvalue being double.
	This double eigenvalue is semi-simple.
	We want to show that $\schemeMatrixBulk$ is ``geometrically regular'' according to \cite[Definition 3]{coulombel00616497}, which entails stability according to \Cref{def:stabLBM} by virtue of \cite[Proposition 8]{coulombel00616497}.

	Let us start by $\fourierShift = 1$.
	Eigenvalues are continuous in $\fourierShift$, hence there (locally) exist continuous functions $\timeShiftOperator_1(\fourierShift)$, $\timeShiftOperator_2(\fourierShift)$, and $\timeShiftOperator_3(\fourierShift)$ such that $\timeShiftOperator_1(\fourierShift)\to 1$ and $\timeShiftOperator_2(\fourierShift), \timeShiftOperator_3(\fourierShift)\to -1$ as $\fourierShift\to 1$.
	Moreover, $\timeShiftOperator_1(\fourierShift)$ is holomorphic by virtue of \cite[Theorem 1]{greenbaum2020first}, which applies since $\schemeMatrixBulkFourier(\fourierShift)$ is holomorphic on $\puncturedPlane$.
	We now show that $\timeShiftOperator_2(\fourierShift)$ and $\timeShiftOperator_3(\fourierShift)$---and the corresponding eigenvectors---are analytic in a complex neighborhood of $\fourierShift = 1$. 
	The proof relies on \cite[Theorem 3.6]{hryniv1999perturbation} (see also \cite[Theorem 9]{lancaster2003perturbation}), which can be applied since $(\timeShiftOperator, \fourierShift)\mapsto \timeShiftOperator\identityMatrix{3} - \schemeMatrixBulkFourier(\fourierShift)$ is analytic on $\complex\times \puncturedPlane$.
	For the considered (double) eigenvalue, the so-called ``generating eigenvectors'' of $\timeShiftOperator\identityMatrix{3} - \schemeMatrixBulkFourier(\fourierShift)$ (\confer{} \cite[Definition 3.3]{hryniv1999perturbation}) are $\canonicalBasisVector{2}$ and $\canonicalBasisVector{3}$, whereas those of $(\timeShiftOperator\identityMatrix{3} - \schemeMatrixBulkFourier(\fourierShift))^{*}$ are $\canonicalBasisVector{1}-\tfrac{3}{2\courantNumber^2 + 1}\canonicalBasisVector{3}$ and $\canonicalBasisVector{2}-\tfrac{3\courantNumber}{2\courantNumber^2 + 1}\canonicalBasisVector{3}$.
	Since for each generating eigenvector of $\timeShiftOperator\identityMatrix{3} - \schemeMatrixBulkFourier(\fourierShift)$, there exists a generating eigenvector of $(\timeShiftOperator\identityMatrix{3} - \schemeMatrixBulkFourier(\fourierShift))^{*}$ which is non-orthogonal to the former, all the assumptions of \cite[Theorem 3.6]{hryniv1999perturbation} are verified, and this result entails that there exists a complex neighborhood $\mathscr{W}_1$ of $\fourierShift = 1$ in which $\timeShiftOperator_2(\fourierShift)$ and $\timeShiftOperator_3(\fourierShift)$ and the associated eigenvectors $\vectorial{\eigenvectorLetter}_2(\fourierShift)$ and $\vectorial{\eigenvectorLetter}_3(\fourierShift)$, such that $\vectorial{\eigenvectorLetter}_2(1)=\canonicalBasisVector{2}$ and $\vectorial{\eigenvectorLetter}_3(1)=\canonicalBasisVector{3}$, are analytic.
	By continuity of  $\vectorial{\eigenvectorLetter}_2(\fourierShift)$ and $\vectorial{\eigenvectorLetter}_3(\fourierShift)$ and their values at $\fourierShift = 1$, we see that they are necessarily independent in the neighborhood $\mathscr{W}_1$ (if too large, take a new neighborhood ${\mathscr{W}}_2\subset \mathscr{W}_1$ around $\fourierShift = 1$).
	Moreover, with all this points, we obtain that $\schemeMatrixBulk$ is ``geometrically regular'' (for the moment, at least $\fourierShift = 1$ is not an issue) , since the function $(\timeShiftOperator, \fourierShift)\mapsto \vartheta(\timeShiftOperator, \fourierShift)= \timeShiftOperator - \timeShiftOperator_1(\fourierShift)$, which does the job as polynomials with complex coefficients split on $\complex$, is holomorphic in the neighborhood $\complex\times \mathscr{W}_2$ (one can take ${\mathscr{W}}_3\subset \mathscr{W}_2$ instead of ${\mathscr{W}}_2$ if needed), and $\vartheta(-1, 1) = -2\neq 0$.

	For the other $\fourierShift\in\unitCircle\smallsetminus\{ 1 \}$, we show that $\schemeMatrixBulk$ is ``geometrically regular'' by arguments analogous to those of the proof of \Cref{prop:stableFDDoncStableLBM}. 
\end{proof}

Since the boundary conditions \eqref{eq:BBD1Q3LW}, \eqref{eq:ABBD1Q3LW}, \eqref{eq:TwoABBD1Q3LW}, \eqref{eq:extrapolationD1Q3LW}, and \eqref{eq:kineticDirichletD1Q3LW} are purely dictated by the velocity set, which is the same here, we also consider them for this fourth-order scheme.

\begin{theorem}[Strong stability--instability of the boundary-value fourth-order \lbmScheme{1}{3} scheme]\label{thm:stabilityBoundaryD1Q3Fourth}
	Under the CFL condition $|\courantNumber|<\tfrac{1}{2}$, the stability of the boundary-value fourth-order \lbmScheme{1}{3} scheme is as follows.
	\begin{center}
		\begin{tabular}{|c||c|c|}
			\cline{2-3}
			\multicolumn{1}{c||}{} & $\courantNumber<0$ (outflow)& $\courantNumber>0$ (inflow)\\
			\hline
			\hline
			Bounce-back \eqref{eq:BBD1Q3LW} & MU-E & SS\\
			Anti-bounce-back \eqref{eq:ABBD1Q3LW} & MU-E & SS\\
			Two-steps anti-bounce-back \eqref{eq:TwoABBD1Q3LW} & SSOO & SSOO \\
			Extrapolation $\orderExtrapolation = 1$ \eqref{eq:extrapolationD1Q3LW} & SSOO & MU-E\\
			Extrapolation $2\leq \orderExtrapolation \leq 4$ \eqref{eq:extrapolationD1Q3LW} & MU-E & MU-E\\
			Kinetic Dirichlet \eqref{eq:kineticDirichletD1Q3LW} & SS & SS\\
			\hline
		\end{tabular}
	\end{center}
\end{theorem}
The situation is totally analogous to the \lbmScheme{1}{2} when $\relaxationParameterLetter_2 = 2$, \confer{} \Cref{thm:stabilityBoundaryD1Q2}, except for \eqref{eq:TwoABBD1Q3LW} when $\courantNumber>0$.
Notice that \eqref{eq:kineticDirichletD1Q3LW} is---once again---strongly stable regardless of the sign of $\courantNumber$.

\subsubsection{Numerical simulations}\label{sec:numericalSimD1Q3Fourth}

\begin{figure}
	\begin{center}
		\includegraphics[width=1\textwidth]{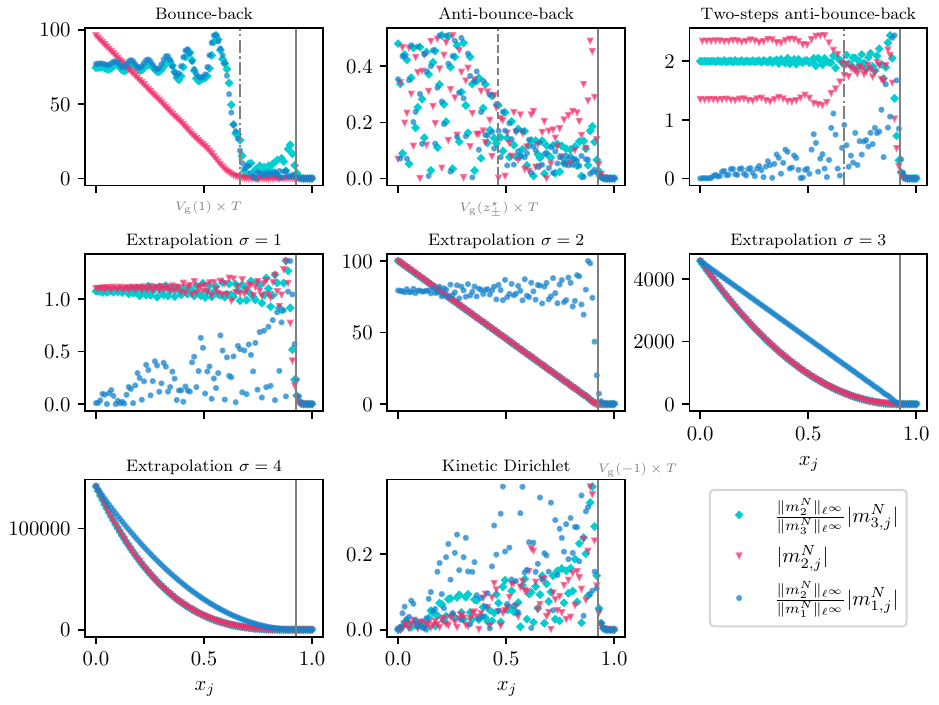}
	\end{center}\caption{\label{fig:D1Q3-4th-C--1_4-revision}Solution at final time $\finalTime$ for the fourth-order \lbmScheme{1}{3} scheme under $\courantNumber = -\tfrac{1}{4}$.
    The full grey line corresponds to $\groupVelocity(-1)\times\finalTime$, with $\groupVelocity(-1)$ is given by \eqref{eq:groupVelocityMinus1}.
	The dash-dotted grey line corresponds to $\groupVelocity(1)\times\finalTime$, with $\groupVelocity(1)$ is given by \eqref{eq:groupVelocityPlus1}.
	The dashed grey line corresponds to $\groupVelocity(\timeShiftOperator^{\star}_{\pm})\times \finalTime$, with $\groupVelocity(\timeShiftOperator^{\star}_{\pm})$ given by \eqref{eq:groupVelocityCompConj}.}
\end{figure}

\begin{figure}
	\begin{center}
		\includegraphics[width=1\textwidth]{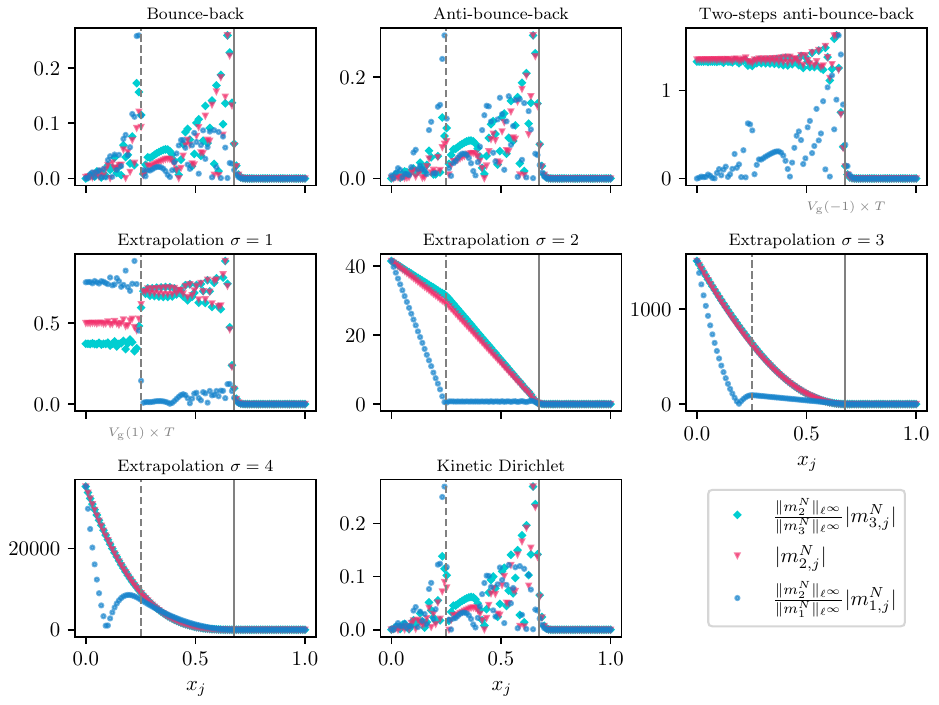}
	\end{center}\caption{\label{fig:D1Q3-4th-C-1_4-reivision}Solution at final time $\finalTime$ for the fourth-order \lbmScheme{1}{3} scheme under $\courantNumber = \tfrac{1}{4}$.
    The full grey line corresponds to $\groupVelocity(-1)\times\finalTime$, with $\groupVelocity(-1)$ is given by \eqref{eq:groupVelocityMinus1}, and the dashed grey line marks $\groupVelocity(1)\times\finalTime = \advectionVelocity\finalTime$.}
\end{figure}

We simulate using a domain with 100 points and show the solution at a time equal to one, with zero initial data and boundary source term $\boundarySourceTerm_{2, -1}^{\indexTime} = \delta_{0\indexTime}$. 
The results for $\courantNumber=- \tfrac{1}{4}$ are in \Cref{fig:D1Q3-4th-C--1_4-revision} and those for $\courantNumber=\tfrac{1}{4}$ in \Cref{fig:D1Q3-4th-C-1_4-reivision}.

\subsubsection{Proof of \Cref{thm:stabilityBoundaryD1Q3Fourth} and comments on the numerical simulations}

\paragraph{Description of the roots of the characteristic equation}

The characteristic equation \eqref{eq:characteristicEquation} features
\begin{align*}
	\coefficientCharEquationInFourier_{-1}(\timeShiftOperator) &=-\tfrac{1}{3} \, {\left(2 \, \courantNumber^{2} + 3 \, \courantNumber - 2\right)} {\timeShiftOperator}^{2} + \tfrac{1}{3} \, {\left(2 \, \courantNumber^{2} - 3 \, \courantNumber - 2\right)} {\timeShiftOperator}
	, \\ 
	\coefficientCharEquationInFourier_{0}(\timeShiftOperator) &= \tfrac{1}{3} \, {\left(4 \, \courantNumber^{2} - 1\right)} {\timeShiftOperator}^{2} + {\timeShiftOperator}^{3} - \tfrac{1}{3} \, {\left(4 \, \courantNumber^{2} - 1\right)} {\timeShiftOperator} - 1,	\\
	\coefficientCharEquationInFourier_{1}(\timeShiftOperator) &= -\tfrac{1}{3} \, {\left(2 \, \courantNumber^{2} - 3 \, \courantNumber - 2\right)} {\timeShiftOperator}^{2} + \tfrac{1}{3} \, {\left(2 \, \courantNumber^{2} + 3 \, \courantNumber - 2\right)} {\timeShiftOperator},
\end{align*}
where $\coefficientCharEquationInFourier_{-1}(\timeShiftOperator)$ and $\coefficientCharEquationInFourier_{1}(\timeShiftOperator)$ can vanish only for specific $\timeShiftOperator$'s.
These are 
\begin{align}
	\text{For}\quad \coefficientCharEquationInFourier_{-1}(\timeShiftOperator)=0, \qquad \timeShiftOperator&=\frac{(2\courantNumber+1)(\courantNumber-2)}{(2\courantNumber-1)(\courantNumber+2)}, \quad \text{with}\quad \fourierShift(\timeShiftOperator) = \frac{4\courantNumber^4-5\courantNumber^2-8}{(2\courantNumber+1)(2\courantNumber-1)(\courantNumber+2)(\courantNumber-2)}.\label{eq:critPointD1Q3FourthMinus1}\\
	\text{For}\quad \coefficientCharEquationInFourier_{1}(\timeShiftOperator)=0, \qquad \timeShiftOperator&=\frac{(2\courantNumber-1)(\courantNumber+2)}{(2\courantNumber+1)(\courantNumber-2)}, \quad \text{with}\quad \fourierShift(\timeShiftOperator) = \frac{(2\courantNumber+1)(2\courantNumber-1)(\courantNumber+2)(\courantNumber-2)}{4\courantNumber^4-5\courantNumber^2-8}.\label{eq:critPointD1Q3FourthPlus1}
\end{align}
For \eqref{eq:critPointD1Q3FourthMinus1}, we see that under the strict CFL condition $|\courantNumber|<\tfrac{1}{2}$, we have $\timeShiftOperator\in\neighborhoodInfinity$ as long as $\courantNumber\geq 0$. 
For the considered Courant numbers, one has that $\fourierShift(\timeShiftOperator)\in\neighborhoodInfinity$, which corresponds to an unstable root.
Conversely, for \eqref{eq:critPointD1Q3FourthPlus1}, we have that $\timeShiftOperator\in\neighborhoodInfinity$ as long as $\courantNumber\leq 0$. Then, $\fourierShift(\timeShiftOperator)\in\unitDisk$, which corresponds to a stable root.
However, we do not need to analyze these cases in the future, as we work under the strict CFL condition and only encounter points on $\unitCircle$. 
This allows---for instance---to simplify $\scalarFactorFromAdjugate(\timeShiftOperator)$ from both sides of \eqref{eq:coefficientWithKreissLopDet}.

\begin{proposition}
	Let the \strong{von Neumann} stability condition from \Cref{prop:stabD1Q3FourthOrder} be fulfilled.
	Then, for $\timeShiftOperator\in\neighborhoodInfinity$, the characteristic equation \eqref{eq:characteristicEquation} for the fourth-order \lbmScheme{1}{3} scheme has two roots (except at points where $\coefficientCharEquationInFourier_{-1}(\timeShiftOperator)$ or $\coefficientCharEquationInFourier_{1}(\timeShiftOperator)$ vanish), one stable root $\stableRoot(\timeShiftOperator)\in\unitDisk$, and one unstable root $\unstableRoot(\timeShiftOperator)\in\neighborhoodInfinity$.
	They can be extended to the unit circle, \idEst{} to $\timeShiftOperator\in\unitCircle$, and this extension is still denoted $\stableRoot(\timeShiftOperator)$ and $\unstableRoot(\timeShiftOperator)$, with
	\begin{equation*}
		\stableRoot(1) = 
		\begin{cases}
			-1, \qquad &\text{if}\quad \courantNumber<0, \\
			1 \qquad &\text{if}\quad \courantNumber>0.
		\end{cases}\qquad 
		\unstableRoot(1) = 
		\begin{cases}
			1, \qquad &\text{if}\quad \courantNumber<0, \\
			-1\qquad &\text{if}\quad \courantNumber>0,
		\end{cases}
		\qquad \text{and}\qquad \stableRoot(-1) = \unstableRoot(-1) = 1.
	\end{equation*}
\end{proposition}
\begin{proof}
	The first part of the claim is the Hersh lemma.
	For the second part, we have 
	\begin{itemize}
		\item Let $(\timeShiftOperator, \fourierShift) = (1+\delta\timeShiftOperator, -1+\fourierShift)$. Into the characteristic equation, this yields $\delta\fourierShift = -\tfrac{1}{3\courantNumber}(2\courantNumber^2 + 1)\delta \timeShiftOperator$.
		Letting $\delta\timeShiftOperator>0$, this gives that $\delta\fourierShift<0$ for $\courantNumber>0$, which indicates that $\fourierShift(1) = -1 =\unstableRoot(1)$, whereas $\delta\fourierShift>0$ for $\courantNumber>0$, which indicates that $\fourierShift(1) = -1 =\stableRoot(1)$.
		\item The case $(\timeShiftOperator, \fourierShift) = (1, 1)$ is treated analogously and yields reversed results between $\stableRoot$ and $\unstableRoot$, \confer{} $\productRootsDOneQThree|_{\timeShiftOperator = 1}= -1$.
		\item For $(\timeShiftOperator, \fourierShift) = (-1, 1)$, the root is double, as already seen in \cite{bellotti:hal-04358349}.
	\end{itemize}
\end{proof}

\paragraph{Shared eigenvalues}

From \eqref{eq:eigenvalueProblemsExcludingKappaZero}, we have the following candidates
\begin{align*}
	\text{For }\eqref{eq:BBD1Q3LW}\qquad &(\timeShiftOperator, \fourierShift) =  (1, -1), \quad (\timeShiftOperator, \fourierShift) = (-1, 1), \\
	\text{For }\eqref{eq:TwoABBD1Q3LW}\qquad &(\timeShiftOperator, \fourierShift) =  (1, -1), \quad (\timeShiftOperator, \fourierShift) = (-1, 1),\\
	\text{For }\eqref{eq:extrapolationD1Q3LW}, \, \orderExtrapolation = 1, 2, 3, 4\qquad &(\timeShiftOperator, \fourierShift) =  (1, 1), \quad (\timeShiftOperator, \fourierShift) = (-1, 1),\\
	\text{For }\eqref{eq:kineticDirichletD1Q3LW}\qquad &(\timeShiftOperator, \fourierShift) = (-1, 1).
\end{align*}
To correctly deal with \eqref{eq:ABBD1Q3LW}, we first solve the second equation in \eqref{eq:eigenvalueProblemsExcludingKappaZero}, which is first-order in $\fourierShift$, in $\fourierShift$. 
We plug this result into the first equation to yield the corresponding $\timeShiftOperator$:
\begin{equation*}
	\text{For }\eqref{eq:ABBD1Q3LW}\qquad (\timeShiftOperator, \fourierShift) =  (-1, 1), \quad \timeShiftOperator
	= \tfrac{1}{3}(1-4\courantNumber^2) \pm i \tfrac{2}{3}\sqrt{2(2\courantNumber^2+1)(1-\courantNumber^2)} =: \timeShiftOperator^{\star}_{\pm}.
\end{equation*}
A straightforward computation shows that $\timeShiftOperator^{\star}_{\pm}\in\unitCircle$, hence it is useful to investigate the corresponding $\fourierShift$'s.
From the second equation in \eqref{eq:eigenvalueProblemsExcludingKappaZero}:
\begin{multline}\label{eq:definitionStarredK}
	\fourierShift(\timeShiftOperator^{\star}_{\mp}) =: \fourierShift^{\star}_{\mp} = \\
	-\frac{32  \courantNumber^{6} + 48  \courantNumber^{5} - 36  \courantNumber^{4} - 42  \courantNumber^{3} - 6  \courantNumber^{2} \pm {\left(16  \courantNumber^{4} + 24  \courantNumber^{3} - 14  \courantNumber^{2} - 15  \courantNumber - 2\right)} \sqrt{4  \courantNumber^{4} - 2  \courantNumber^{2} - 2} - 6  \courantNumber + 10}{32  \courantNumber^{6} - 48  \courantNumber^{5} - 36  \courantNumber^{4} + 42  \courantNumber^{3} - 6  \courantNumber^{2} \pm {\left(16  \courantNumber^{4} - 24  \courantNumber^{3} - 14  \courantNumber^{2} + 15  \courantNumber - 2\right)} \sqrt{4  \courantNumber^{4} - 2  \courantNumber^{2} - 2} + 6  \courantNumber + 10}.
\end{multline}
Call $\mathscr{R}(\courantNumber) \definitionEquality -4  \courantNumber^{4} + 2  \courantNumber^{2} + 2 = 2(2\courantNumber^2+1)(1-\courantNumber^2)$, $\fourierShift^{\star}_{\mp}$ reads 
\begin{equation*}
	\fourierShift^{\star}_{\mp} =-\frac{\mathscr{E}_{\Re}(\courantNumber) + \mathscr{O}_{\Re}(\courantNumber) \pm i (\mathscr{E}_{\Im}(\courantNumber) + \mathscr{O}_{\Im}(\courantNumber))\sqrt{\mathscr{R}(\courantNumber)}}{\mathscr{E}_{\Re}(\courantNumber) - \mathscr{O}_{\Re}(\courantNumber) \pm i (\mathscr{E}_{\Im}(\courantNumber) - \mathscr{O}_{\Im}(\courantNumber))\sqrt{\mathscr{R}(\courantNumber)}}, 
\end{equation*}
where $\mathscr{E}_{\Re}$ and $\mathscr{E}_{\Im}$ are even polynomials, and $\mathscr{O}_{\Re}$ and $\mathscr{O}_{\Im}$ are odd polynomials, given by 
\begin{align*}
	\mathscr{E}_{\Re} &= 2(8\courantNumber^2-5)(2\courantNumber^2 + 1)(\courantNumber^2 - 1), \qquad \mathscr{O}_{\Re} = 6(8\courantNumber^2+1)(\courantNumber^2 - 1)\courantNumber,\\
	\mathscr{O}_{\Im} &= 3(8\courantNumber^2-5)\courantNumber, \qquad \mathscr{E}_{\Im} = 2(8\courantNumber^2+1)(\courantNumber^2 - 1).
\end{align*}
A straightforward multiplication and division by the conjugate of the denominator gives, omitting the dependence on the Courant number $\courantNumber$
\begin{equation*}
	\fourierShift^{\star}_{\mp}  = - \frac{\mathscr{E}_{\Re}^2 - \mathscr{O}_{\Re}^2 + (\mathscr{E}_{\Im}^2 - \mathscr{O}_{\Im}^2)\mathscr{R} \pm 2 i (\mathscr{E}_{\Re}\mathscr{O}_{\Im}-\mathscr{E}_{\Im}\mathscr{O}_{\Re}) \sqrt{\mathscr{R}}}{(\mathscr{E}_{\Re}-  \mathscr{O}_{\Re})^2 + (\mathscr{E}_{\Im} - \mathscr{O}_{\Im})^2\mathscr{R}}.
\end{equation*}
This yields 
\begin{equation*}
	| \fourierShift^{\star}_{\mp} |^2 = \frac{(\mathscr{E}_{\Re}^2 - \mathscr{O}_{\Re}^2 + (\mathscr{E}_{\Im}^2 - \mathscr{O}_{\Im}^2)\mathscr{R})^2 + 4  (\mathscr{E}_{\Re}\mathscr{O}_{\Im}-\mathscr{E}_{\Im}\mathscr{O}_{\Re})^2 \mathscr{R}}{((\mathscr{E}_{\Re}-  \mathscr{O}_{\Re})^2 + (\mathscr{E}_{\Im} - \mathscr{O}_{\Im})^2\mathscr{R})^2}
	= \frac{2916(4\courantNumber^4-11\courantNumber^2-2)^2(\courantNumber+1)^2(\courantNumber-1)^2}{2916(4\courantNumber^4-11\courantNumber^2-2)^2(\courantNumber+1)^2(\courantNumber-1)^2} 
	= 1
\end{equation*}
and shows that $\fourierShift^{\star}_{\mp}\in\unitCircle$.

In order to understand if these values of $\fourierShift^{\star}_{\pm}$ correspond either to $\stableRoot$ or to $\unstableRoot$, we solve \eqref{eq:characteristicEquation} for $\timeShiftOperator^{\star}_{\pm}$.
In both cases, the additional $\fourierShift$ is equal to $-1$ (notice that this is not---by the way---a solution of the second equation in \eqref{eq:eigenvalueProblemsExcludingKappaZero}, relative to the boundary condition).
We therefore consider $(\timeShiftOperator, \fourierShift) = ((1+\delta \timeShiftOperator)\timeShiftOperator^{\star}_{\pm}, -1+\delta\fourierShift)$ with $\delta\timeShiftOperator>0$ playing the role of a perturbation of the modulus, $\delta\fourierShift\in\complex$, and insert this into \eqref{eq:characteristicEquation}.
At leading order, this gives
\begin{equation*}
	\delta\fourierShift = \frac{2(2\courantNumber^2+1)}{3\courantNumber}\delta\timeShiftOperator,
\end{equation*}
from which we deduce that $\delta\fourierShift\in\reals$ and $\delta\fourierShift>0$ when $\courantNumber>0$, hence the other root, the one equal to $\fourierShift = -1$, is indeed $\stableRoot$, thus $\fourierShift^{\star}_{\mp} $ is $\unstableRoot$. 
Conversely, for $\courantNumber<0$, we obtain $\delta\fourierShift<0$, hence the other root equal to $\fourierShift = -1$ is $\unstableRoot$, which gives that $\fourierShift^{\star}_{\mp}$ is $\stableRoot$.

The whole picture is in \Cref{tab:modesD1Q4Fourth}.

\begin{table}\caption{\label{tab:modesD1Q4Fourth}Unstable candidate eigenvalues (shared between bulk and boundary scheme) for the fourth-order \lbmScheme{1}{3} scheme.}
	\begin{center}
		\begin{tabular}{|c||c|c|}
			\cline{2-3}
			\multicolumn{1}{c||}{} & $\courantNumber<0$ (outflow) & $\courantNumber>0$ (inflow)\\
			\hline
			\hline
			Bounce-back \eqref{eq:BBD1Q3LW} & $(1, -1)$ and $(-1, 1)$ & $(-1, 1)$ \\
			Anti-bounce-back \eqref{eq:ABBD1Q3LW} & $(-1, 1)$ and $(\timeShiftOperator^{\star}_{\pm}, \fourierShift^{\star}_{\pm})$ by \eqref{eq:definitionStarredK}  & $(-1, 1)$ \\
			Two-steps anti-bounce-back \eqref{eq:TwoABBD1Q3LW} & $(1, -1)$ and $(-1, 1)$ & $(-1, 1)$ \\
			Extrapolation $1\leq \orderExtrapolation \leq 4$ \eqref{eq:extrapolationD1Q3LW} & $(-1, 1)$ & $(1, 1)$ and $(-1, 1)$ \\
			Kinetic Dirichlet \eqref{eq:kineticDirichletD1Q3LW} & $(-1, 1)$ & $(-1, 1)$ \\
			\hline
		\end{tabular}
	\end{center}
\end{table}

\paragraph{Eigenvectors: conclusion of the proof and description of the instabilities}

For $\timeShiftOperator\in\neighborhoodInfinity$, the root $\stableRoot(\timeShiftOperator)$ is simple. 
Its extension to $\timeShiftOperator\in\unitDisk$ is also simple, except when $\timeShiftOperator = -1$, where it equals $\unstableRoot$. 
Still, $\dimension{\kernel(\matrixPolynomialBulk{-1}(1))} = 2$.
\Cref{prop:noDiracBoundary} applies.

\subparagraph{Continuity of the extension of $\stableSubspace(\timeShiftOperator)$ to $\unitCircle$}

We compute an element of $\kernel(\matrixPolynomialBulk{\timeShiftOperator}(\fourierShift_{\pm}(\timeShiftOperator)))$, with the first component normalized to one.
The second and third components come under the form of fractions. For the second component, the denominator vanishes at $\timeShiftOperator = \pm 1$.
For the third component, it does so only at $\timeShiftOperator = -1$.
These values are points where a continuous extension of $\stableSubspace(\timeShiftOperator)$ to $\unitCircle$ may lack.
Let us study this in more detail.
\begin{itemize}
	\item Let $\courantNumber<0$. 
	Consider $\timeShiftOperator = -1$.
	Inserting $\timeShiftOperator = -1 -\delta\timeShiftOperator$ with $\delta\timeShiftOperator>0$ into $\fourierShift_{\pm}(\timeShiftOperator)$ and performing Taylor expansions, we obtain
	\begin{equation}\label{eq:groupVelocityMinus1Preliminary}
		\fourierShift_{\pm}(-1 -\delta\timeShiftOperator) = 1 - \frac{3\courantNumber \pm \sqrt{3}\sqrt{8-5\courantNumber^2}}{4(1-\courantNumber^2)}\delta\timeShiftOperator + \bigO{(\delta\timeShiftOperator)^2},
	\end{equation}
	where $3\courantNumber + \sqrt{3}\sqrt{8-5\courantNumber^2} >0$ whereas $3\courantNumber - \sqrt{3}\sqrt{8-5\courantNumber^2} <0$.
	Therefore, we locally parametrize $\stableRoot$ by $\fourierShift_+$.
	The first order term in \eqref{eq:groupVelocityMinus1Preliminary} can be readily employed in \eqref{eq:groupVelocity} to yield the group velocity of this mode:
	\begin{equation}\label{eq:groupVelocityMinus1}
		\groupVelocity(-1) = \latticeVelocity\tfrac{\sqrt{3}}{6}(-\sqrt{3}\courantNumber + \sqrt{8-5\courantNumber^2}) > -\latticeVelocity\courantNumber,
	\end{equation}
	which we already found in \cite{bellotti:hal-04358349}.
	For an eigenvector $\vectorial{\eigenvectorLetter}_{\textnormal{s}}(\timeShiftOperator) \in \kernel(\matrixPolynomialBulk{\timeShiftOperator}(\stableRoot(\timeShiftOperator)))$, we have 
	\begin{equation}\label{eq:eigenvectorD1Q34thMinus1}
		\vectorial{\eigenvectorLetter}_{\textnormal{s}}(\timeShiftOperator) = 
		\begin{pmatrix}
			1\\
			{\eigenvectorLetter}_{\textnormal{s}, 2}(\timeShiftOperator)\\
			{\eigenvectorLetter}_{\textnormal{s}, 3}(\timeShiftOperator)
		\end{pmatrix}, 
		\qquad \text{with}\qquad
		\begin{array}{l}
			{\eigenvectorLetter}_{\textnormal{s}, 2}(\timeShiftOperator) = \frac{4 \sqrt{3}  {\left(- \sqrt{3}   \courantNumber+  \sqrt{8 - 5  \courantNumber^{2} } \right)}{\left(1 -\courantNumber^2\right)}}{-3 \sqrt{3} \courantNumber\sqrt{8-5  \courantNumber^{2}}  - 3  \courantNumber^{2} + 12} \frac{1}{\timeShiftOperator+1}+\bigO{1}, \\
			{\eigenvectorLetter}_{\textnormal{s}, 3}(\timeShiftOperator) = \tfrac{4}{3}(\courantNumber^2-1) \frac{1}{\timeShiftOperator+1}+\bigO{1}.
		\end{array}
	\end{equation}

	For $\timeShiftOperator = 1$, the parametrization of the stable root is also given by $\fourierShift_+$.
	This is easily checked without any issue linked to multiplicity, contrarily to $\timeShiftOperator = -1$.
	We obtain the group velocity 
	\begin{equation}\label{eq:groupVelocityPlus1}
		\groupVelocity(1) = -\latticeVelocity \frac{3\courantNumber}{2\courantNumber^2 + 1}>0,
	\end{equation}
	which is interestingly different from $\latticeVelocity\courantNumber$, the ``physical'' velocity that equals the group velocity of $\unstableRoot$.
	The eigenvector $\vectorial{\eigenvectorLetter}_{\textnormal{s}}(\timeShiftOperator) = \transpose{(1, {\eigenvectorLetter}_{\textnormal{s}, 2}(\timeShiftOperator), {\eigenvectorLetter}_{\textnormal{s}, 3}(\timeShiftOperator))}$ has
	\begin{equation*}
			{\eigenvectorLetter}_{\textnormal{s}, 2}(\timeShiftOperator) =-\frac{2\courantNumber}{\timeShiftOperator-1}+\bigO{\timeShiftOperator-1}\qquad \text{and}\qquad
			{\eigenvectorLetter}_{\textnormal{s}, 3}(\timeShiftOperator) = \tfrac{1}{3}(1+2\courantNumber^2) +\bigO{\timeShiftOperator-1}.
	\end{equation*}
	\item Let $\courantNumber>0$. When considering $\timeShiftOperator = -1$, the same expansion as \eqref{eq:groupVelocityMinus1Preliminary} holds, thus also in this case $\stableRoot$ is locally parametrized by $\fourierShift_+$. The group velocity of the mode is also given by \eqref{eq:groupVelocityMinus1}, and the principal part of $\vectorial{\eigenvectorLetter}_{\textnormal{s}}(\timeShiftOperator)$ by \eqref{eq:eigenvectorD1Q34thMinus1} applies.
	
	For $\timeShiftOperator = 1$, once again $\stableRoot$ is parametrized by $\fourierShift_+$, and we obtain the group velocity $\groupVelocity(1) = \latticeVelocity\courantNumber>0$, which is different from \eqref{eq:groupVelocityPlus1} since it is the ``physical'' velocity of the flow.
	There is no blowup in $\eigenvectorLetter_{\textnormal{s}, 2}(\timeShiftOperator)$ thanks to a zero-pole cancellation in the inherent fraction: $\eigenvectorLetter_{\textnormal{s}, 2}(\timeShiftOperator) = \courantNumber + \bigO{\timeShiftOperator - 1}$.
\end{itemize}

\subparagraph{Kreiss-Lopatinskii determinants in the outflow case $\courantNumber<0$}
Now consider $\courantNumber<0$.
\begin{itemize}
	\item Consider \eqref{eq:BBD1Q3LW}. We have $\kreissLopatinskiiScalarProduct(\timeShiftOperator) = \bigO{(\timeShiftOperator + 1)^{-1}}$ with involved coefficients, hence \threeboxes{\notContinuousExtensionMark}{\eigenvalueMark}{\infiniteKL}.
	This is coherent with \Cref{fig:D1Q3-4th-C--1_4-revision} (this mode fades away).
	However, $\kreissLopatinskiiScalarProduct(1) = 0$ (\idEst{}, \threeboxes{\notContinuousExtensionMark}{\eigenvalueMark}{\zeroKL}), hence the scheme is neither SS, nor SSOO.
	\item Consider \eqref{eq:ABBD1Q3LW}. We have $\kreissLopatinskiiScalarProduct(\timeShiftOperator) = \frac{4(1-\courantNumber^2)}{3} (\timeShiftOperator + 1)^{-1} + \bigO{1}$, hence \threeboxes{\notContinuousExtensionMark}{\eigenvalueMark}{\infiniteKL}. 
	For the eigenvalue $(\timeShiftOperator_{\pm}^{\star}, \fourierShift_{\pm}^{\star})$, we obtain $\kreissLopatinskiiScalarProduct(\timeShiftOperator_{\pm}^{\star}) = 0$, hence \threeboxes{\continuousExtensionMark}{\eigenvalueMark}{\zeroKL} and the scheme is neither SS nor SSOO.
	The group velocity of this mode could be expressed in a closed form, which would however be quite involved.
	For $\courantNumber=-\tfrac{1}{4}$ considered in \Cref{sec:numericalSimD1Q3Fourth}, it is explicitly given by 
	\begin{equation}\label{eq:groupVelocityCompConj}
		\groupVelocity(\timeShiftOperator^{\star}_{\pm}) =  \latticeVelocity \tfrac{19}{41}.
	\end{equation}
	Notice that resonant boundary source terms for this mode are proportional to  
	\begin{equation*}
		\cos  ( \indexTime \,\textnormal{arg}(\timeShiftOperator^{\star}_{\pm}))=  \cos \Bigl ( \indexTime \arctan\Bigl ( \frac{2\sqrt{2(2\courantNumber^2+1)(1-\courantNumber^2)}}{1-4\courantNumber^2}\Bigr )\Bigr ).
	\end{equation*} 
	This whole description is perfectly coherent with \Cref{fig:D1Q3-4th-C--1_4-revision}.
	\item Consider \eqref{eq:TwoABBD1Q3LW}. 
	We have that $\kreissLopatinskiiScalarProduct(-1) = -1\neq 0$, so \threeboxes{\notContinuousExtensionMark}{\eigenvalueMark}{\nonZeroKL}. 
	We obtain $\kreissLopatinskiiScalarProduct(1) = 1 \neq 0$, hence \threeboxes{\notContinuousExtensionMark}{\eigenvalueMark}{\nonZeroKL}.
	This entails that the scheme is SSOO only: once arrived at something similar to \eqref{eq:tmp7}, we can invoke \cite[Lemma 11.3.2]{strikwerda2004finite}, except in the vicinity of $\timeShiftOperator = -1$, since here the corresponding Finite Difference scheme obtained by the characteristic polynomial fails to be stable. 
		However, we proceed as in the proof of \cite[Lemma 11.3.2]{strikwerda2004finite}: in this vicinity, $\fourierShift$ solution of \eqref{eq:characteristicEquation} is analytic in $\timeShiftOperator$ and $\stableRoot$ is simple (except at the limit point). 
	Moreover, since we have a parametrization of $\stableRoot$ given by $\fourierShift_+$, we can propose an \emph{ad hoc} analysis at this point:
	\begin{equation*}
		\frac{1}{1-|\stableRoot(\timeShiftOperator)|^2} = \frac{2(\courantNumber^2-1)}{3\courantNumber-\sqrt{3}\sqrt{8-5\courantNumber^2}}\frac{1}{\timeShiftOperator + 1} + \bigO{1},
	\end{equation*}
	which proves that there exists $\tilde{K}>0$ such that $\frac{1}{1-|\stableRoot(\timeShiftOperator)|^2} \leq \frac{\tilde{K}}{|\timeShiftOperator|-1}$ for all $\timeShiftOperator\in\neighborhoodInfinity$, which concludes the proof as for the \lbmScheme{1}{2} scheme.
	\item Consider \eqref{eq:extrapolationD1Q3LW} with $\orderExtrapolation = 1$.
	We obtain $\kreissLopatinskiiScalarProduct(-1) \neq 0$ (the precise expression is involved), \idEst{} \threeboxes{\notContinuousExtensionMark}{\eigenvalueMark}{\nonZeroKL}.
	This boundary condition is thus SSOO.
	\item Consider \eqref{eq:extrapolationD1Q3LW} for $\orderExtrapolation = 2, 3, 4$.
	We have $\kreissLopatinskiiScalarProduct(\timeShiftOperator) = \bigO{(\timeShiftOperator + 1)^{\sigma-1}}$, hence \threeboxes{\notContinuousExtensionMark}{\eigenvalueMark}{\zeroKL}, proving that these boundary conditions are neither SS, nor SSOO.
	\item Consider \eqref{eq:kineticDirichletD1Q3LW}.
	For the eigenvalue $(-1, 1)$, we have $\kreissLopatinskiiScalarProduct(\timeShiftOperator) = \bigO{(\timeShiftOperator + 1)^{-1}}$, so \threeboxes{\notContinuousExtensionMark}{\eigenvalueMark}{\infiniteKL}.
	The boundary condition is thus SS (and not only SSOO).
\end{itemize}

\subparagraph{Kreiss-Lopatinskii determinants in the outflow case $\courantNumber>0$}
We finish on $\courantNumber>0$.
\begin{itemize}
	\item Consider \eqref{eq:BBD1Q3LW} and \eqref{eq:ABBD1Q3LW}. 
	For both boundary conditions, we obtain $\kreissLopatinskiiScalarProduct(\timeShiftOperator) = \bigO{(\timeShiftOperator + 1)^{-1}}$ (\idEst{} \threeboxes{\notContinuousExtensionMark}{\eigenvalueMark}{\infiniteKL}), hence they are SS.
	\item Consider \eqref{eq:TwoABBD1Q3LW}.
	We have $\kreissLopatinskiiScalarProduct(-1) = \tfrac{4}{3}(1-\courantNumber^2) \neq 0$, so \threeboxes{\notContinuousExtensionMark}{\eigenvalueMark}{\nonZeroKL}.
	The boundary condition is thus only SSOO.
	Even if this boundary condition is used to enforce exact Dirichlet boundary conditions on $\momentDiscrete_1$, this fact is not problematic for smooth solutions, for one is interested in taking advantage of the fourth-order accuracy. Indeed, checkerboard modes $\propto (-1)^{\indexTime}$ are not excited, for they are far from being smooth.
	\item Consider \eqref{eq:extrapolationD1Q3LW} for $\orderExtrapolation = 1, 2, 3, 4$.
	We obtain
	\begin{equation*}
		\kreissLopatinskiiMatrix(\timeShiftOperator)\canonicalBasisVector{1} = \frac{2\courantNumber^2+3\courantNumber+1}{6\courantNumber^{\orderExtrapolation}}(\timeShiftOperator-1)^{\orderExtrapolation}(\canonicalBasisVector{1}+\canonicalBasisVector{2}+\canonicalBasisVector{3}) + \bigOVectorial{(\timeShiftOperator-1)^{\orderExtrapolation - 1}},
	\end{equation*}
	which entails that $\kreissLopatinskiiScalarProduct(\timeShiftOperator) = \bigO{(\timeShiftOperator-1)^{\orderExtrapolation}}$, so \threeboxes{\continuousExtensionMark}{\eigenvalueMark}{\zeroKL}, and the boundary condition is neither SS nor SSOO.
	In \Cref{fig:D1Q3-4th-C--1_4-revision}, we see a second mode: $\kreissLopatinskiiScalarProduct(\timeShiftOperator) = \bigO{(\timeShiftOperator+1)^{\orderExtrapolation-1}}$, hence \threeboxes{\notContinuousExtensionMark}{\eigenvalueMark}{\nonZeroKL} when $\orderExtrapolation = 1$ (and this mode is mastered on $\momentDiscrete_1$), and \threeboxes{\notContinuousExtensionMark}{\eigenvalueMark}{\zeroKL} when $\orderExtrapolation = 2, 3, 4$.
	\item Consider \eqref{eq:kineticDirichletD1Q3LW}. We have $\kreissLopatinskiiScalarProduct(\timeShiftOperator) = \bigO{(\timeShiftOperator + 1)^{-1}}$, so \threeboxes{\notContinuousExtensionMark}{\eigenvalueMark}{\infiniteKL}, and the boundary condition is SS.
\end{itemize}

\section{Conclusions and perspectives}\label{sec:conclusions}

We have introduced the notion of strong stability for \lbm{} schemes, which we have studied for three representative schemes endowed with boundary conditions from the literature.
The intrinsic characteristic nature of \lbm{} schemes causes a lack of continuous extension of the stable vector bundle to the unit circle when relaxation parameters equal two.
We conjecture that this is the only case where this can generally happen.
In this setting, the characteristic equation \eqref{eq:characteristicEquation} is the one of the leap-frog scheme (see \Cref{sec:D1Q2}) or has similar properties (\confer{}, \Cref{sec:D1Q3FourthOrder}), namely does not dissipate any harmonics.
This lack of continuous extension underpins radically different behaviors between different components of the numerical solution, some of which can be strongly stable while others are not.

Besides consistency and thus actual ``physical'' usefulness of boundary conditions---which has not been addressed in the paper---we can draw the following conclusions upon the three schemes we have considered.
First, the ``kinetic Dirichlet'' boundary condition, boiling down to setting incoming distribution functions to zero, was strongly stable regardless of the nature (inflow or outflow) of the boundary.
This is analogous to the so-called \strong{Goldberg-Tadmor} lemma, stipulating the same thing for Dirichlet boundary conditions for stable Finite Difference schemes.
Second, the anti-bounce back condition, usually employed to enforce inflow boundary conditions, is strongly stable when the boundary is an inflow. However, it is not strongly stable when used at outflows, which is probably one of the reasons why it is not usually employed in this case.
Third, the two-steps anti-bounce back condition, which can be used to enforce exact Dirichlet conditions on the conserved moment, is always at least strongly stable on the observed output (\idEst{}, on the conserved moment).

Several stimulating research perspectives are envisioned.
The main ones are discussed below.
First, introduce bulk source terms---in addition to the already considered boundary source terms---in the definition of strong stability for \lbm{} schemes.
Second, derive semigroup estimates \cite{coulombel2011semigroup, coulombel2015leray} of the numerical solution---to be measured, roughly speaking, in the $\ell_{\timeStep}^{\infty}\ell_{\spaceStep}^{2}$-norm rather than with respect to the $\ell_{\timeStep}^{2}\ell_{\spaceStep}^{2}$-topology as in \Cref{def:strongStability}.
This would allow considering non-zero initial data (although this has also been investigated within strong stability, see \cite{coulombel2015fully}) and hopefully attain convergence results in the spirit of \cite{coulombel2020neumann}.
Finally, it is interesting to consider multi-dimensional problems, first with a boundary only in one direction, \strong{e.g.} \cite{abarbanel1979stability, michelson1983stability} (called ``half-space'' problem in 2D) and later with boundaries in several directions \cite{benoit2023} (a ``quarter-space'' problem in 2D).
For the half-space problem, the GKS procedure can be adapted by considering a Fourier transform in the direction tangential to the boundary.
The quarter-space problem needs more subtle ways of proceeding.
In terms of richer target equations, a first step forward is to consider linear hyperbolic systems of equations. While adapting the present GKS-like framework is indeed possible, one must be aware that algebra to actually check the stability of given boundary conditions (find explicit \strong{von Neumann}--$\ell^2$ stability conditions, solve the characteristic equation, \emph{etc.}) may become involved. 
An interesting starting point would be to consider vectorial schemes, such as the vectorial \lbmScheme{1}{2} \cite{graille2014approximation}, for which the characteristic equation---upon using one common relaxation parameter---multiplicatively splits across the waves of the hyperbolic system, see \cite{bellotti2025perfectly}.
Secondly, adopting a GKS approach to study non-linear problems is unlikely to work, as this approach is intrinsically linear. 
Moreover, we are not aware of any existing research in this direction concerning Finite Difference schemes.
This setting definitely needs different tools to deal with the stability of boundary conditions, \strong{e.g.} stability structure \cite{junk2009weighted, junk2009convergence} or monotonicity \cite{aregba2025equilibrium}.

\section*{Acknowledgements}

The author thanks Jean-François Coulombel for patient advice on these topics and for his welcome in Toulouse (France) that provided new inputs for this research.
The author also acknowledges the help of Tommaso Tenna and Vincent Lescarret, who read and usefully commented the draft of this manuscript.
Useful advice by two anonymous referees is also acknowledged.

\section*{Data Availability Statement}

Codes are available at \href{https://github.com/thomasbellotti/GKS-LBM}{https://github.com/thomasbellotti/GKS-LBM}.

\bibliographystyle{apalike}
\bibliography{biblio}

\appendix

\section{Proof of \Cref{lemma:semiSimple}}\label{app:lemma:semiSimple}

Let us start by pointing out few facts and introduce useful notations.
	\begin{itemize}
		\item Since $\timeShiftOperator$ is an eigenvalue of $\matricial{A}$, it is a root both of the minimal and characteristic polynomial of $\matricial{A}$. We denote its multiplicity as root of the minimal polynomial by $m\geq 1$.
		\item The sizes of the $g$ Jordan blocks associated with $\timeShiftOperator$ are denoted $s_1, \dots, s_{g}$.
		We denote the algebraic multiplicity of $\timeShiftOperator$, which equals $\sum_{\indexFreeOne = 1}^{g} s_{\indexFreeOne}$, by $a\geq 1$. The geometric multiplicity, which equals the number of Jordan blocks associated with $\timeShiftOperator$, is indicated by $g\geq 1$.
		It is well known that $a\geq g$ and that $m$ equals the size of the largest Jordan block, \idEst{} $m = \max(s_1, \dots, s_{g})$.
		\begin{equation}\label{eq:minimalAndMultiplicitiesInequality}
			a - g = \sum_{\indexFreeOne = 1}^{g} (s_{\indexFreeOne} - 1) \geq (\max(s_1, \dots, s_{g})-1) = m - 1, \qquad \text{hence}\qquad m\leq a - g + 1.
		\end{equation}
	\end{itemize}
	Then, we prove each direction in the equivalence.
	\begin{itemize}
		\item[$\Longrightarrow$] Assume $\timeShiftOperator$ be semi-simple. Hence, $a=g$, so by \eqref{eq:minimalAndMultiplicitiesInequality}, we obtain $1\leq m\leq 1$, so $m=1$.
		\item[$\Longleftarrow$] Assume $m=1$, hence the size of the largest Jordan block is one. Since $a = \sum_{\indexFreeOne = 1}^{g} s_{\indexFreeOne} $, we obtain $a\leq g$. On the other hand, $a\geq g$, thus $a = g$.
	\end{itemize}

	\section{Proof of \Cref{prop:maximumStencilCharacteristics}}\label{app:prop:maximumStencilCharacteristics}
	Assume, without loss of generality, that $\matricial{K} \in \linearGroup{\numberVelocities}{\baseField}$. Otherwise, one can proceed using the density of $\linearGroup{\numberVelocities}{\baseField}$ in $\matrixSpace{\numberVelocities}{\baseField}$. However, in the latter case, bounds can be non-sharp.
	Using simple manipulations on the determinant, we obtain 
	\begin{multline*}
		L_{\timeShiftOperator}(\fourierShift) = \determinant(\matricial{K}) \determinant(\timeShiftOperator\momentMatrix^{-1}\matricial{K}^{-1}\momentMatrix - \diagonalMatrix(\fourierShift^{\delta_1}, \dots, \fourierShift^{\delta_{\numberVelocities}})) \\
		= (-1)^{\numberVelocities}\determinant(\matricial{K}) \determinant(\diagonalMatrix(\fourierShift^{\delta_1}, \dots, \fourierShift^{\delta_{\numberVelocities}}) - \timeShiftOperator\momentMatrix^{-1}\matricial{K}^{-1}\momentMatrix ) .
	\end{multline*}
	We obtain 
	\begin{equation*}
		\diagonalMatrix(\fourierShift^{\delta_1}, \dots, \fourierShift^{\delta_{\numberVelocities}}) - \timeShiftOperator\momentMatrix^{-1}\matricial{K}^{-1}\momentMatrix = \underbrace{\sum_{\substack{\indexFreeOne = \min_{\indexVelocity}\delta_{\indexVelocity}\\\indexFreeOne\neq 0}}^{\max_{\indexVelocity }\delta_{\indexVelocity}} \fourierShift^{\indexFreeOne}\sum_{\delta_{\indexFreeTwo} = \indexFreeOne}\canonicalBasisVector{\indexFreeTwo}\transpose{\canonicalBasisVector{\indexFreeTwo}}}_{\text{diagonal part}} + \underbrace{\sum_{\delta_{\indexFreeTwo} = 0}\canonicalBasisVector{\indexFreeTwo}\transpose{\canonicalBasisVector{\indexFreeTwo}} - \timeShiftOperator\momentMatrix^{-1}\matricial{K}^{-1}\momentMatrix}_{\text{indep of }\fourierShift\text{ + non-diagonal part}}.
	\end{equation*}
	We call $\tilde{\matricial{K}}\definitionEquality \sum_{\delta_{\indexFreeTwo} = 0}\canonicalBasisVector{\indexFreeTwo}\transpose{\canonicalBasisVector{\indexFreeTwo}} - \timeShiftOperator\momentMatrix^{-1}\matricial{K}^{-1}\momentMatrix$ (its dependence on $\timeShiftOperator$ can be totally forgotten here), and consider the minimal degree $\min\degree(L_{\timeShiftOperator})$.
	The maximal degree $\max\degree(L_{\timeShiftOperator})$ can be bounded in exactly the same fashion.
	What we want to prove is that the following property $P(A)$ be true, by induction on $A\leq 0$:
	\begin{multline*}
		P(A) \quad : \quad \text{For every } \numberVelocities \geq 2 \text{, every }\tilde{\matricial{K}}\in\matrixSpace{\numberVelocities}{\complex}, \text{ and every }\delta_1, \dots, \delta_{\numberVelocities} \in \relatives \text{ such that }\sum\nolimits_{\delta_{\indexVelocity}<0}\delta_{\indexVelocity} = A, \\
		\text{then}\quad 
		\min\degree\Bigl ( \determinant\Bigl (\sum_{\indexFreeOne\neq 0}\fourierShift^{\indexFreeOne}\sum_{\delta_{\indexFreeTwo} = \indexFreeOne}\canonicalBasisVector{\indexFreeTwo}\transpose{\canonicalBasisVector{\indexFreeTwo}} + \tilde{\matricial{K}}\Bigr ) \Bigr ) \geq A.
	\end{multline*}
	We conduct the proof by strong induction.
	\begin{itemize}
		\item Base case $P(0)$. In this case, $\delta_1, \dots, \delta_{\numberVelocities} \in \naturals$, so $\determinant (\sum_{\indexFreeOne\neq 0}\fourierShift^{\indexFreeOne}\sum_{\delta_{\indexFreeTwo} = \indexFreeOne}\canonicalBasisVector{\indexFreeTwo}\transpose{\canonicalBasisVector{\indexFreeTwo}} + \tilde{\matricial{K}} )$ is a classical polynomial, hence its minimal degree if seen as a Laurent polynomial is zero.
		\item Induction step. Let $A < 0$ and assume that $P(A+1), \dots, P(0)$ hold true.
		Let $\alpha \in \integerInterval{1}{\numberVelocities}$ be such that $\delta_{\alpha}<0$. Therefore $A = \delta_{\alpha} + \sum_{\indexVelocity\neq \alpha, \delta_{\indexVelocity}<0} \delta_{\indexVelocity}<\sum_{\indexVelocity\neq \alpha, \delta_{\indexVelocity}<0} \delta_{\indexVelocity}\leq 0$.
		Using the matrix-determinant lemma, we obtain 
		\begin{equation}\label{eq:tmpMatrixDeterminant}
			\determinant\Bigl (\sum_{\indexFreeOne\neq 0}\fourierShift^{\indexFreeOne}\sum_{\delta_{\indexFreeTwo} = \indexFreeOne}\canonicalBasisVector{\indexFreeTwo}\transpose{\canonicalBasisVector{\indexFreeTwo}} + \tilde{\matricial{K}}\Bigr ) = \determinant\Bigl (\sum_{\indexFreeOne\neq 0}\fourierShift^{\indexFreeOne}\sum_{\substack{\delta_{\indexFreeTwo} = \indexFreeOne\\ \indexFreeTwo \neq \alpha}}\canonicalBasisVector{\indexFreeTwo}\transpose{\canonicalBasisVector{\indexFreeTwo}} + \tilde{\matricial{K}}\Bigr ) + \fourierShift^{\delta_{\alpha}} \transpose{\canonicalBasisVector{\alpha}} \adjugate\Bigl (\sum_{\indexFreeOne\neq 0}\fourierShift^{\indexFreeOne}\sum_{\substack{\delta_{\indexFreeTwo} = \indexFreeOne\\ \indexFreeTwo \neq \alpha}}\canonicalBasisVector{\indexFreeTwo}\transpose{\canonicalBasisVector{\indexFreeTwo}} + \tilde{\matricial{K}}\Bigr ) \canonicalBasisVector{\alpha},
		\end{equation}
		where the smallest degree of the determinant term on the right-hand side is larger or equal to $\sum_{\indexVelocity\neq \alpha, \delta_{\indexVelocity}<0} \delta_{\indexVelocity} = A - \delta_{\alpha}\leq 0$ by induction assumption. 
		Let us check the one of the term featuring the adjugate:
		\begin{equation*}
			\transpose{\canonicalBasisVector{\alpha}} \adjugate\Bigl (\sum_{\indexFreeOne\neq 0}\fourierShift^{\indexFreeOne}\sum_{\substack{\delta_{\indexFreeTwo} = \indexFreeOne\\ \indexFreeTwo \neq \alpha}}\canonicalBasisVector{\indexFreeTwo}\transpose{\canonicalBasisVector{\indexFreeTwo}} + \tilde{\matricial{K}}\Bigr ) \canonicalBasisVector{\alpha} = \Bigl [ \sum_{\indexFreeOne\neq 0}\fourierShift^{\indexFreeOne}\sum_{\substack{\delta_{\indexFreeTwo} = \indexFreeOne\\ \indexFreeTwo \neq \alpha}}\canonicalBasisVector{\indexFreeTwo}\transpose{\canonicalBasisVector{\indexFreeTwo}} + \tilde{\matricial{K}}\Bigr ]_{\alpha\alpha},
		\end{equation*}
		where $[\,\cdot\,]_{\alpha\alpha}$ indicates the minor with respect to the $\alpha$-th row and column.
		This is a determinant of a matrix of size $\numberVelocities - 1$ where $\delta_1, \dots, \delta_{\alpha-1}, \delta_{\alpha+1}, \dots, \delta_{\numberVelocities} \in \relatives$ are used. Hence, using $P(\sum_{\indexVelocity\neq \alpha, \delta_{\indexVelocity}<0} \delta_{\indexVelocity})$, which is assumed to be true, we deduce that the smallest degree of $\transpose{\canonicalBasisVector{\alpha}} \adjugate (\sum_{\indexFreeOne\neq 0}\fourierShift^{\indexFreeOne}\sum_{\substack{\delta_{\indexFreeTwo} = \indexFreeOne\\ \indexFreeTwo \neq \alpha}}\canonicalBasisVector{\indexFreeTwo}\transpose{\canonicalBasisVector{\indexFreeTwo}} + \tilde{\matricial{K}} ) \canonicalBasisVector{\alpha}$ is larger or equal to $\sum_{\indexVelocity\neq \alpha, \delta_{\indexVelocity}<0} \delta_{\indexVelocity}$, so that finally, using \eqref{eq:tmpMatrixDeterminant}, $\determinant (\sum_{\indexFreeOne\neq 0}\fourierShift^{\indexFreeOne}\sum_{\delta_{\indexFreeTwo} = \indexFreeOne}\canonicalBasisVector{\indexFreeTwo}\transpose{\canonicalBasisVector{\indexFreeTwo}} + \tilde{\matricial{K}} )$ has smallest degree larger or equal to $A = \delta_{\alpha} + \sum_{\indexVelocity\neq \alpha, \delta_{\indexVelocity}<0} \delta_{\indexVelocity}$.

        This shows that $P(A+1), \dots, P(0)$ imply $P(A)$,
	\end{itemize}
    and concludes the proof.

\section{Proof of \Cref{prop:noDiracBoundary}}\label{app:prop:noDiracBoundary}

	We first observe that $\matrixPolynomialBulk{\timeShiftOperator}(0) = -\schemeMatrixBulkByPower_{-1} = -\momentMatrix\canonicalBasisVector{\positiveVelocityIndex}\transpose{\canonicalBasisVector{\positiveVelocityIndex}} \momentMatrix^{-1}\collisionMatrix$ is a rank-one matrix, thus by the rank-nullity theorem, we obtain that $\dimension{\kernel(\matrixPolynomialBulk{\timeShiftOperator}(0))} = \numberVelocities - 1$.
	Equation \eqref{eq:equalityCharEquations} gives that the algebraic multiplicity of $\fourierShift\equiv 0$ equals $\numberVelocities - 1$, as it entails $\determinant(\matrixPolynomialBulk{\timeShiftOperator}(\fourierShift)) = \fourierShift^{\numberVelocities-1}(\coefficientCharEquationInFourier_{-1}(\timeShiftOperator) + \coefficientCharEquationInFourier_0(\timeShiftOperator)\fourierShift + \dots + \coefficientCharEquationInFourier_{\stencilRightCharacteristic}(\timeShiftOperator)\fourierShift^{\stencilRightCharacteristic+1})$, where the parentheses-enclosed polynomial in $\fourierShift$ does not vanish at $\fourierShift = 0$, since $\coefficientCharEquationInFourier_{-1}(\timeShiftOperator)\not\equiv 0$.
	We thus have a full set of $\numberVelocities-1$ eigenvectors associated with $\fourierShift\equiv 0$, which are independent of $\timeShiftOperator$ because $\matrixPolynomialBulk{\timeShiftOperator}(0)$ does not depend on it.

	Space-stable solutions in $\stableSubspace(\timeShiftOperator)$ thus read as in \eqref{eq:generalStableSolutionAllSchemes}, see \cite[Theorem 8.3]{matrixpoly09}, which plugged into \eqref{eq:resolventBoundary} entails \eqref{eq:lopatiskiiSystem}.
	By the assumptions made here, we have $\laplaceTransformed{\vectorial{\boundarySourceTermMoments}}_0(\timeShiftOperator) = \momentMatrix\canonicalBasisVector{\positiveVelocityIndex}\laplaceTransformed{\boundarySourceTerm}_{\positiveVelocityIndex, -1}(\timeShiftOperator)$, where we observe that $\momentMatrix\canonicalBasisVector{\positiveVelocityIndex}$ is nothing but the $\positiveVelocityIndex$-th column of the moment matrix $\momentMatrix$.
	Multiplying $\kreissLopatinskiiMatrix(\timeShiftOperator) \transpose{(\coefficientStableSolution(\timeShiftOperator), \coefficientZero^{1}(\timeShiftOperator), \dots, \coefficientZero^{\numberVelocities-1}(\timeShiftOperator))} = \momentMatrix\canonicalBasisVector{\positiveVelocityIndex}\laplaceTransformed{\boundarySourceTerm}_{\positiveVelocityIndex, -1}(\timeShiftOperator)$ by $\adjugate(\kreissLopatinskiiMatrix(\timeShiftOperator))$, we obtain 
	\begin{equation*}
		\kreissLopatinskiiDet(\timeShiftOperator)
		\transpose{(\coefficientStableSolution(\timeShiftOperator), \coefficientZero^{1}(\timeShiftOperator), \dots, \coefficientZero^{\numberVelocities-1}(\timeShiftOperator))} = \adjugate(\kreissLopatinskiiMatrix(\timeShiftOperator)) \momentMatrix\canonicalBasisVector{\positiveVelocityIndex} \laplaceTransformed{\boundarySourceTerm}_{\positiveVelocityIndex, -1}(\timeShiftOperator).
	\end{equation*}
	Now let $i\in\integerInterval{1}{\numberVelocities}$.
	The matrix-determinant lemma gives 
	\begin{equation*}
		\transpose{\canonicalBasisVector{i}}\adjugate(\kreissLopatinskiiMatrix(\timeShiftOperator)) \momentMatrix\canonicalBasisVector{\positiveVelocityIndex}  = \determinant(\kreissLopatinskiiMatrix(\timeShiftOperator) + \momentMatrix\canonicalBasisVector{\positiveVelocityIndex} \transpose{\canonicalBasisVector{i}}) - \kreissLopatinskiiDet(\timeShiftOperator)\\
		=\determinant(\kreissLopatinskiiMatrix(\timeShiftOperator)\xleftarrow[]{i} \momentMatrix\canonicalBasisVector{\positiveVelocityIndex}),
	\end{equation*}
	where the second equality comes from the multilinearity of the determinant, and the notation $\matricial{A}\xleftarrow[]{i} \vectorial{v}$ represents the matrix $\matricial{A}$ where the $i$-th column has been replaced by the column vector $\vectorial{v}$.
	Interestingly, this last equality is analogous to the Cramer's rule.

	Let now $\indexFreeOne\in\integerInterval{1}{\numberVelocities-1}$. We obtain $\kreissLopatinskiiDet(\timeShiftOperator)
	\coefficientZero^{\indexFreeOne}(\timeShiftOperator) = \determinant(\kreissLopatinskiiMatrix(\timeShiftOperator)\xleftarrow[]{\indexFreeOne + 1} \momentMatrix\canonicalBasisVector{\positiveVelocityIndex}) \laplaceTransformed{\boundarySourceTerm}_{\positiveVelocityIndex, -1}(\timeShiftOperator)$. We want to show that $\determinant(\kreissLopatinskiiMatrix(\timeShiftOperator)\xleftarrow[]{\indexFreeOne + 1} \momentMatrix\canonicalBasisVector{\positiveVelocityIndex}) \equiv 0$.
	This is achieved by showing that the first column of $\kreissLopatinskiiMatrix(\timeShiftOperator)$ and $\kreissLopatinskiiMatrix(\timeShiftOperator)\xleftarrow[]{\indexFreeOne + 1} \momentMatrix\canonicalBasisVector{\positiveVelocityIndex}$, namely $(\timeShiftOperator\identityMatrix{\numberVelocities} - \sum_{\indexFreeTwo = 0}^{\stencilBoundaryCondition}\schemeMatrixBoundaryByPower{0}{\indexFreeTwo}\stableRoot(\timeShiftOperator)^{\indexFreeTwo})\vectorial{\eigenvectorLetter}_{\textnormal{s}}(\timeShiftOperator)$, is collinear to $\momentMatrix\canonicalBasisVector{\positiveVelocityIndex}$, the $(\indexFreeOne + 1)$-th column of $\kreissLopatinskiiMatrix(\timeShiftOperator)\xleftarrow[]{\indexFreeOne + 1} \momentMatrix\canonicalBasisVector{\positiveVelocityIndex}$. 
	Using the fact that $\vectorial{\eigenvectorLetter}_{\textnormal{s}}(\timeShiftOperator)\in\kernel (\matrixPolynomialBulk{\timeShiftOperator}(\stableRoot(\timeShiftOperator)))$, thus also $\vectorial{\eigenvectorLetter}_{\textnormal{s}}(\timeShiftOperator)\in\kernel (\timeShiftOperator\identityMatrix{\numberVelocities} - \schemeMatrixBulkFourier(\stableRoot(\timeShiftOperator)))$, we can subtract a term equal to the zero vector and obtain 
	\begin{equation*}
		\Bigl (\timeShiftOperator\identityMatrix{\numberVelocities} - \sum_{\indexFreeTwo = 0}^{\stencilBoundaryCondition}\schemeMatrixBoundaryByPower{0}{\indexFreeTwo}\stableRoot(\timeShiftOperator)^{\indexFreeTwo} \Bigr )\vectorial{\eigenvectorLetter}_{\textnormal{s}}(\timeShiftOperator) = \Bigl (  \sum_{\indexFreeTwo = -1}^{\max(\stencilRight, \stencilBoundaryCondition)} (\schemeMatrixBulkByPower_{\indexFreeTwo} - \schemeMatrixBoundaryByPower{0}{\indexFreeTwo}) \stableRoot(\timeShiftOperator)^{\indexFreeTwo} \Bigr ) \vectorial{\eigenvectorLetter}_{\textnormal{s}}(\timeShiftOperator),
	\end{equation*}
	where the fact that $\schemeMatrixBoundaryByPower{0}{-1}=\zeroMatrix{\numberVelocities\times \numberVelocities}$ and $\schemeMatrixBulkByPower_{\indexFreeTwo}=\zeroMatrix{\numberVelocities\times\numberVelocities}$ for $\indexFreeTwo>\stencilRight$ is understood.
	As we deal with kinetic boundary conditions, which only replace lacking information for the $\positiveVelocityIndex$-th distribution function, and discrete velocities are all distinct, from \eqref{eq:matrixByPowerBulk}--\eqref{eq:matrixByPowerBoundary}, discrepancies between bulk and boundary schemes impact only the $\positiveVelocityIndex$-th distribution function, so mathematically
	\begin{equation*}
		\forall\indexFreeTwo \in\integerInterval{-1}{\max(\stencilRight, \stencilBoundaryCondition)}\qquad  \schemeMatrixBulkByPower_{\indexFreeTwo} - \schemeMatrixBoundaryByPower{0}{\indexFreeTwo} = \momentMatrix\canonicalBasisVector{\positiveVelocityIndex}\transpose{\vectorial{v}_{\indexFreeTwo}}, 
	\end{equation*}
	where $\vectorial{v}_{\indexFreeTwo}\in\reals^{\numberVelocities}$ are suitable vectors which depend on the boundary condition, but whose precise expression is pointless here.
	This yields
	\begin{equation*}
		\Bigl (\timeShiftOperator\identityMatrix{\numberVelocities} - \sum_{\indexFreeTwo = 0}^{\stencilBoundaryCondition}\schemeMatrixBoundaryByPower{0}{\indexFreeTwo}\stableRoot(\timeShiftOperator)^{\indexFreeTwo} \Bigr )\vectorial{\eigenvectorLetter}_{\textnormal{s}}(\timeShiftOperator) =  \momentMatrix\canonicalBasisVector{\positiveVelocityIndex}\underbrace{\sum_{\indexFreeTwo = -1}^{\max(\stencilRight, \stencilBoundaryCondition)} \stableRoot(\timeShiftOperator)^{\indexFreeTwo} \transpose{\vectorial{v}_{\indexFreeTwo}} \vectorial{\eigenvectorLetter}_{\textnormal{s}}(\timeShiftOperator)}_{\text{scalar}},
	\end{equation*}
	and shows that $(\timeShiftOperator\identityMatrix{\numberVelocities} - \sum_{\indexFreeTwo = 0}^{\stencilBoundaryCondition}\schemeMatrixBoundaryByPower{0}{\indexFreeTwo}\stableRoot(\timeShiftOperator)^{\indexFreeTwo})\vectorial{\eigenvectorLetter}_{\textnormal{s}}(\timeShiftOperator)$ is collinear to $\momentMatrix\canonicalBasisVector{\positiveVelocityIndex}$. Therefore, $\determinant(\kreissLopatinskiiMatrix(\timeShiftOperator)\xleftarrow[]{\indexFreeOne + 1} \momentMatrix\canonicalBasisVector{\positiveVelocityIndex}) \equiv 0$.

	Concerning $\coefficientStableSolution(\timeShiftOperator)$ we have 
	\begin{equation*}
		\kreissLopatinskiiDet(\timeShiftOperator)\coefficientStableSolution(\timeShiftOperator) = 
		\determinant
		\begin{bmatrix}
			\momentMatrix\canonicalBasisVector{\positiveVelocityIndex}\, | \,(\timeShiftOperator\identityMatrix{\numberVelocities} - \schemeMatrixBoundaryByPower{0}{0})  \vectorial{\eigenvectorLetter}_{0}^1 \, | \, \cdots \, | \, (\timeShiftOperator\identityMatrix{\numberVelocities} - \schemeMatrixBoundaryByPower{0}{0})  \vectorial{\eigenvectorLetter}_{0}^{\numberVelocities - 1} 
		\end{bmatrix}
		\laplaceTransformed{\boundarySourceTerm}_{\positiveVelocityIndex, -1}(\timeShiftOperator).
	\end{equation*}
	Setting $\scalarFactorFromAdjugate(\timeShiftOperator)\definitionEquality \determinant
	\begin{bmatrix}
		\momentMatrix\canonicalBasisVector{\positiveVelocityIndex}\, | \,(\timeShiftOperator\identityMatrix{\numberVelocities} - \schemeMatrixBoundaryByPower{0}{0})  \vectorial{\eigenvectorLetter}_{0}^1 \, | \, \cdots \, | \, (\timeShiftOperator\identityMatrix{\numberVelocities} - \schemeMatrixBoundaryByPower{0}{0})  \vectorial{\eigenvectorLetter}_{0}^{\numberVelocities - 1} 
	\end{bmatrix}$, it is not self-evident that this quantity does not depend on boundary conditions, for $\schemeMatrixBoundaryByPower{0}{0}$ appears in it.
	Under the current assumptions, one has 
	\begin{equation*}
		\schemeMatrixBoundaryByPower{0}{0} = \momentMatrix \Bigl( \sum_{\dimensionlessDiscreteVelocityLetter_{\indexVelocity}= 0 }\canonicalBasisVector{\indexVelocity}\transpose{\canonicalBasisVector{\indexVelocity}} + \canonicalBasisVector{\positiveVelocityIndex} \sum_{\indexFreeOne = 1}^{\numberVelocities} \boundaryCoefficient_{\positiveVelocityIndex, \indexFreeOne, 1, 0} \transpose{\canonicalBasisVector{\indexFreeOne}}\Bigr ) \momentMatrix^{-1}\collisionMatrix, 
	\end{equation*}
	where the first summation sign involves at most one term.
	Hence, for $\indexFreeTwo\in\integerInterval{1}{\numberVelocities - 1}$
	\begin{equation*}
		\schemeMatrixBoundaryByPower{0}{0} \vectorial{\eigenvectorLetter}_{0}^{\indexFreeTwo} = \underbrace{\momentMatrix \sum_{\dimensionlessDiscreteVelocityLetter_{\indexVelocity}= 0 }\canonicalBasisVector{\indexVelocity}\transpose{\canonicalBasisVector{\indexVelocity}} \momentMatrix^{-1}\collisionMatrix \vectorial{\eigenvectorLetter}_{0}^{\indexFreeTwo}}_{\text{indep. of the boundary cond.}} + \momentMatrix \canonicalBasisVector{\positiveVelocityIndex} \underbrace{\sum_{\substack{\indexFreeOne = 1\\\indexFreeOne\neq\positiveVelocityIndex}}^{\numberVelocities} \boundaryCoefficient_{\positiveVelocityIndex, \indexFreeOne, 1, 0} \transpose{\canonicalBasisVector{\indexFreeOne}}\momentMatrix^{-1}\collisionMatrix \vectorial{\eigenvectorLetter}_{0}^{\indexFreeTwo}}_{\text{scalar}}, 
	\end{equation*}
	using the specific expression of $\matrixPolynomialBulk{\timeShiftOperator}(0)$ discussed at the beginning of the proof.
	As the second addendum in the previous expression is collinear to $\momentMatrix \canonicalBasisVector{\positiveVelocityIndex}$, repeated use of the multilinearity of the determinant gives 
	\begin{equation*}
		\scalarFactorFromAdjugate(\timeShiftOperator) = 
		\determinant
		\begin{bmatrix}
			\momentMatrix\canonicalBasisVector{\positiveVelocityIndex}\, | \,(\timeShiftOperator\identityMatrix{\numberVelocities} - \momentMatrix \sum_{\dimensionlessDiscreteVelocityLetter_{\indexVelocity}= 0 }\canonicalBasisVector{\indexVelocity}\transpose{\canonicalBasisVector{\indexVelocity}} \momentMatrix^{-1}\collisionMatrix )  \vectorial{\eigenvectorLetter}_{0}^1 \, | \, \cdots \, | \, (\timeShiftOperator\identityMatrix{\numberVelocities} - \momentMatrix \sum_{\dimensionlessDiscreteVelocityLetter_{\indexVelocity}= 0 }\canonicalBasisVector{\indexVelocity}\transpose{\canonicalBasisVector{\indexVelocity}} \momentMatrix^{-1}\collisionMatrix )  \vectorial{\eigenvectorLetter}_{0}^{\numberVelocities - 1} 
		\end{bmatrix},
	\end{equation*}
    thus a simplified expression for $\scalarFactorFromAdjugate(\timeShiftOperator)$, which by the way shows that this term deserves its name, as it depends only on the bulk scheme.

	Notice that $\momentMatrix \sum_{\dimensionlessDiscreteVelocityLetter_{\indexVelocity}= 0 }\canonicalBasisVector{\indexVelocity}\transpose{\canonicalBasisVector{\indexVelocity}} \momentMatrix^{-1}\collisionMatrix = \schemeMatrixBulkByPower_0 $.
	The matrix-determinant lemma entails that 
	\begin{align*}
		\determinant(\timeShiftOperator\identityMatrix{\numberVelocities}-\schemeMatrixBulkFourier(\fourierShift)) &= \overbrace{\determinant\Bigl (\timeShiftOperator\identityMatrix{\numberVelocities}-\sum_{\indexFreeOne=0}^{\stencilRight}\schemeMatrixBulkByPower_{\indexFreeOne}\fourierShift^{\indexFreeOne} \Bigr )}^{\text{polynomial in }\fourierShift}  - \transpose{\canonicalBasisVector{\positiveVelocityIndex}} \momentMatrix^{-1}\collisionMatrix\overbrace{\adjugate \Bigl (\timeShiftOperator\identityMatrix{\numberVelocities}-\sum_{\indexFreeOne=0}^{\stencilRight}\schemeMatrixBulkByPower_{\indexFreeOne}\fourierShift^{\indexFreeOne} \Bigr )}^{\text{polynomial in }\fourierShift}\momentMatrix \canonicalBasisVector{\positiveVelocityIndex}\fourierShift^{-1}\\
		&= - \transpose{\canonicalBasisVector{\positiveVelocityIndex}} \momentMatrix^{-1}\collisionMatrix\adjugate  (\timeShiftOperator\identityMatrix{\numberVelocities}-\schemeMatrixBulkByPower_0 )\momentMatrix \canonicalBasisVector{\positiveVelocityIndex}\fourierShift^{-1} + (\text{non-negative powers in }\fourierShift),
	\end{align*}
	which shows that $\coefficientCharEquationInFourier_{-1}(\timeShiftOperator) = - \transpose{\canonicalBasisVector{\positiveVelocityIndex}} \momentMatrix^{-1}\collisionMatrix\adjugate  (\timeShiftOperator\identityMatrix{\numberVelocities}-\schemeMatrixBulkByPower_0 )\momentMatrix \canonicalBasisVector{\positiveVelocityIndex}$.
	As by definition of the eigenvectors of $\fourierShift\equiv 0$, we have $\momentMatrix\canonicalBasisVector{\positiveVelocityIndex}\transpose{\canonicalBasisVector{\positiveVelocityIndex}}\momentMatrix^{-1}\collisionMatrix\vectorial{\eigenvectorLetter}_0^{\indexFreeOne} = \zeroMatrix{\numberVelocities}$ for $\indexFreeOne\in\integerInterval{1}{\numberVelocities-1}$, and $\momentMatrix\canonicalBasisVector{\positiveVelocityIndex}\neq \zeroMatrix{\numberVelocities}$ as $\momentMatrix\in\linearGroup{\numberVelocities}{\reals}$, we deduce that $\transpose{\canonicalBasisVector{\positiveVelocityIndex}}\momentMatrix^{-1}\collisionMatrix\vectorial{\eigenvectorLetter}_0^{\indexFreeOne}= 0$ (\idEst{} $\transpose{\collisionMatrix}\momentMatrix^{-\mathsf{T}}\canonicalBasisVector{\positiveVelocityIndex}$ and $\vectorial{\eigenvectorLetter}_0^{\indexFreeOne}$ are orthogonal).
	Introduce the linear form $\phi(\vectorial{x}) \definitionEquality  - \transpose{\canonicalBasisVector{\positiveVelocityIndex}} \momentMatrix^{-1}\collisionMatrix\adjugate  (\timeShiftOperator\identityMatrix{\numberVelocities}-\schemeMatrixBulkByPower_0 )\vectorial{x}$ over $\reals^{\numberVelocities}$.
	Let $\indexFreeOne\in\integerInterval{1}{\numberVelocities-1}$:
	\begin{equation*}
		\phi((\timeShiftOperator\identityMatrix{\numberVelocities}-\schemeMatrixBulkByPower_0)\vectorial{\eigenvectorLetter}_0^{\indexFreeOne}) = - \determinant(\timeShiftOperator\identityMatrix{\numberVelocities}-\schemeMatrixBulkByPower_0) \transpose{\canonicalBasisVector{\positiveVelocityIndex}} \momentMatrix^{-1}\collisionMatrix\vectorial{\eigenvectorLetter}_0^{\indexFreeOne} = 0.
	\end{equation*}
	This linear form $\phi$ hence vanishes on $\spanSpace{(\timeShiftOperator\identityMatrix{\numberVelocities}-\schemeMatrixBulkByPower_0)\vectorial{\eigenvectorLetter}_0^{\indexFreeOne}}_{\indexFreeOne = 1}^{\numberVelocities - 1}$, and can thus be written, up to a multiplicative factor $\tilde{c} = \tilde{c}(\vectorial{\eigenvectorLetter}_0^1, \dots, \vectorial{\eigenvectorLetter}_0^{\numberVelocities-1}) \neq 0$ depending---for instance---on the normalization of the eigenvectors
	\begin{equation*}
		\phi(\vectorial{x}) = \tilde{c} \times \determinant[\vectorial{x} \, | \, (\timeShiftOperator\identityMatrix{\numberVelocities}-\schemeMatrixBulkByPower_0)\vectorial{\eigenvectorLetter}_0^{1} \, | \, \cdots \, | \, (\timeShiftOperator\identityMatrix{\numberVelocities}-\schemeMatrixBulkByPower_0)\vectorial{\eigenvectorLetter}_0^{\numberVelocities - 1}].
	\end{equation*}
	We obtain $\coefficientCharEquationInFourier_{-1}(\timeShiftOperator) = \phi(\momentMatrix\canonicalBasisVector{\positiveVelocityIndex}) = \tilde{c}\times  \scalarFactorFromAdjugate(\timeShiftOperator)$, so the claim by setting $c = 1/\tilde{c}$.

	We can further develop the Kreiss-Lopatinskii determinant, noticing that 
	\begin{align*}
		\timeShiftOperator\identityMatrix{\numberVelocities} - \sum_{\indexFreeOne = 0}^{\stencilBoundaryCondition}\schemeMatrixBoundaryByPower{0}{\indexFreeOne}\stableRoot(\timeShiftOperator)^{\indexFreeOne} &= 
		\timeShiftOperator\identityMatrix{\numberVelocities} - \schemeMatrixBulkFourier(\stableRoot(\timeShiftOperator)) + \schemeMatrixBulkByPower_{-1}\stableRoot(\timeShiftOperator)^{-1} - \momentMatrix\canonicalBasisVector{\positiveVelocityIndex}\sum_{\indexFreeOne = 0}^{\stencilBoundaryCondition}\sum_{\indexFreeTwo = 1}^{\numberVelocities}\transpose{\canonicalBasisVector{\indexFreeTwo}}\boundaryCoefficient_{\positiveVelocityIndex, \indexFreeTwo, 1, \indexFreeOne}\stableRoot(\timeShiftOperator)^{\indexFreeOne} \momentMatrix^{-1}\collisionMatrix\\
		&=\timeShiftOperator\identityMatrix{\numberVelocities} - \schemeMatrixBulkFourier(\stableRoot(\timeShiftOperator)) + \momentMatrix\canonicalBasisVector{\positiveVelocityIndex} \Bigl (  \transpose{\canonicalBasisVector{\positiveVelocityIndex}}\stableRoot(\timeShiftOperator)^{-1} -\sum_{\indexFreeOne = 0}^{\stencilBoundaryCondition}\sum_{\indexFreeTwo = 1}^{\numberVelocities}\transpose{\canonicalBasisVector{\indexFreeTwo}}\boundaryCoefficient_{\positiveVelocityIndex, \indexFreeTwo, 1, \indexFreeOne}\stableRoot(\timeShiftOperator)^{\indexFreeOne} \Bigr ) \momentMatrix^{-1}\collisionMatrix
	\end{align*}
	and using the multilinearity of the determinant and the definition of stable eigenvector $\vectorial{\eigenvectorLetter}_{\textnormal{s}}(\timeShiftOperator)$:
	\begin{align*}
		\kreissLopatinskiiDet(\timeShiftOperator) &= 
		\determinant 
		\begin{bmatrix}
			(\timeShiftOperator\identityMatrix{\numberVelocities} - \sum_{\indexFreeOne = 0}^{\stencilBoundaryCondition}\schemeMatrixBoundaryByPower{0}{\indexFreeOne}\stableRoot(\timeShiftOperator)^{\indexFreeOne})\vectorial{\eigenvectorLetter}_{\textnormal{s}}(\timeShiftOperator) \, | \,(\timeShiftOperator\identityMatrix{\numberVelocities} - \schemeMatrixBoundaryByPower{0}{0})  \vectorial{\eigenvectorLetter}_{0}^1 \, | \, \cdots \, | \, (\timeShiftOperator\identityMatrix{\numberVelocities} - \schemeMatrixBoundaryByPower{0}{0})  \vectorial{\eigenvectorLetter}_{0}^{\numberVelocities - 1}
		\end{bmatrix}\\
		&=\Bigl (  \transpose{\canonicalBasisVector{\positiveVelocityIndex}}\stableRoot(\timeShiftOperator)^{-1} -\sum_{\indexFreeOne = 0}^{\stencilBoundaryCondition}\sum_{\indexFreeTwo = 1}^{\numberVelocities}\transpose{\canonicalBasisVector{\indexFreeTwo}}\boundaryCoefficient_{\positiveVelocityIndex, \indexFreeTwo, 1, \indexFreeOne}\stableRoot(\timeShiftOperator)^{\indexFreeOne} \Bigr ) \momentMatrix^{-1}\collisionMatrix\vectorial{\eigenvectorLetter}_{\textnormal{s}}(\timeShiftOperator)\scalarFactorFromAdjugate(\timeShiftOperator).
	\end{align*}

\section{Analysis of additional boundary conditions for the \lbmScheme{1}{2} scheme}\label{app:moreBCD1Q2}

Let us start considering $\relaxationParameterLetter_2 = \tfrac{3}{2}$ and $\courantNumber = -\tfrac{1}{2}$.
We obtain:
\begin{equation*}
	\text{For }\eqref{eq:extrapolatedEquilibrium},\, \orderExtrapolation = 1 \qquad (\timeShiftOperator, \fourierShift) =  (1, 1), \quad (\timeShiftOperator, \fourierShift) = \Bigl (\underbrace{\frac{3 + \sqrt{73}}{32}}_{|\cdot|\approx 0.361}, \underbrace{\frac{9 - \sqrt{73}}{4}}_{|\cdot|\approx 0.113}\Bigr ), \quad 
	(\timeShiftOperator, \fourierShift) = \Bigl (\underbrace{\frac{3 - \sqrt{73}}{32}}_{|\cdot|\approx 0.173}, \underbrace{\frac{9 + \sqrt{73}}{4}}_{|\cdot|\approx 0.407}\Bigr ).
\end{equation*}
The only potential eigenvalue is the first one.
However, $\fourierShift = \unstableRoot$, so it is not an eigenvalue.
\begin{multline*}
	\text{For }\eqref{eq:extrapolatedEquilibrium},\,\orderExtrapolation = 3 \qquad (\timeShiftOperator, \fourierShift) =  (1, 1), \quad (\timeShiftOperator, \fourierShift) \approx \Bigl (\underbrace{0.702\pm0.193 i}_{|\cdot|\approx 0.728}, \underbrace{-0.303\pm0.272i}_{|\cdot|\approx 0.113}\Bigr ), \\
	(\timeShiftOperator, \fourierShift) \approx \Bigl (\underbrace{-0.405\pm0.297 i}_{|\cdot|\approx 0.502}, \underbrace{0.701\pm1.298 i}_{|\cdot|\approx 1.475}\Bigr ), \quad (\timeShiftOperator, \fourierShift) \approx (1.301, 1.593),  \quad (\timeShiftOperator, \fourierShift) \approx (-0.270, 2.610)
\end{multline*}
The penultimate couple, even if $\timeShiftOperator\in\neighborhoodInfinity$, does not yield a \strong{Godunov-Ryabenkii} instability since its corresponding $\fourierShift = \unstableRoot$: not an eigenvalue.
\begin{equation*}
	\text{For }\eqref{eq:kinrod}\qquad (\timeShiftOperator, \fourierShift) =  (1, 1), \quad (\timeShiftOperator, \fourierShift) \approx (\underbrace{-0.586\pm0.567i}_{|\cdot|\approx 0.816}, \underbrace{-0.256\pm1.456i}_{|\cdot|\approx 1.479}), \quad  (\timeShiftOperator, \fourierShift) \approx (0.141,  -5.487).
\end{equation*}
Again, we see that we have strong stability.
\begin{align*}
	\text{For }\eqref{eq:godunovRyabenkii}\qquad (\timeShiftOperator, \fourierShift) =  (1/2, -1), \quad (\timeShiftOperator, \fourierShift) &= \Bigl ( \overbrace{\frac{23+3\sqrt{105}}{26}}^{|\cdot|\approx2.067}, \overbrace{\frac{29-3\sqrt{105}}{26}}^{|\cdot|\approx0.067}\Bigr ),\\
	(\timeShiftOperator, \fourierShift) &= \Bigl ( \underbrace{\frac{23-3\sqrt{105}}{26}}_{|\cdot|\approx0.297}, \underbrace{\frac{29+3\sqrt{105}}{26}}_{|\cdot|\approx2.298}\Bigr ).
\end{align*}
The second couple is an eigenvalue and yields a \strong{Godunov-Ryabenkii} instability.

Let us now consider $\relaxationParameterLetter_2 = 2$ and $\courantNumber = -\tfrac{1}{2}$.
\begin{equation*}
	\text{For }\eqref{eq:extrapolatedEquilibrium}, \,\orderExtrapolation = 1 \qquad (\timeShiftOperator, \fourierShift) =  (1, 1), \quad (\timeShiftOperator, \fourierShift) = \Bigl (\underbrace{\frac{1 + \sqrt{17}}{8}}_{|\cdot|\approx 0.640}, \underbrace{\frac{5 - \sqrt{17}}{2}}_{|\cdot|\approx 0.438}\Bigr ), \quad (\timeShiftOperator, \fourierShift) = \Bigl (\underbrace{\frac{1 - \sqrt{17}}{8}}_{|\cdot|\approx 0.390}, \underbrace{\frac{5 + \sqrt{17}}{2}}_{|\cdot|\approx 4.562}\Bigr ).
\end{equation*}
This boundary condition is strongly stable.
\begin{multline*}
	\text{For }\eqref{eq:extrapolatedEquilibrium}, \,\orderExtrapolation = 3 \qquad (\timeShiftOperator, \fourierShift) =  (1, 1), \quad (\timeShiftOperator, \fourierShift) \approx \Bigl (\underbrace{1.019\pm0.465 i}_{|\cdot|\approx 1.120}, \underbrace{-0.438\pm0.568i}_{|\cdot|\approx 0.717}\Bigr ), \\
	(\timeShiftOperator, \fourierShift) \approx \Bigl (\underbrace{-0.769\pm0.429 i}_{|\cdot|\approx 0.881}, \underbrace{0.736\pm1.407 i}_{|\cdot|\approx 1.587}\Bigr ), \quad (\timeShiftOperator, \fourierShift) \approx (1.330, 1.734),  \quad (\timeShiftOperator, \fourierShift) \approx (-0.579, 2.670)
\end{multline*}
The second couple is an eigenvalue yielding a catastrophic \strong{Godunov-Ryabenkii} instability.
\begin{equation*}
	\text{For }\eqref{eq:kinrod}\qquad (\timeShiftOperator, \fourierShift) =  (1, 1), \quad (\timeShiftOperator, \fourierShift) \approx (\underbrace{-0.820\pm0.546 i}_{|\cdot|\approx 0.985}, \underbrace{0.083\pm1.587 i}_{|\cdot|\approx 1.590}), \quad  (\timeShiftOperator, \fourierShift) \approx (0.515,  -3.166).
\end{equation*}
We have strong stability.
\begin{equation*}
	\text{For }\eqref{eq:godunovRyabenkii}\qquad (\timeShiftOperator, \fourierShift) =  (1, -1), \quad (\timeShiftOperator, \fourierShift) = \Bigl ( \underbrace{\frac{5+\sqrt{73}}{6}}_{|\cdot|\approx2.257}, \underbrace{\frac{7-\sqrt{73}}{6}}_{|\cdot|\approx0.257}\Bigr ), \quad (\timeShiftOperator, \fourierShift) = \Bigl ( \underbrace{\frac{5-\sqrt{73}}{6}}_{|\cdot|\approx0.591}, \underbrace{\frac{7+\sqrt{73}}{6}}_{|\cdot|\approx2.591}\Bigr ).
\end{equation*}
The first couple is an eigenvalue as well as the second one, which gives a \strong{Godunov-Ryabenkii} instability.

We now go to $\relaxationParameterLetter_2 = \tfrac{3}{2}$ and $\courantNumber = \tfrac{1}{2}$.
\begin{align*}
	\text{For }\eqref{eq:extrapolatedEquilibrium}, \,\orderExtrapolation = 1 \qquad (\timeShiftOperator, \fourierShift) =  (1, 1), \quad (\timeShiftOperator, \fourierShift) &= \Bigl (\overbrace{\frac{-7 + \sqrt{241}}{32}}^{|\cdot|\approx 0.266}, \overbrace{\frac{11 - \sqrt{241}}{12}}^{|\cdot|\approx 0.377}\Bigr ), \\
	(\timeShiftOperator, \fourierShift) &= \Bigl (\underbrace{\frac{-7 - \sqrt{241}}{32}}_{|\cdot|\approx 0.704}, \underbrace{\frac{11 + \sqrt{241}}{12}}_{|\cdot|\approx 2.210}\Bigr ).
\end{align*}
The first couple is an eigenvalue.
\begin{multline*}
	\text{For }\eqref{eq:extrapolatedEquilibrium}, \,\orderExtrapolation = 3 \qquad (\timeShiftOperator, \fourierShift) =  (1, 1), \quad (\timeShiftOperator, \fourierShift) \approx \Bigl (\underbrace{-0.575\pm0.199 i}_{|\cdot|\approx 0.609}, \underbrace{0.906\mp0.991 i}_{|\cdot|\approx 1.343}\Bigr ),\\
	(\timeShiftOperator, \fourierShift)\approx ( 1.395, 0.565), \quad (\timeShiftOperator, \fourierShift)\approx ( 0.459, -0.850), \quad (\timeShiftOperator, \fourierShift)\approx ( 0.860, 1.383), \quad (\timeShiftOperator, \fourierShift)\approx (-0.689, 2.091).
\end{multline*}
The first couple is an eigenvalue, as the third one, which triggers a \strong{Godunov-Ryabenkii} instability.
\begin{equation*}
	\text{For }\eqref{eq:kinrod}\qquad (\timeShiftOperator, \fourierShift) =  (1, 1), \quad (\timeShiftOperator, \fourierShift) \approx (\underbrace{-0.599\pm0.253i}_{|\cdot|\approx 0.650}, \underbrace{0.628\mp1.299i}_{|\cdot|\approx 1.442}), \quad  (\timeShiftOperator, \fourierShift) \approx (0.666,  -1.922).
\end{equation*}
The first couple is an eigenvalue. 
From \Cref{fig:D1Q2-s-3_2-C-1_2-revision}, we see that the boundary condition is not strongly stable.
\begin{equation*}
	\text{For }\eqref{eq:godunovRyabenkii}\qquad (\timeShiftOperator, \fourierShift) =  (\tfrac{1}{2}, -1), \quad (\timeShiftOperator, \fourierShift) = \Bigl ( \underbrace{\frac{5+\sqrt{249}}{14}}_{|\cdot|\approx1.484}, \underbrace{\frac{23-\sqrt{249}}{14}}_{|\cdot|\approx0.516}\Bigr ), \quad (\timeShiftOperator, \fourierShift) = \Bigl ( \underbrace{\frac{5-\sqrt{249}}{14}}_{|\cdot|\approx0.770}, \underbrace{\frac{23+\sqrt{249}}{14}}_{|\cdot|\approx2.770}\Bigr ).
\end{equation*}
The second couple triggers a \strong{Godunov-Ryabenkii} instability.

We now go to $\relaxationParameterLetter_2 = 2$ and $\courantNumber = \tfrac{1}{2}$.
\begin{align*}
	\text{For }\eqref{eq:extrapolatedEquilibrium}, \,\orderExtrapolation = 1 \qquad (\timeShiftOperator, \fourierShift) =  (1, 1), \quad (\timeShiftOperator, \fourierShift) &= \Bigl (\overbrace{\frac{-5 + \sqrt{73}}{8}}^{|\cdot|\approx 0.443}, \overbrace{\frac{7 - \sqrt{73}}{6}}^{|\cdot|\approx 0.257}\Bigr ), \\
	(\timeShiftOperator, \fourierShift) &= \Bigl (\underbrace{\frac{-5 - \sqrt{73}}{8}}_{|\cdot|\approx 1.693}, \underbrace{\frac{7 + \sqrt{73}}{6}}_{|\cdot|\approx 2.591}\Bigr ).
\end{align*}
Once again, the first couple is an eigenvalue: the boundary condition is not strongly stable.
\begin{multline*}
	\text{For }\eqref{eq:extrapolatedEquilibrium}, \, \orderExtrapolation = 3 \qquad (\timeShiftOperator, \fourierShift) =  (1, 1), \quad (\timeShiftOperator, \fourierShift) \approx \Bigl (\underbrace{-1.017\pm0.455 i}_{|\cdot|\approx 1.114}, \underbrace{0.849\mp1.070 i}_{|\cdot|\approx 1.366}\Bigr ),\\
	(\timeShiftOperator, \fourierShift)\approx ( 1.603, 0.420), \quad (\timeShiftOperator, \fourierShift)\approx ( 0.889, -0.792), \quad (\timeShiftOperator, \fourierShift)\approx ( 0.828, 1.450), \quad (\timeShiftOperator, \fourierShift)\approx (-1.537, 2.223).
\end{multline*}
A \strong{Godunov-Ryabenkii} instability is triggered by the third listed mode.
\begin{equation*}
	\text{For }\eqref{eq:kinrod}\qquad (\timeShiftOperator, \fourierShift) =  (1, 1), \quad (\timeShiftOperator, \fourierShift) \approx (\underbrace{-0.966\pm0.464 i}_{|\cdot|\approx 1.072}, \underbrace{0.672\mp1.066i}_{|\cdot|\approx 1.261}), \quad  (\timeShiftOperator, \fourierShift) \approx (1.306,  -1.678).
\end{equation*}
A mild instability is triggered by the first mode.
\begin{equation*}
	\text{For }\eqref{eq:godunovRyabenkii}\qquad (\timeShiftOperator, \fourierShift) =  (1, -1), \quad (\timeShiftOperator, \fourierShift) = \Bigl ( \underbrace{\frac{-1+\sqrt{17}}{2}}_{|\cdot|\approx1.562}, \underbrace{\frac{5-\sqrt{17}}{2}}_{|\cdot|\approx0.438}\Bigr ), \quad (\timeShiftOperator, \fourierShift) = \Bigl ( \underbrace{\frac{-1-\sqrt{17}}{2}}_{|\cdot|\approx2.562}, \underbrace{\frac{5+\sqrt{17}}{2}}_{|\cdot|\approx 4.562}\Bigr ).
\end{equation*}
The second mode yields a catastrophic \strong{Godunov-Ryabenkii} instability.

\end{document}